\documentclass[12pt]{amsart}
\usepackage[T1]{fontenc}
\usepackage[utf8]{inputenc}
\usepackage[english]{babel}
\usepackage{amssymb, amsmath,mathtools,mathabx}
\usepackage{mathrsfs}%stix}
\usepackage{esint}
\usepackage{amscd}
\usepackage{verbatim}
\usepackage{stmaryrd}
\usepackage[dvipsnames]{xcolor}
\usepackage{pgf,tikz}
\usetikzlibrary{arrows}

\usepackage[margin=1in]{geometry}

\usepackage[shortlabels]{enumitem}

\usepackage[colorlinks,linkcolor={cyan},citecolor={magenta},urlcolor={purple},]{hyperref}

\usepackage[colorinlistoftodos,prependcaption,textsize=tiny]{todonotes}

\usepackage{forest}
\usepackage{tikz}
\tikzset{>=latex}

\usepackage{tabularx}
\newcolumntype{L}{>{\arraybackslash}X}
\usepackage{multirow}

\theoremstyle{plain}
\newtheorem{theorem}{Theorem}[section]
\theoremstyle{remark}
\newtheorem{remark}[theorem]{Remark}

\theoremstyle{plain}
\newtheorem{corollary}[theorem]{Corollary}
\newtheorem{lemma}[theorem]{Lemma}
\newtheorem{proposition}[theorem]{Proposition}
\newtheorem{definition}[theorem]{Definition}

\numberwithin{equation}{section}

\def\N{{\mathbb N}}
\def\Z{{\mathbb Z}}

\def\R{{\mathbb R}}
\def\C{{\mathbb C}}

\newcommand{\E}{{\mathbb E}}
\renewcommand{\P}{{\mathbb P}}
\newcommand{\F}{{\mathscr F}}
\newcommand{\g}{\gamma}

\newcommand{\om}{\omega}
\renewcommand{\O}{\Omega}

\renewcommand{\a}{\kappa}

\newcommand{\loc}{{\rm loc}}

\newcommand{\Tor}{\mathbb{T}}
\newcommand{\T}{\mathbb{T}}

\newcommand{\A}{\mathcal{A}}
\newcommand{\im}{\mathrm{i}}
\newcommand{\lambdalin}{\lambda_{{\rm lin}}}
\newcommand{\InD}{\mathcal{X}_{{\rm in}}}

\usepackage{stmaryrd}

\newcommand{\vcd}{\wh{v}_{\mathrm{cut}}}
\newcommand{\vcdi}{\wh{v}_{\mathrm{cut},i}}
\newcommand{\vd}{\wh{v}}
\newcommand{\wdd}{\wh{w}}
\newcommand{\taud}{\wh{\tau}}
\newcommand{\vdi}{\wh{v}_{i}}
\newcommand{\vcut}{v_{\mathrm{cut}}}
\newcommand{\wcut}{w_{\mathrm{cut}}}
\newcommand{\vcuti}{v_{\mathrm{cut},i}}
\newcommand{\vmod}{w}
\newcommand{\vmodi}{w_i}
\newcommand{\vmodhat}{\wh{w}}
\newcommand{\vmodhati}{\wh{w}_i}

\newcommand{\Ent}{\mathcal{E}}
\newcommand{\sol}{v}
\newcommand{\D}{\mathcal{D}}
\newcommand{\dd}{\mathrm{d}}

\newcommand{\ellip}{\nu}
\newcommand{\embed}{\hookrightarrow}

\newcommand{\s}{\delta}
\newcommand{\W}{\mathcal{W}}
\newcommand{\q}{\mathcal{Q}}
\newcommand{\y}{\boldsymbol{y}}
\newcommand{\veq}{v_{\infty}}
\newcommand{\veqi}{v_{\infty,i}}

\newcommand{\const}{D}

\newcommand{\Cyl}{\mathcal{B}}
\newcommand{\one}{\mathbf{1}}
\newcommand{\wt}{\widetilde}
\newcommand{\wh}{\widehat}

\newcommand{\mreg}{{\mathrm{MR}}_{p,\a}^{\s,q}}

\newcommand{\Borel}{\mathscr{B}}

\newcommand{\coleq}{\mathrel{\mathop:}=}
\newcommand{\eqcol}{=\mathrel{\mathop:}}

\allowdisplaybreaks

\begin{document}

\author{Antonio Agresti}
\address{Department of Mathematics Guido Castelnuovo\\ Sapienza University of Rome\\
P.le Aldo Moro 5\\ 00185 Rome\\ Italy}
\email{antonio.agresti@uniroma1.it}

\author{Michael Kniely}
\address{Department of Mathematics and Scientific Computing \\
University of Graz, Heinrichstra\ss e 36, 8010 Graz, Austria} 
\email{michael.kniely@uni-graz.at}

\author{Bao Quoc Tang}
\address{Department of Mathematics and Scientific Computing \\
University of Graz, Heinrichstra\ss e 36, 8010 Graz, Austria} 
\email{quoc.tang@uni-graz.at}

\thanks{The first author is a member of GNAMPA (INdAM) and acknowledges support from INdAM through the GNAMPA 2026 project ``Fluidodinamica stocastica: irregolarità, trasporto e fenomeni di regolarizzazione''}

\date\today

\title[Global classical solutions by transport noise]{Global classical solutions by transport noise for reaction--diffusion systems with entropy dissipation}

\keywords{Regularization by noise, global well-posedness, enhanced dissipation, stochastic reaction--diffusion equations, complex balanced reaction networks, transport noise, entropy dissipation.}

\subjclass[2020]{Primary: 60H50, Secondary: 60H15, 35K57, 35B65}

\begin{abstract}
The existence of global classical solutions for reaction--diffusion systems arising from chemical reaction networks remains a major open problem in the deterministic setting, especially for reactions with high polynomial growth.
We prove that a suitably chosen, physically motivated transport noise yields unique strong solutions for complex balanced chemical reaction networks that are global in time with arbitrarily high probability. These solutions possess paths in $C^{\theta}_t C^{\infty}_x$ for all $\theta<1/2$ and are, in particular, classical in space. Furthermore, we show that a suitable transport noise can enhance the dissipation of spatial fluctuations at an arbitrarily prescribed exponential rate. Our proofs rely on a combination of scaling-limit arguments, maximal $L^p(L^q)$-regularity, and entropy--entropy dissipation estimates.
\end{abstract}

\maketitle

\tableofcontents

\section{Introduction}
\label{s:intro}
In this work, we study the global existence of classical solutions to the following stochastic reaction--diffusion equations (RDEs) on the $d$-dimensional torus $\T^d=\R^d/\Z^d$ with $d\geq 2$:
\begin{equation}
\label{eq:reaction_diffusion_intro}
\left\{
\begin{aligned}
\dd v_i - \nu_i \Delta v_i \,\dd t
& = f_i(v)\,\dd t +\sqrt{c_d \nu }\sum_{k\in \Z_0^d}\sum_{\alpha=1}^{d-1} \theta_k (\sigma_{k,\alpha}\cdot\nabla) v_i\circ \dd W^{k,\alpha}_t
& \text{ on }&\Tor^d,\\ 
v_i(0)&=v_{0,i} &\text{ on }&\Tor^d.
\end{aligned}\right.
\end{equation}
Here $(\O,\A,(\F_t)_{t\geq 0},\P)$ is a filtered probability space, $v=(v_i)_{i=1}^{\ell}:[0,\infty)\times\O\times\T^d\to\R^\ell$ denotes the vector of concentrations, $c_d=d/(d-1)$,
$\Z_0^d=\Z^d\setminus\{0\}$, and $\nu>0$ is the noise intensity. The family $(\sigma_{k,\alpha})_{k,\alpha}$ consists
of smooth divergence-free vector fields, while $(W^{k,\alpha})_{k,\alpha}$
is a family of complex-valued Brownian motions, which are specified in Subsection~\ref{ss:noise_description}.
Finally, $f_i$ is a given nonlinearity, which models the reaction, see Subsection \ref{subsec:CRN}. 
One key feature of these nonlinearities is that they may have polynomial growth of arbitrarily high order, i.e., there exists $h\ge 1$ such that
\begin{equation*}
   |f_i(v)| \lesssim 1 + |v|^{h} \ \text{ for all }  v\in [0,\infty)^\ell \text{ and }i\in \{1,\ldots, \ell\}.
\end{equation*}
As discussed below, polynomial growth of order $h>2$ in the reaction terms represents a major obstacle to global well-posedness, both for deterministic and stochastic RDEs.

\smallskip

Reaction--diffusion systems constitute a classical framework for describing
the interplay between local reactions and spatial transport in a wide range of
applications, including chemical kinetics, population dynamics, ecology,
pattern formation, and mathematical biology; see, for instance,
\cite{erdi1989,mendez2010,murray2003,R84_global}
and the references therein. From the viewpoint of chemical reaction network
theory, the nonlinearities $f_i$ are determined by the underlying reaction
network and its kinetics; we refer, e.g., to \cite{feinberg2019} for a
systematic account of this connection.

\smallskip

In many applications, the reacting species are transported by an ambient
fluid, and the interplay between fluid mixing and chemical reactions can have
a substantial effect on the reaction dynamics; see, e.g.,
\cite{BaldygaBourne99,Fox03,Peters00}. When the underlying flow is turbulent,
the velocity field exhibits fluctuations over a broad range of spatial and
temporal scales, including scales much smaller than those at which the
transported concentrations are resolved; see, e.g., \cite{Warhaft00,Dimotakis05}. Therefore, a natural way to model the action of a turbulent fluid is to consider advection via a random field which is white-in-time and correlated-in-space.
Such stochastic transport models have a long history in the theory of passive scalars and turbulent diffusion; see,
for instance, \cite{F15_book,MK99_simplified}. 
In particular, white-in-time, spatially correlated velocity fields form the basis of the synthetic
turbulence models introduced by Kraichnan \cite{K68,K94}.
Rigorous derivations of transport-type noise from multiscale fluid models can be found, e.g., in \cite{DP22_two_scale,FP21}. In the case of RDEs, a heuristic derivation based on a similar scale-separation argument is given in \cite[Subsection 1.3]{AV23_RDE_local}.

\smallskip

In the absence of noise ($\nu = 0$), the existence of global-in-time smooth solutions to reaction--diffusion systems \eqref{eq:reaction_diffusion_intro} is a classical topic. While the local existence of smooth solutions is standard (see, e.g., \cite{R84_global}), the issue of possible finite-time blow-up poses significant challenges. 
Although reaction--diffusion systems arising from chemical reaction networks often possess mass conservation laws and natural entropy structures (see, e.g., \cite{DFT17,HJ72,P10_survey}), such properties are generally not sufficient to prevent singularity formation. Indeed, in \cite{PS97_blow_up}, the authors demonstrated that mass dissipation alone is insufficient to rule out finite-time blow-up of smooth solutions.

Since then, the quest to determine the optimal growth rate $h$ for global well-posedness has attracted considerable attention. The case of quadratic growth ($h=2$) and mass control has been successfully settled in \cite{CGV19_global,FMT20,souplet2018global}, where global classical solutions were established. Conversely, for any super-quadratic growth rate $h>2$, it was shown in \cite{pierre2023examples} that one can construct nonlinearities with mass control leading to finite-time blow-up in higher dimensions, provided the diffusion coefficients are sufficiently distinct. However, these counterexamples rely on artificially constructed nonlinearities rather than genuine chemical reactions. Consequently, the global existence of classical solutions for realistic networks with high growth rates remains a major open problem.

To bypass these difficulties, recent efforts have focused on exploiting specific structural properties, such as intermediate sum conditions \cite{bouton2025global,morgan2020boundedness} or entropic structures. Regarding the latter, the breakthrough result \cite{F15_global_renormalized} established the existence of global renormalized solutions to \eqref{eq:reaction_diffusion_intro} equipped with an entropic structure and $\nu = 0$, irrespective of the growth rate $h$. Furthermore, these solutions enjoy a weak--strong uniqueness property \cite{F17_weak_strong_uniqueness}. Nevertheless, the unconditional uniqueness of renormalized solutions remains entirely open. 
This scenario is therefore reminiscent of the 3D Navier--Stokes problem: global weak (or renormalized) solutions are available, while their uniqueness remains unresolved (see the breakthrough result \cite{ABC22_annals} in the presence of forcing, and the recent preprint \cite{hou2025nonuniqueness} for the unforced case), whereas classical solutions are unique but only known to exist locally in time  \cite{NS_problem,LePi}.

\smallskip

Here, we prove that a suitable transport noise improves the well-posedness theory of RDEs considerably. The possibility of improving a well-posedness theory by noise is, by now, a standard topic in stochastic analysis, and this is often referred to as \emph{regularization by noise}, see \cite{F15_book}. Additional discussion on the latter is given in Subsections \ref{ss:reg_noise_intro} and \ref{ss:novelties_intro} below.

Our main result can be informally stated as follows. A rigorous version is presented in Theorem \ref{t:global} after the necessary preliminaries.

\begin{theorem}[Global classical solutions -- Informal version of Theorem \ref{t:global}]\label{thm1:informal}
\label{t:global_intro}
Consider system \eqref{eq:reaction_diffusion_intro} where the nonlinearities $f_i$ arise from complex balanced chemical reaction networks without boundary equilibria (see Definition \ref{def:complex_boundary_equilibria}) and have arbitrary growth rates $h>0$. Fix $N\geq 1$, $\varepsilon\in (0,1)$, and a compact set $K\subseteq \q \R^\ell_{>0}$, where $\q$ is the matrix associated with the Wegscheider matrix determining the conservation laws of the underlying chemical reaction (see Subsection \ref{subsec:CRN}). Then, there are $\nu>0$ and $\theta\in \ell^2(\Z^d_0)$ such that $\#\{k\,:\, \theta_k\neq 0\}<\infty$ and for all initial data $v_0\in L^\infty(\T^d;\R^\ell)$ satisfying $\int_{\T^d} v_{0,i}\,\dd x \neq 0$ for all $i\in \{1,\dots,\ell\}$, $\q (\int_{\T^d} v_{0}\,\dd x )\in K$,
$$
v_0\geq 0 \text{ a.e. on $\T^d$ (component-wise)} \quad \text{ and } \quad \|v_0\|_{L^\infty(\T^d;\R^\ell)}\leq N,
$$
there exists a unique solution $v:[0,\tau)\times \O\times \T^d\to \R^\ell$ to \eqref{eq:reaction_diffusion_intro} that is \emph{classical in space}, i.e., 
\begin{align}
\label{eq:regularity_intro_2}
v&\in C^{1/2-,\infty}((0,\tau)\times \T^d;\R^\ell)\ \text{ a.s., }
\end{align}
and is \emph{global in time} with probability larger than $1-\varepsilon$, i.e., 
\begin{equation}
\label{eq:global_intro_1}
\P(\tau=\infty)>1-\varepsilon.
\end{equation}
\end{theorem}

As commented above, such a result is not known in general in the deterministic setting ($\nu=0$).
The absence of boundary equilibria can be viewed as a non-degeneracy condition on the underlying chemical reaction network, excluding equilibria at which one or more species have zero concentration. 
Concrete examples of chemical reactions for which global classical well-posedness is not known in the deterministic setting, but to which our results apply, are discussed in Subsection \ref{ss:examples_reaction_diffusion_intro}.

\smallskip

It is worth pointing out that transport noise does not influence conservation laws of the chemical reactions, namely, the following pathwise equalities hold, a.s.\ for all $t<\tau$,
$$
\q\,\overline{ v(t,\cdot)}=
\q\,\overline{ v_0} ,
$$ 
where $\overline{\cdot} =\int_{\T^d} \cdot \,\dd x$ denotes the spatial average, and
\begin{align*}
\|v_i(t)\|_{L^q(\T^d)}^q+
q(q-1)\nu_i\int_0^t \int_{\T^d} |v_i|^{q-2} |\nabla v_i|^2\,\dd x \,\dd s
 = \|v_{0,i}\|_{L^q(\T^d)}^q
 +q\int_0^t \int_{\T^d} |v_i|^{q-2}v_i f_i(v)\,\dd x\,\dd s,
\end{align*}
for any $q\geq 2$ and $i\in \{1,\dots,\ell\}$.
In particular, one obtains exactly the same energy and $L^q$-estimates as in the deterministic case ($\nu=0$).
The key to understanding the regularizing effect of transport noise lies in its enhanced dissipation effect.
Roughly speaking, enhanced dissipation refers to the acceleration, due to advection, of the decay of spatial fluctuations compared with the case of pure diffusion.
Enhanced dissipation has been extensively studied in the deterministic setting; see the seminal work \cite{CKRZ08_mixing}, the contributions \cite{BCZ17,FI19}, and the recent surveys \cite{CZC24,MFN26}.
In the stochastic setting, important contributions include \cite{BBPS21_PTRF,BBP22_AP}, where passive scalars advected by stochastic Navier--Stokes flows are considered. 

To illustrate enhanced dissipation in the setting of \eqref{eq:reaction_diffusion_intro}, let us first consider a passive scalar $\phi$ subject to the same noise:
\begin{equation}
\label{eq:linear_PDE_intro}
\dd \phi =\Delta \phi\,\dd t+\sqrt{c_d \nu }\sum_{k\in \Z_0^d}\sum_{\alpha=1}^{d-1} \theta_k (\sigma_{k,\alpha}\cdot\nabla) \phi\circ \dd W^{k,\alpha}_t, \qquad \phi(\cdot,0)=\phi_0.
\end{equation}
%Let $\overline{\phi}_0=\int_{\T^d} \phi_0\,\dd x$ be the mean of the initial data. 
Due to the divergence-free condition of the vector fields $\sigma_{k,\alpha}$, smooth solutions of \eqref{eq:linear_PDE_intro} satisfy, almost surely, for all $t>0$,
$$
\frac{1}{2}\|\phi(t)-\overline{\phi}_0\|_{L^2(\T^d)}^2+ \int_0^t \int_{\T^d} |\nabla \phi|^2 \,\dd x \,\dd s = 
\frac{1}{2}\|\phi_0-\overline{\phi}_0\|_{L^2(\T^d)}^2.
$$
The above and Poincar\'e's inequality yield 
\begin{equation}
\label{eq:basic_decay_L2}
\|\phi(t)-\overline{\phi}_0\|_{L^2(\T^d)}^2\leq e^{- 8 \pi^2 t } \|\phi_0-\overline{\phi}_0\|_{L^2(\T^d)}^2.
\end{equation}
In the absence of noise $\nu=0$, the above estimate is optimal. However, as shown in \cite{FGL24_quantitative,L21}, for each $\chi>0$, there exist $\nu\gg 1$ and $\theta\in \ell^2(\Z^d_0)$ such that, almost surely, for all $t>0$,
\begin{equation}
\label{eq:basic_decay_enhanced}
\|\phi(t)-\overline{\phi}_0\|_{L^2(\T^d)}^2\leq \const e^{- \chi t} \|\phi_0-\overline{\phi}_0\|_{L^2(\T^d)}^2 ,
\end{equation}
where $\const$ is a random constant with finite (sufficiently high) moments. The reader is referred to \cite{zelati2025stochastic,GY21} for additional results in this direction.
Comparing \eqref{eq:basic_decay_L2} and \eqref{eq:basic_decay_enhanced}, one sees that the transport noise enhances the convergence of $\phi(t)$ towards its spatial mean $\overline{\phi}_0$.

\smallskip

In the context of reaction-diffusion systems, the study of enhanced dissipation is much more limited. To the best of our knowledge, the only work in this direction is \cite{KS22_stirring}, where the authors considered a specific reaction--diffusion system modelling a single irreversible reaction and showed that, under suitable assumptions on the initial data and on the deterministic advecting vector fields, stirring enhances the convergence towards the equilibrium of the underlying reaction.

In general, however, the rate of convergence to equilibrium cannot be arbitrarily enhanced. Indeed, for spatially constant initial data, the dynamics are entirely determined by the reaction. Hence, no velocity field can accelerate such convergence uniformly over all initial data.
A natural notion of enhanced dissipation for RDEs, which is the one we pursue here, is accelerated convergence towards the instantaneous spatial average, namely the fast decay of
\begin{equation}
\label{eq:oscillation_energy}
\| v(t,\cdot)-\overline{v}(t)\|_{L^2(\T^d;\R^\ell)}
\qquad\text{ where }\qquad
\overline{v}(t) \coleq \int_{\T^d}v(t,x)\,\dd x.
\end{equation}
For passive scalars as in \eqref{eq:linear_PDE_intro}, hence in the absence of reactions, the two viewpoints essentially coincide: the spatial average is conserved, and the corresponding constant state is precisely the equilibrium towards which the solution converges. 
Physically, a fast decay of the quantity in \eqref{eq:oscillation_energy} means that the solution rapidly homogenizes towards its spatial mean.

\smallskip

Our second main result shows that transport noise can make these spatial fluctuations decay at an arbitrarily prescribed exponential rate for large noise intensity and suitably chosen noise coefficients.

\begin{theorem}[Enhanced dissipation -- Informal version of Theorem \ref{t:enhanced_dissipation_reaction}]\label{thm2:informal}
	Consider system \eqref{eq:reaction_diffusion_intro} where the nonlinearities $f_i$ arise from complex balanced chemical reaction networks without boundary equilibria (see Definition \ref{def:complex_boundary_equilibria}) and have arbitrary growth rates $h>0$. Fix $N\ge 1$, $\varepsilon\in (0,1)$,  a compact set $K\subseteq \q\R^\ell_{>0}$ (see Subsection \ref{subsec:CRN}), $\chi\in(0,\infty)$ and $b\in (1,\infty)$. Then, there exist $\nu>0$ and $\theta\in \ell^2(\Z^d_0)$ such that $\#\{k\,:\, \theta_k\neq 0\}<\infty$ and the unique solution to \eqref{eq:reaction_diffusion_intro} obtained in Theorem \ref{thm1:informal} satisfies
	\begin{equation*}
		\|v(t,\cdot)-\overline{v}(t)\|_{L^2(\T^d;\R^\ell)}
		\leq \const e^{-\chi t} \|v_0\|_{L^2(\T^d;\R^\ell)}\ \text{ a.s.\ for all }t<\mu,
	\end{equation*}
	where $\mu$ is a stopping time and $\const$ is a random variable with
	\begin{equation*}
		\mathbb{P}(\mu = \infty) > 1 - \varepsilon \quad \text{ and } \quad \mathbb{E}\const^b < \infty.
	\end{equation*}
\end{theorem}

In the following, we discuss related works on stochastic RDEs and on regularization by noise (Subsection \ref{ss:reg_noise_intro}) followed by a discussion of the main novelties of our work (Subsection \ref{ss:novelties_intro}). An outline of the proofs is given in Subsection \ref{ss:proof_strategy}, see also Figure \ref{fig:proof_architecture}.

\subsection{Previous works}
\label{ss:reg_noise_intro} 
There is an extensive literature on stochastic RDEs, and we only mention here some of the results that are closer to the present manuscript. In the presence of transport noise, local well-posedness and high-order regularity results were established in \cite{AV23_RDE_local}, while global well-posedness for several classes of dissipative and weakly dissipative systems was obtained in \cite{AV24_dissipative}.
In the case of non-trace-class noise, fundamental results were established in \cite{Cer}, and more recently optimal well-posedness and regularity were obtained in \cite{AGV26}. Global well-posedness under suitable triangular mass-control conditions was recently established in \cite{milesis2026global}.
Finally, in the case of nonlinearities with quadratic growth and H\"older-continuous multiplicative noise, but not of transport type, global well-posedness in space dimension two was obtained in \cite{leocata2024global}.

\smallskip

In the context of reaction--diffusion systems, the possibility that noise improves the well-posedness theory compared with the corresponding deterministic problem, a phenomenon usually referred to as \emph{regularization by noise}, has attracted considerable attention. Stabilization by multiplicative non-trace-class noise was established in \cite{Cer05}. Pathwise uniqueness in the presence of irregular drifts and non-degenerate noise was established in \cite{DFPR13} in an abstract Hilbert-space setting, and in \cite{CDF13} for stochastic reaction--diffusion equations with a H\"older drift. More recently, regularization by noise, not of transport type, has been obtained for RDEs and related SPDEs with distributional drifts; see, e.g., \cite{anzeletti2025uniqueness,BOS25,DHL26}.

\smallskip

In the presence of transport noise, a basic mechanism behind its possible regularizing effect was discovered in \cite{G20}. There, the author proved that, under a suitable scaling of the noise coefficients, solutions to stochastic transport equations converge to solutions of a deterministic parabolic equation with an additional effective diffusivity. This procedure is by now well-known and is typically referred to as the \emph{scaling limit}.
The additional diffusion appearing in the limit is key to understanding the enhanced dissipation effect of transport noise discussed above. Building on this approach, several results on delayed blow-up and global well-posedness by transport noise have been obtained for nonlinear PDEs; see, e.g., \cite{A24_global_small,FGL21,FL19,L23_regularization_NS} and the references therein.

The work closest to the present one is \cite{A22}, where the first author exploited the scaling-limit approach in an $L^p(L^q)$-setting, as required in the context of RDEs, to prove delayed blow-up for reaction--diffusion systems with mass control. A detailed comparison with that work is given below; see also Subsection \ref{ss:proof_strategy} for the proof strategy.

\subsection{Novelties}
\label{ss:novelties_intro}
To prove Theorems \ref{t:global_intro} and \ref{thm2:informal}, we move beyond the finite-time delayed-blow-up framework developed for RDEs in \cite{A22}.
There, an analogue of Theorem \ref{t:global_intro} was established on any prescribed finite interval $[0,T]$, but not on the entire time half-line. The main advance is to combine this finite-time regularization mechanism with the entropy structure of complex balanced reaction networks, thereby upgrading delayed blow-up to global well-posedness with high probability. This global-in-time control is also essential for the asymptotic enhanced dissipation estimate in Theorem \ref{thm2:informal}.

\begin{itemize}
\item \emph{Global classical solutions by transport noise.}
To rule out singularities with large probability on the entire time horizon, we further exploit the structure of complex balanced reaction networks through entropy methods. More precisely, Theorem \ref{t:global_intro} is obtained by combining $L^p(L^q)$-techniques with entropy dissipation and entropy--entropy dissipation estimates (see Section \ref{s:entropy_review}). In particular, we prove a version of the latter estimates that is uniform over compact sets of conserved mass vectors; see Theorem \ref{t:entropy_dissipation} and Lemma \ref{l:continuous_cinfty}.
A key ingredient is the close-to-equilibrium result of Theorem \ref{t:small_implies_global}, which allows us to pass from a finite to an infinite time horizon. This result is also of independent interest as it extends the deterministic close-to-equilibrium theory to stochastic RDEs with \emph{rough} transport noise, and removes restrictions on the growth of the nonlinearities imposed in \cite{tang2018close}.

\smallskip

\item \emph{Enhanced dissipation by transport noise.}
Exploiting global well-posedness, we obtain the quantitative enhanced dissipation result of Theorem \ref{thm2:informal}. The proof builds on quantitative estimates for transport noise developed in \cite{FGL24_quantitative,L21}, which we adapt to the setting of nonlinear reaction--diffusion systems. A key point is to identify the cancellation that arises when considering fluctuations around the instantaneous spatial average $v(t,\cdot)-\overline{v}(t)$. Indeed, the corresponding reaction term takes the centered form $f_i(v)-\overline{f_i(v)}$, whose zero-mean structure is crucial for controlling the nonlinear deterministic convolution. This also reflects the distinction discussed above between enhanced homogenization and enhanced convergence towards a reaction equilibrium.
As illustrated in \eqref{eq:linear_PDE_intro}--\eqref{eq:oscillation_energy}, this provides a natural nonlinear counterpart of enhanced dissipation for passive scalars: the transport noise rapidly homogenizes the concentrations towards their instantaneous spatial averages, while the latter continue to evolve according to the reaction dynamics.
\end{itemize}

\subsection{Concrete reaction networks with large stoichiometric coefficients}
\label{ss:examples_reaction_diffusion_intro}
We briefly discuss some concrete chemical reactions with large stoichiometric coefficients for which global existence of classical solutions to the deterministic reaction--diffusion system \eqref{b-2} is currently open, while Theorems \ref{t:global_intro} and \ref{thm2:informal} apply.

\smallskip

The first example is the synthesis of adenosine triphosphate (ATP) from adenosine diphosphate (ADP) together with inorganic phosphate ($\mathrm{P_i}$) and $n \geq 3$ protons ($\mathrm{H^+}$) \cite{Boyer97, YMH01}. This process is highly nontrivial and at the same time essential for the storage of energy in any form of life. A central reversible reaction is 
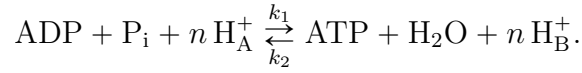
\begin{align*}
\mathrm{ADP + P_i} + n \, \mathrm{H_A^+} \underset{k_2}{\overset{k_1}{\rightleftarrows}}
 \mathrm{ATP + H_2 O} + n \, \mathrm{H_B^+}. 
\end{align*}
This is a single reversible reaction with disjoint species on the two sides and therefore belongs to the detailed-balance class considered here; in particular, the relevant positive stoichiometric compatibility classes contain no boundary equilibria; see \cite{fellner2017explicit}.
The reaction above is only one part of the full ATP synthesis mechanism.
The reverse, ATP hydrolysis, is used in nature to gain energy by transforming ATP back to ADP (with reaction rate constant $k_2 > 0$). The indices $\mathrm A$ and $\mathrm B$ of $\mathrm H^+$ refer to different origins of the involved protons; see \cite{Boyer97, YMH01} for further details. According to the mass-action law, the reaction--diffusion system \eqref{eq:reaction_diffusion_intro} has nonlinearities given by 
$$
f_i (v_1,\dots,v_6)
=\gamma_i R(v) \ \  \text{ for }\ i\in \{1,\dots,6\},
$$
where 
$
R(v)=k_1v_1v_2v_3^n-k_2v_4v_5v_6^n,
$ 
$\gamma=(-1,-1,-n,1,1,n)$ and $v_1,\dots,v_6$ are the concentrations of $\mathrm{ADP}$, $\mathrm{P_i}$, $\mathrm{H_A^+}$, $\mathrm{ATP}$, $\mathrm{H_2 O}$, and $\mathrm{H_B^+}$, respectively. 
The nonlinearities on the right-hand side have order $n+2 \geq 5$, and global existence of classical solutions for the deterministic counterpart of  \eqref{eq:reaction_diffusion_intro} (i.e., $\nu=0$) is not known in general.

\smallskip

As a second example, directly related to combustion processes in fluids, we consider the reversible reaction
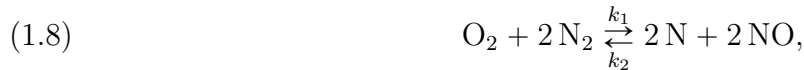
\begin{equation}
\label{eq:combustion1}
\mathrm{O_2 + 2\,N_2}
\underset{k_2}{\overset{k_1}{\rightleftarrows}}
\mathrm{2\,N + 2\,NO},
\end{equation}
which appears as a kinetic reaction in reduced models for methane combustion in spark-ignition engines; see, e.g., \cite[Table 3]{SYS00} (see also \cite[Table 4]{Akbarian2018}). 
Denoting by $v_1,v_2,v_3,v_4$ the concentrations of $\mathrm{O_2}$, $\mathrm{N_2}$, $\mathrm{N}$, and $\mathrm{NO}$, respectively, set
$
\wt{R}(v) \coleq k_1v_1v_2^2-k_2v_3^2v_4^2.
$
Then, the reaction nonlinearities are given by
$$
f_1(v)=-\wt{R}(v),\qquad
f_2(v)=-2\wt{R}(v),\qquad
f_3(v)=2\wt{R}(v),\qquad
f_4(v)=2\wt{R}(v).
$$
In particular, the nonlinearities have polynomial growth of order $h=4$. Note that \eqref{eq:combustion1} is a single and reversible reaction with disjoint species on the two sides and therefore falls within the detailed-balance class considered above; in particular, the relevant positive stoichiometric compatibility classes do not contain boundary equilibria; see also \cite{fellner2017explicit} for the general single-reaction setting.
Moreover, the corresponding system has a cubic intermediate-sum structure. In dimensions $d\geq2$, the available deterministic theory does not yield global classical solutions for arbitrary positive diffusion coefficients in this regime; see, e.g., \cite{STY23} and the references therein. To the best of our knowledge, global classical well-posedness for the corresponding deterministic reaction--diffusion system with arbitrary positive diffusion coefficients remains open. Hence, Theorems \ref{t:global_intro} and \ref{thm2:informal} provide global classical solutions with high probability in a regime currently not covered by the deterministic theory.

\subsection{Notation} Here, we collect the basic notation used in the manuscript.  
For two quantities $x$ and $y$, we write $x\lesssim y$, if there exists a constant $C>0$ independent of $x$ and $y$ such that $x\le Cy$. If such a $C$ depends on the parameters $p_1,\dots,p_n$ we either mention it explicitly or indicate this by writing $C_{p_1,\dots,p_n}$ and correspondingly $x\lesssim_{p_1,\dots,p_n}y$ whenever $x\le C_{p_1,\dots,p_n}y$. We write $x\eqsim_{p_1,\dots,p_n} y$ whenever $x\lesssim_{p_1,\dots,p_n} y$ and $y\lesssim_{p_1,\dots,p_n}x$. Moreover, we write $a\vee b=\max\{a,b\}$ and $a\wedge b=\min\{a,b\}$.

\subsubsection*{Probabilistic setting.}
Below, $(\O, \mathcal{A},(\F_t)_{t\geq 0}, \P)$ denotes a filtered probability space carrying a sequence of independent standard Brownian motions, which changes depending on the SPDE under consideration, and $\E[\cdot]=\int_{\O} \cdot \,\dd \P$ for the associated expected value. 
A process $\phi:[0,\infty)\times \O\to X$ is progressively measurable if $\phi|_{[0,t]\times \O}$ is $\Borel([0,t])\otimes \F_t$-measurable for all $t\geq 0$, where $\Borel$ is the Borel $\sigma$-algebra on $[0,t]$ and $X$ a Banach space. Moreover, a stopping time $\tau$ is a measurable map $\tau:\O\to [0,\infty]$ such that $\{\tau\leq t\}\in \F_t$ for all $t\geq 0$. Finally, a stochastic process $\phi:[0,\tau)\times \O\to X$ is progressively measurable if $\one_{[0,\tau)\times \O}\,\phi$ is progressively measurable where $
[0,\tau)\times \O \coleq \{(t,\om)\in[0,\infty)\times \O\,:\,\,0\leq t<\tau(\om)\}$
and $\one_{[0,\tau)\times \O}$ (or simply $\one_{[0,\tau)}$) stands for the extension by zero outside $[0,\tau)\times \O$. The definitions of the stochastic intervals $(0,\tau)\times \O$ and $[0,\tau]\times \O$ are similar.

\subsubsection*{Function spaces.}
Let $X$ be a Banach space. 
We write $L^p(S,\mu;X)$ for the Bochner space of strongly measurable, $p$-integrable $X$-valued functions for a measure space $(S,\mu)$ and $p\in (1,\infty)$; see, e.g., \cite[Section 1.2b]{Analysis1}. As usual, $\one_A$ denotes the indicator function of $A\subseteq S$. 
Fix $\a\in \R$. We denote by $w_{\a}$ the associated power weight:
$$
 w_\a(t) \coleq |t|^\a \quad \text{ for }t\in \R. 
$$ 
If $S=(s,t)$ for some $-\infty \leq s<t\leq \infty$, $p\in (1,\infty)$, and $\dd \mu= w_\a \,\dd r $, we simply write either $L^p((s,t),w_{\a} ;X )$ or  $L^p(s,t,w_{\a} ;X )$ instead of $L^p((s,t),w_{\a}\,\dd r ;X )$. In particular,
$$
\|f\|_{L^p(s,t,w_{\a};X)}= \Big(\int_{s}^t \|f(r)\|_X^p \,w_\a(r) \,\dd r \Big)^{1/p}.
$$
We denote by $W^{1,p}(s,t,w_\a;X)$ the set of all $f\in L^p(s,t,w_{\a};X)$ such that $f'\in L^p(s,t,w_{\a};X)$ endowed with the natural norm, see \cite[Section 2.5]{Analysis1} for distributional derivatives of $X$-valued maps. 
As usual, for $I\subseteq (-\infty,\infty)$, we say $f\in L^{p}_{\loc}(I;X)$ if $f\in L^{p}(J;X)$ for all compact sets $J\subseteq I$. A similar notation is employed if $L^{p}$ is replaced by either $W^{1,p}$. 

\smallskip

We write $\T^d$ for the $d$-dimensional torus $\R^d/ \Z^d$. Fix $q\in (1,\infty)$. The usual Sobolev space is denoted by $W^{k,q}(\T^d)$ for an integer smoothness index $k\in \N_{\geq 0}$. For the definition of the Besov spaces $B^{s}_{q,p}(\R^d)$ and Bessel-potential spaces $H^{s,q}(\R^d)$, the reader is referred to \cite[Section 2.1.2]{Runst_Sickel} and for the periodic spaces $B^{s}_{q,p}(\T^d)$ and $H^{s,q}(\T^d)$ to \cite[Section 3.5.4]{schmeisser1987topics}. Recall that $H^{-s,q}(\T^d)=(H^{s,q'}(\T^d))^*$, where $q'$ is the Holder conjugate exponent of $q$, and $H^{s,q}(\T^d)=W^{s,q}(\T^d)$ if $s$ is an integer. Moreover, we also set $W^{-k,q}(\T^d)=H^{-k,q}(\T^d)$ when $k\in \N_{\geq 0}$.

For all $\vartheta_0,\vartheta_1>0$ and $-\infty<s<t<\infty$, we let 
\begin{align*}
C^{\vartheta_0,\vartheta_1}((s,t)\times \T^d)
&\coleq C^{\vartheta_0}(s,t;C(\T^d))\cap C([s,t];C^{\vartheta_1}(\T^d)),\\ 
C^{\vartheta_0,\infty}((s,t)\times \T^d)
&\coleq \textstyle{\cap}_{\vartheta_1<\infty}\,
C^{\vartheta_0,\vartheta_1}((s,t)\times \T^d),
\end{align*}
and $C^{\vartheta_0,\vartheta_1}_{\loc}((s,t)\times \T^d) \coleq \cap_{\varepsilon>0}\, C^{\vartheta_0,\vartheta_1}((s+\varepsilon,t-\varepsilon)\times \T^d)$, and similarly if $\vartheta_1=\infty$. 

Finally, in the case of $\R^\ell$-valued maps, the above definitions extend trivially.

\section{Preliminaries}

\subsection{Reaction--diffusion systems modelling reaction networks}\label{subsec:CRN}
Let $2\le \ell\in \mathbb N$ and $S_1, S_2, \ldots, S_{\ell}$ be given different chemicals. We consider a mixture of these chemicals, which react with each other in the following $m$ reactions
\begin{equation}\label{reaction-full}
    y_{r,1}S_1 + y_{r,2}S_2 + \cdots + y_{r,\ell}S_\ell \xrightarrow{k_r} y_{r,1}'S_1 + y_{r,2}'S_2 + \cdots + y_{r,\ell}'S_\ell, \quad r = 1, \ldots, m, 
\end{equation}
where $k_r>0$ is the reaction rate constant, $y_{r,i}, y_{r,i}' \in \N_{\geq 0}$ are stoichiometric coefficients. By using the shorthand notations $\y_r \coleq (y_{r,1},\ldots, y_{r,\ell})$ and $\y_r' \coleq (y_{r,1}', \ldots, y_{r,\ell}')$, the reaction \eqref{reaction-full} can be written more compactly as
\begin{equation}\label{reaction-short}
    \y_r \xrightarrow{k_r} \y_r', \quad r = 1,\ldots, m.
\end{equation}
In our main results, we can also use stoichiometric coefficients in $\{0\}\cup [1,\infty)$, with the only difference being a limited spatial regularity for the solutions, see the comments below Proposition \ref{prop:LWP}.
Assume that the reactions take place on the torus $\T^d$ and denote by $v_i(t,x)$ the concentration of $S_i$ at $x\in \T^d$ and $t\ge 0$. The (deterministic) reaction--diffusion system modelling these chemical reactions reads as
\begin{equation}\label{b-2}
\left\{
    \begin{aligned}
        \partial_t v_i - d_i\Delta v_i &= f_i(v) \quad &\text{ on }&\T^d,\\
        v_i(0) &= v_{i,0} \quad  &\text{ on }&\T^d,
    \end{aligned}
    \right.
\end{equation}
where $-d_i\Delta v_i$ represents the molecular diffusion of $S_i$, according to Fick's second law, with the diffusion coefficient $d_i>0$, and $f_i(v)$ represents the reactions following the law of mass action, i.e.,
\begin{equation}\label{b-1}
    f_i(v) = \sum_{r=1}^mk_r(y_{r,i}' - y_{r,i})v^{\y_r} \quad \text{ with } \quad v^{\y_r} \coleq \prod_{j=1}^{\ell}v_j^{y_{r,j}}
\end{equation}
for $v\in [0,\infty)^\ell$. For convenience, throughout the manuscript, we extend the nonlinearities $f_i$ from $[0,\infty)^\ell$ to $\R^\ell$ by replacing $v$ by $(v_1\vee 0, \dots, v_\ell\vee 0)$ in the above formula whenever $v\not\in [0,\infty)^\ell$.
By setting $h \coleq \max_{r=1,\ldots, m}|\y_r|$ with $|\y_r| = y_{r,1} + \ldots + y_{r,\ell}$, we see that the nonlinearities have at most polynomial growth rates of order $h$, i.e.,
\begin{equation}
\label{b-11}
    |f_i(v)| \lesssim 1 + |v|^{h} \ \text{ for all } v \in \R^{\ell} \text{ and } i \in \{1, \ldots, \ell\}.
\end{equation}
To demonstrate an example of a chemical reaction network, we consider the general single reversible reaction
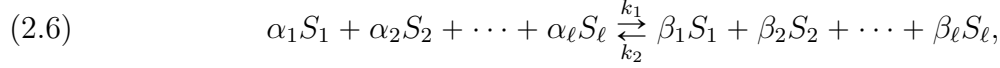
\begin{equation}\label{general-reversible}
    \alpha_1S_1 + \alpha_2S_2 + \cdots + \alpha_{\ell}S_{\ell} \underset{k_2}{\overset{k_1}{\rightleftarrows}} \beta_1 S_1 + \beta_2 S_2 + \cdots + \beta_{\ell}S_{\ell},
\end{equation}
where $\alpha_i, \beta_i\in \{0\}\cup [1,\infty)$ for all $i \in \{1, \ldots, \ell\}$. This reaction network consists of two reactions $\y_1 \xrightarrow{k_1} \y_1'$ and $\y_2 \xrightarrow{k_2} \y_2'$ where $\y_1 = \y_2' = (\alpha_1, \ldots, \alpha_\ell)$ and $\y_2 = \y_1' = (\beta_1, \ldots, \beta_{\ell})$. The corresponding reaction--diffusion system reads as
\begin{equation}\label{b3}
    \left\{
    \begin{aligned}
    &\partial_t v_i - d_i\Delta v_i = f_i(v) = (\beta_i - \alpha_i)\Big(k_1\prod_{j=1}^{\ell}v_j^{\alpha_j} - k_2\prod_{j=1}^{\ell}v_j^{\beta_j} \Big)&\text{ on }&\T^d,\\
    &v_i(0) = v_{i,0} &\text{ on }& \T^d.
    \end{aligned}
    \right.
\end{equation}
System \eqref{b3} includes the reactions in the Subsection \ref{ss:examples_reaction_diffusion_intro} as special cases.
Comments on the well-posedness of the above system can be found in the text preceding Theorem \ref{t:global_intro}.

\smallskip

Next, we discuss conservation laws of the general system \eqref{b-2}--\eqref{b-1} and its complex balanced equilibria. We define the Wegscheider matrix 
$
\W \coleq
(\y_{r}'-\y_{r})_{r\in \{1,\dots,m\}}\in \R^{m\times \ell}
$ 
and let $\ell' = \mathrm{codim} \, \W$. If $\ell' > 0$, then there exists a non-unique matrix 
$
\q \coleq (\boldsymbol q_j)_{j \in \{1,\ldots,\ell'\}} \in \R^{\ell' \times \ell}, 
$
where $\{\boldsymbol q_1, \ldots, \boldsymbol q_{\ell'}\} \subset \mathbb R^\ell$ is a basis of $\mathrm{ker} \, \W$. We now easily obtain the identity 
\begin{align}
\label{eq:ranWT=kerQ}
\mathrm{ran} \, \W^{\top} = \mathrm{ker} \, \q.
\end{align}
Let $f(v) \coleq (f_i(v))_{i=1}^{\ell}$ where $f_i$ is as in \eqref{b-1}. By construction, 
\begin{equation}
\label{eq:Q_anniled_f}
f(v)\in 
\mathrm{ran} \, \W^{\top} \quad \text{ for all }
v\in \R_{\geq 0}^{\ell}.
\end{equation}
It is clear that for all (sufficiently smooth) solutions of \eqref{b-2}--\eqref{b-1}, we have 
$$
\partial_t \q  \overline{v} (t)=0, \quad \text{ where }\quad \overline{v}(t) \coleq \int_{\T^d} v(t,x)\,\dd x.
$$
In particular, we have $\ell'$ linearly independent conservation laws represented by $\q \overline{v} (t)=\q \overline{v_{0}}$ for all $t\geq 0$.

\begin{definition}[Complex balanced and boundary equilibria]
\label{def:complex_boundary_equilibria}
A constant state $v_{\infty}=(v_{\infty,i})_{i=1}^{\ell}\in \R_{\geq 0}^{\ell}$ is called
\begin{itemize}
\item an \emph{equilibrium} point of the system \eqref{b-2}--\eqref{b-1} if
$
f(v_{\infty})=0
$; 
\item a \emph{boundary equilibrium} if $v_{\infty}$ is an equilibrium of \eqref{b-2}--\eqref{b-1} and 
$$
v_{\infty,i}=0 \ \ \text{ for some } \ i\in \{1,\dots,\ell\};
$$
\item a \emph{complex balanced equilibrium} if for each complex $\y\in \{\y_r, \y_r': r=1,\ldots, m\}$, there is a balance between the inflow and the outflow at $\y$ at the equilibrium $v_{\infty}$:
$$
\underbrace{\sum_{\{r\,:\, \y = \y_r\}} k_r \prod_{j=1}^{\ell} v_{\infty,j}^{y_{r,j}} }_{\textsc{total outflow from $\y$}}=
\underbrace{\sum_{\{r\,:\, \y = \y_r'\}} k_r \prod_{j=1}^{\ell} v_{\infty,j}^{y_{r,j}}}_{\textsc{total inflow into $\y$}},
$$
where $\{r\,:\, \y = \y_r\}$ and $\{r\,:\, \y = \y_r'\}$ denote the sets of reactions having the complex $\y$ as reactant and product, respectively.
\end{itemize}
\end{definition}

It is known (see, e.g., \cite{horn1972necessary,HJ72}) that if \eqref{reaction-full} has an equilibrium which is complex balanced, then all other possible equilibria are also complex balanced. Thus, we can define the main object of study of the current work.
\begin{definition}[Complex balanced chemical reaction networks]
We say that a chemical reaction network \eqref{reaction-full}, or the chemical reaction--diffusion system \eqref{b-2}--\eqref{b-1}, is complex balanced if there exists a positive complex balanced equilibrium $v_{\infty}\in \R^\ell_{>0}$. 
\end{definition}
Moreover, we have the following uniqueness result.
\begin{lemma}\cite[Proposition 5.1]{feinberg1995existence}\label{lem1}
    Assume that the reaction network \eqref{reaction-full} is complex balanced. Then, for each fixed initial mass vector $M \in \q \R^{\ell}_{> 0}$, there exists a unique positive equilibrium $v_\infty \in \R_{>0}^{\ell}$. 
\end{lemma}
Note that although the positive complex balanced equilibrium is unique, for a fixed initial mass vector, there might exist (possibly infinitely many) boundary complex balanced equilibria (or just boundary equilibria for short). For instance, consider \eqref{general-reversible} with $\ell = 2$, $\alpha_1 = \beta_2 = 1$ and $\alpha_2 = \beta_1 = 2$. The corresponding reaction--diffusion system is
\begin{equation*}
\left\{
    \begin{aligned}
        \partial_t v_1 - d_1\Delta v_1 &= k_1v_1 v_2^2 - k_2v_1^2v_2 \qquad &\text{ on }&\T^d,\\
        \partial_t v_2 - d_2\Delta v_2& = -k_1v_1 v_2^2 + k_2v_1^2v_2\qquad &\text{ on }&\T^d,\\
        v_i(0,x)& = v_{i,0}(x) \qquad&\text{ on }&\T^d.
    \end{aligned}
    \right.
\end{equation*}
The previous system has one conservation law $\int_{\T^d}(v_1(t) + v_2(t))\,\dd x = \int_{\T^d}(v_{1,0}+ v_{2,0})\,\dd x \eqcol M$. For each fixed initial mass $M>0$, there exists a unique positive equilibrium $v_\infty = (v_{\infty,1}, v_{\infty,2})$ solving $v_{\infty,1} + v_{\infty,2} = M$ and $k_1v_{\infty,2} = k_2v_{\infty,1}$. However, there are two boundary equilibria $(0,M)$ and $(M,0)$. In \cite{FT18_normalized_entropy}, it has been shown that in the case of complex balanced systems \emph{without} boundary equilibria, renormalized solutions converge to the positive equilibrium exponentially in $L^1$. This will be used to obtain our objectives in the current paper.

\smallskip

We conclude this subsection by remarking that there is an extensive literature on large time dynamics of complex balanced chemical reaction networks. They include both the setting of ODEs, see e.g.\ \cite{Anderson11,Craciun09,Pantea,Siegel04,Sontag}, or PDEs, see e.g.\ \cite{DFT17,FT18_normalized_entropy,glitzky1996free,glitzky1997energetic,groger1986existence,groger1983asymptotic,groger1992free}. Without boundary equilibria, the convergence towards equilibrium has been completely solved, see e.g. \cite{FT18_normalized_entropy}. The case with boundary equilibria gives rise to the so-called {\it Global Attractor Conjecture}, which is still remained open despite many partial answers, see e.g. \cite{Anderson11,FT18_normalized_entropy,Pantea}.
 
\subsection{Description of the noise}
\label{ss:noise_description}
In this subsection, we describe the quantities $(\theta_k,\sigma_{k,\alpha},W^{k,\alpha})$ appearing in \eqref{eq:reaction_diffusion}. Here, we follow \cite{FGL21,FL19}. 
Recall that $\Z_0^d=\Z^d\setminus\{0\}$. Throughout this paper, we consider $\theta=(\theta_k)_{k\in \Z^d_0}\in \ell^2(\Z^d_0)$, and we assume that $\theta$ is normalized and radially symmetric, i.e.,
\begin{equation}
\label{eq:theta_normalized_symmetric}
\|\theta\|_{\ell^2(\Z^d_0)}=1 \quad \text{ and } \quad \theta_{j}=\theta_k \  \text{ for all $j,k\in\Z_0^d$ \  such that }|j|=|k|.
\end{equation}
Moreover, for simplicity of exposition, we assume that $\#\{k\,:\, \theta_k\neq 0\}<\infty$. However, this can be weakened, see \cite[Remark 3.8]{A22}.

Next, we define the family of vector fields $(\sigma_{k,\alpha})_{k,\alpha}$. 
 Let $\Z_{+}^d$ and $\Z_-^d$ be a partition of $\Z_0^d$ such that $-\Z_+^d=\Z_-^d$. For any $k\in \Z_+^d$, select a complete orthonormal basis  $\{a_{k,\alpha}\}_{\alpha\in \{1,\dots,d-1\}}$ of the hyperplane $k^{\bot} \coleq \{k'\in \R^d\,:\, k\cdot k'=0\}$, 
and set $a_{k,\alpha} \coleq a_{-k,\alpha} $ for $k\in \Z^d_-$. Then, let
\begin{equation}
\label{eq:sigma_choice_vector_field}
\sigma_{k,\alpha} \coleq a_{k,\alpha} e^{2\pi \im k\cdot x}  \ \  \text{ for all } \ \ x\in \Tor^d,\ k\in \Z^d_0,\ \alpha\in \{1,\dots,d-1\}.
\end{equation}
Clearly, $\sigma_{k,\alpha}$ are smooth and divergence-free vector fields.
Finally, we introduce the family of complex Brownian motions $(W^{k,\alpha})_{k,\alpha}$. 
Let $(B^{k,\alpha})_{k,\alpha}$ be a family of independent standard Brownian motions on a filtered probability space $(\O,\A,(\F_t)_{t\geq 0},\P)$. Then, we set 
$$
W^{k,\alpha}
\coleq
\left\{
\begin{aligned}
&B^{k,\alpha}+\im B^{-k,\alpha},  &\quad & k\in \Z^d_+,  \\
&B^{-k,\alpha}-\im B^{k,\alpha},  & \quad& k\in \Z^d_-.
\end{aligned}
\right.
$$
As in \cite[Remark 1.1]{FGL21} or \cite[Section 2.3]{FL19}, from \eqref{eq:theta_normalized_symmetric} and the definition of the vector fields $\sigma_{k,\alpha}$, we can, at least formally, convert the Stratonovich integration into an It\^o integration. Indeed, if $v$ is a progressively measurable process with sufficiently regular paths, formally we have:
\begin{align}
\label{eq:Ito_stratonovich_change}
\sqrt{c_d \ellip}\sum_{k,\alpha}\theta_k (\sigma_{k,\alpha}\cdot \nabla) v\circ \dd W_t^{k,\alpha}=\ellip \Delta v\, \dd t+  \sqrt{c_d \ellip}\sum_{k,\alpha}\theta_k (\sigma_{k,\alpha}\cdot \nabla) v\,  \dd W_t^{k,\alpha}.
\end{align}
To prove the above equality, one uses that $\nabla \cdot\sigma_{k,\alpha}=0$ and the identity (see \cite[eq.\ (2.3)]{FL19} or \cite[eq.\ (3.2)]{FGL21})
\begin{equation}
\label{eq:ellipticity_noise}
\sum_{k,\alpha} \theta_k^2 \sigma_{k,\alpha}^{n} \overline{\sigma_{k,\alpha}^{m}}
=
\sum_{k,\alpha} \theta_k^2 a_{k,\alpha}^{n} a_{k,\alpha}^{m}
 =\frac{1}{c_d}\delta_{n,m}\ \   \text{ on }\Tor^d  \text{ for all $m, n \in \{1, \ldots, d\}$}.
\end{equation}

Note that the It\^o correction in \eqref{eq:Ito_stratonovich_change} is a dissipative operator
$\nu \Delta$, which is independent of $\theta$. The latter fact is both surprising and crucial for our purposes. However, it is important to observe that the term $\ellip\Delta v_i$ does not add any extra diffusion to the dynamics but it counterbalances It\^o's corrections in energy estimates, see the discussion below Theorem \ref{t:global_intro}. %To see this, it is enough to perform a standard energy estimate and recognize that the additional dissipation is precisely balanced by the It\^o correction arising from the It\^o noise.

Throughout the remainder of this paper, we will interpret the Stratonovich formulation of the noise on the LHS of \eqref{eq:Ito_stratonovich_change} as its RHS, i.e.\ as an It\^o noise plus a diffusion term.

\section{Statement of the main results}
In this section, we state our main results on the following stochastic RDEs for the unknown concentration $v=(v_i)_{i=1}^{\ell}:\R_+\times \O\times \Tor^d\to [0,\infty)^{\ell}$,
\begin{equation}
\label{eq:reaction_diffusion}
\left\{
\begin{aligned}
\dd v_i -\nu_i \Delta v_i \,\dd t &= f_i(v)\, \dd t+\sqrt{c_d \nu}\sum_{k,\alpha} \theta_k (\sigma_{k,\alpha} \cdot\nabla) v_i\circ \dd W^{k,\alpha}_{t}   & \text{ on }&\Tor^d,\\
v_{i}(0)&=v_{0,i}& \text{ on }&\Tor^d,
\end{aligned}
\right.
\end{equation}
where $i\in \{1,\dots,\ell\}$, $d\geq 2$, $c_d \coleq \frac{d}{d-1}$ and $f_i$ are the nonlinearities associated with the above-mentioned chemical reaction network \eqref{b-1}. Finally, $(W^{k,\alpha})_{k,\alpha}$ and $(\sigma_{k,\alpha})_{k,\alpha}$ are as described in Subsection \ref{ss:noise_description}.
Recall from Section \ref{s:intro} or \cite[Subsection 1.3]{AV23_RDE_local} that the transport noise in \eqref{eq:reaction_diffusion} serves as a simplified model for a turbulent flow in which the reactions take place.

\subsection{Solution concepts and local well-posedness}
Here, we introduce the solution concept for \eqref{eq:reaction_diffusion}, and recall a local well-posedness result from \cite{AV23_RDE_local}, which serves as a starting point for our investigation.
As usual in SPDEs, we interpret the stochastic RDEs \eqref{eq:reaction_diffusion} as a stochastic evolution equation on a suitable Banach space driven by a cylindrical noise. To this end, note that, for real-valued functions $v_i$, we formally have, 
\begin{equation*}
\sum_{k,\alpha} \theta_k (\sigma_{k,\alpha} \cdot\nabla) v_i\,\dd W^{k,\alpha}_{t}
= \sum_{k\in \Z^d_+,\alpha} 2\theta_k (\Re \sigma_{k,\alpha}\cdot \nabla) v_i\, \dd B^{k,\alpha}_t
+\sum_{k\in \Z^d_-,\alpha} 2\theta_k (\Im \sigma_{k,\alpha}\cdot \nabla) v_i\, \dd B^{k,\alpha}_t,
\end{equation*}
where $(B^{k,\alpha})_{k,\alpha}$ are a family of standard independent Brownian motions, and $\Z^d_{\pm}$ are as in Subsection \ref{ss:noise_description} (in particular, $\Z^d_0=\Z^d_+\cup \Z^d_-$). 
In particular, the stochastic RDEs \eqref{eq:reaction_diffusion} can be reformulated in terms of the (real) cylindrical Brownian motion $\Cyl_{\ell^2}$ on $L^2(\R_+;\ell^2)$, where $\ell^2=\ell^2(\Z^d\times \{1,\dots,d-1\})$, defined via the expression
$$
\Cyl_{\ell^2}(f)=\sum_{k,\alpha}\int_{\R_+} f_{k,\alpha}(s)\,\dd B^{k,\alpha}_s \ \  \text{ for }\ \ f=(f_{k,\alpha})_{k,\alpha}\in L^2(\R_+;\ell^2);
$$
see, e.g., \cite[Definition 2.11 and Example 2.12]{AV19_QSEE_1}. For notational convenience, we also let 
\begin{equation*}
\xi_{k,\alpha}=2 \Re \sigma_{k,\alpha} \ \text{ if }\ k\in\Z^d_+ \qquad \text{ and } \qquad \xi_{k,\alpha}=2\Im\sigma_{k,\alpha} \ \text{ if }\ k\in\Z^d_-. 
\end{equation*}
Although the above shows that the noise in \eqref{eq:reaction_diffusion} can be reformulated using only real-valued Brownian motions, we still employ the complex-valued formulation as it is the one used in the related literature (see, e.g., \cite{FGL21,FL19,L21}) and is more convenient for the subsequent analysis.

\smallskip

With the above notation, we can define solutions to the stochastic RDEs \eqref{eq:reaction_diffusion}.
Recall that we interpret the Stratonovich integration in \eqref{eq:reaction_diffusion} as an It\^o one plus the It\^o--Stratonovich correction, see \eqref{eq:Ito_stratonovich_change}.
To handle the possibly high growth of reactions in stochastic RDEs \eqref{eq:reaction_diffusion}, we employ the $L^p(L^q)$-setting for stochastic PDEs.

\begin{definition}[$(p,\a,\s,q)$-solutions]
\label{def:solution}
Let $p,q\in (2,\infty)$, $\a\in [0,\frac{p}{2}-1)$ and $\delta\in [1,2)$. Suppose that $v_0\in B^{2-\delta-2\frac{1+\a}{p}}_{q,p}(\T^d;\R^\ell)$, and let
$$
\tau:\O\to [0,\infty] \quad \text{ and }\quad 
v:[0,\tau)\times \O\to H^{2-\delta,q}(\T^d;\R^\ell)
$$
be a stopping time and a progressively measurable process, respectively.
\begin{enumerate}[{\rm(1)}]
\item\label{it:solution1} We say that $(v,\tau)$ is a \emph{local $(p,\a,\s,q)$-solution} to \eqref{eq:reaction_diffusion} if there exists a sequence of stopping times $(\tau_n)_{n}$ such that $\tau_n \uparrow \tau$ a.s.\ and such that the following holds for all $i\in \{1,\dots,\ell\}$:
\begin{itemize}
\item a.s.\ for all $n\geq 1$:%, it holds that
$$
v_i\in L^p(0,\tau_n,w_{\a};H^{2-\delta,q}(\T^d))\ \  \text{ and }\ \  
f_i (v)\in L^p(0,\tau_n,w_{\a};H^{-\delta,q}(\T^d));
$$
\item a.s.\ for all $n\geq 1$ and $t\in [0,\tau_n]$:
\begin{align*}
v_i(t) &=v_{0,i}+\int_0^t\big[ (\nu+\nu_i) \Delta v_i + f_i(v)\big]\,\dd s \\ &\quad +\sqrt{c_d\nu} \int_0^t \one_{[0,\tau_n]} \big(\theta_k(\xi_{k,\alpha}\cdot \nabla) v_i\big)_{k,\alpha}\,\dd \Cyl_{\ell^2}
\end{align*}
in the space $H^{-\delta, q}(\T^d)$.
\end{itemize}
\item We say that $(v,\tau)$ is a \emph{$(p,\a,\s,q)$-solution} to \eqref{eq:reaction_diffusion} if for any other local $(p,\a,\s,q)$-solution $(v',\tau')$, we have $\tau'\leq \tau$ a.s.\ and $v=v'$ a.e.\ on $[0,\tau')\times \O$.
\item We say that $(v,\tau)$ is a \emph{global} $(p,\a,\s,q)$-solution to \eqref{eq:reaction_diffusion} if $(v,\tau)$ is a $(p,\a,\s,q)$-solution and $\tau=\infty$ a.s. In this situation, we write $v$ instead of $(v,\tau)$.
\end{enumerate}
\end{definition}

It follows from the above that $(p,\a,\s,q)$-solutions are unique. Moreover, as commented in \cite[Definition 2.3]{AV23_RDE_local}, due to the assumed regularity for local $(p,\a,\delta,q)$-solutions, all the integrals appearing in the last item of \ref{it:solution1} are well-defined. 

The assumption on the initial data $v_0$ is optimal in the $L^p(L^q)$-approach to SPDE; see, e.g., \cite[Subsection 3.3]{AV25_survey}. In particular, the parameter $\delta$ is necessary to allow for initial data of zero smoothness. As $B_{q,r_1}^{\sigma}\embed B^{\sigma}_{q,r_2}$ for all $\sigma\in \R$ and $1\leq r_1<r_2\leq \infty$, the time weight $\a$ allows one to enlarge the space of initial data.

\smallskip

For future reference, we recall the following consequence of \cite[Theorem 2.13 and Proposition 3.1]{AV23_RDE_local} and the elementary embedding (see \cite[Theorem 3.3]{A22} for details)
\begin{equation}
\label{eq:embedding_Lq_B0qp}
L^q(\T^d)\subseteq B^0_{q,r}(\T^d)\ \  \text{ for all }\ r\geq q.
\end{equation}

\begin{proposition}[Local well-posedness and regularity]
\label{prop:LWP}
Let \eqref{reaction-full} be a complex balanced chemical reaction network, and let $h>1$ be the order of the reactions in \eqref{b-1}--\eqref{b-11}. Let $\theta$ be normalized radially symmetric, i.e., \eqref{eq:theta_normalized_symmetric} holds, and $\#\{k\,:\,\theta_k\neq 0\}<\infty$.
Assume that $p,q\in (2,\infty)$ and $\delta\in (1,2)$ satisfy
\begin{equation}
\label{eq:pq_condition_local_WP}
q\geq \frac{d(h-1)}{2}\vee \frac{d}{d-\delta}\qquad \text{ and }\qquad p\geq \frac{2}{2-\delta}\vee q.
\end{equation}
Set $\kappa_{p,\delta}=p(1-\frac{\delta}{2})-1$. Then, for all 
$
v_0\in L^q(\T^d;\R^\ell) 
$
such that $v_0\geq 0$ componentwise a.e.\ on $\T^d$, 
there exists a $(p,\a_{p,\delta},\delta,q)$-solution to \eqref{eq:reaction_diffusion} such that 
\begin{align}
\label{eq:LWP1}
v&\in  C([0,\tau);B^0_{q,p}(\T^d;\R^\ell))\cap H^{\vartheta,p}_{\loc}([0,\tau),w_{\a_{p,\delta}};H^{2-2\vartheta-\delta,q}(\T^d;\R^\ell))\text{ a.s.\ for all }\vartheta<1/2,\\
\label{eq:LWP2}
v&\in C^{\vartheta,\infty}_{\loc}((0,\tau)\times \T^d;\R^\ell)\text{ a.s.\ for all }\vartheta<1/2,\\
\label{eq:LWP3}
v&\geq 0 \text{ componentwise a.e.\ on } (0,\tau)\times \T^d.
\end{align}
\end{proposition}

In particular, due to \eqref{eq:LWP2}, \emph{all} $(p,\a_{p,\delta},\delta,q)$-solutions to stochastic RDEs \eqref{eq:reaction_diffusion} instantaneously become $C^{\infty}$ in space as long as $(p,\delta,q)$ satisfy \eqref{eq:pq_condition_local_WP}. Therefore, we refer to them as {\it classical solutions}. 
The space $L^q$ with $q=\frac{d(h-1)}{2}$ is particularly relevant for RDEs with nonlinearity growth $h$. Indeed, as shown in \cite[Subsection 1.4]{AV23_RDE_local}, the system \eqref{eq:reaction_diffusion} is (locally) invariant under the rescaling 
$$
v\mapsto \lambda^{1/(h-1)}v(\lambda \cdot ,\lambda^{1/2}\cdot)
$$
for $\lambda>0$. The space $L^{\frac{d}{2}(h-1)}$ is therefore (locally) invariant under the induced mapping on the initial data $v_0\mapsto \lambda^{1/(h-1)}v_0(\lambda^{1/2}\cdot)$. More precisely, 
$$
\lambda^{1/(h-1)}\|v_0(\lambda^{1/2}\cdot)\|_{L^{\frac{d}{2}(h-1)}(\R^d;\R^\ell)}
= \|v_0\|_{L^{\frac{d}{2}(h-1)}(\R^d;\R^\ell)}.
$$
Therefore, as is common in the study of PDEs, we say that the space $L^{\frac{d}{2}(h-1)}(\T^d;\R^\ell)$ is critical for \eqref{eq:reaction_diffusion}. Further comments on criticality and its role are given in Subsection \ref{ss:proof_strategy} below. As commented at the beginning of Subsection \ref{subsec:CRN}, in the reaction \eqref{reaction-full} we can allow for stoichiometric coefficients $y_{r,i},y_{r,i}'\in \{0\}\cup [1,\infty)$. However, the regularity statement in \eqref{eq:LWP2} has to be replaced by $v\in C^{\vartheta,\xi}_{\loc}((0,\tau)\times \T^d;\R^\ell)$ a.s.\ for all $\vartheta<1/2$, where $\xi\geq 1$ depends only on the values of the stoichiometric coefficients, see \cite[Theorem 4.2]{AV23_RDE_local}.

\smallskip

The following result complements Proposition \ref{prop:LWP} and improves the regularity at time $t=0$ in \eqref{eq:LWP2}, compensating the loss of regularity in the embedding \eqref{eq:embedding_Lq_B0qp}.

\begin{lemma}[Improved regularity at the initial time -- Lebesgue initial data]
\label{lem:Lebesgue_initial_data_continuity}
Under the assumption of Proposition \ref{prop:LWP}, the $(p,\a_{p,\delta},\delta,q)$-solution $(v,\tau)$ to \eqref{eq:reaction_diffusion} additionally satisfies
\begin{align*}
v&\in C([0,\tau);L^q(\T^d;\R^\ell))\cap L^2_{\loc}([0,\tau);H^{1}(\T^d;\R^\ell)) \text{ a.s.\ }\\
|v_i|^{q-2}|\nabla v_i|^2 &\in L^1_{\loc}([0,\tau)\times \T^d) \text{ a.s.\ for all }i\in \{1,\dots,\ell\}.
\end{align*}
\end{lemma}

In particular, the above result implies $|v_i|^{q/2}\in L^2_{\loc}([0,\tau);H^1(\T^d;\R^\ell))$ a.s. The proof of Lemma \ref{lem:Lebesgue_initial_data_continuity} follows a standard approximation procedure and, for completeness, is provided in Appendix \ref{app:continuity_Lq} (see \cite[Proposition 5.3]{AV24_dissipative} for a similar situation).

\subsection{Statement of the main results}
\label{ss:statement_main}
We are now ready to state the main results of the current manuscript. 
First, we prove that there exists a suitable choice of the noise coefficients $\theta=(\theta_k)_{k\in \Z^d_0}$ for which there exist global classical (in space) solutions to the stochastic RDEs \eqref{eq:reaction_diffusion}. 
Below, for convenience, we let
\begin{align}
\label{eq:initial_data_Xin}
\InD^q \coleq \Big\{v_0\in L^q(\T^d;\R^\ell)\,:\ 
&v_{0,i}\geq 0 \text{ a.e. on }\T^d \text{ and } \\ &\int_{\T^d} v_{0,i}\,\dd x \neq 0 \text{ for all } i\in \{1,\dots,\ell\} \Big\}. \nonumber
\end{align}

\begin{theorem}[Global classical solutions by transport noise]
\label{t:global}
Let \eqref{reaction-full} be a complex balanced chemical reaction network, and consider the reaction nonlinearities given by \eqref{b-1}. Fix a compact set $K\subseteq \q \R^{\ell}_{> 0}$ and 
$$
q>\frac{d(h-1)}{2}\vee 2, \qquad N \geq 1, \quad\text{ and }\quad \varepsilon\in (0,1).
$$ 
Assume that for each $M\in \q\R^\ell_{> 0}$, there are no boundary equilibria $\wt{v}_\infty$ such that $\q\wt{v}_\infty= M$ (see Definition \ref{def:complex_boundary_equilibria}).
Then, there exist $\nu>0$ and $\theta\in \ell^2$ such that $\#\{k\,:\,\theta_k\neq 0\}<\infty$ for which the following assertion holds: for all $\delta\in (1,2)$, $p\geq \frac{2}{2-\delta}\vee q$ such that $q>\frac{d}{d-\delta}$, and initial data $v_0$ satisfying
$$
v_0\in \InD^q, \qquad  \|v_0\|_{L^q(\T^d;\R^\ell)}\leq N, \qquad \text{ and }\qquad 
\q \Big(\int_{\T^d} v_0(x)\,\dd x\Big) \in K,
$$
the unique $(p,\a_{p,\delta},\s,q)$-solution with $\a_{p,\delta}=p(1-\frac{\delta}{2})-1$ to the {\normalfont{stochastic RDEs}} \eqref{eq:reaction_diffusion} is \emph{global in time with large probability}:
$$
\P(\tau=\infty)>1-\varepsilon.
$$
\end{theorem}

From \eqref{eq:LWP2} in Proposition \ref{prop:LWP}, it follows that the solution to \eqref{eq:reaction_diffusion} with $\theta$ and $\nu$ as in Theorem \ref{t:global} is global and \emph{classical} in space with large probability. 
A priori estimates accompanying the above global well-posedness with high probability result are given in Corollary \ref{cor:uniform_estimates_half_line}. 
The proof of the above result is given in Subsection \ref{ss:proof_global_well_posedness}.

As commented in Section \ref{s:intro}, the above result is not known in the deterministic setting $\nu=0$ when $h>2$.
The reader is referred to Subsection \ref{ss:examples_reaction_diffusion_intro} for concrete examples of chemical reactions to which our results apply, but for which the corresponding deterministic result is unknown.

\smallskip

Our second main result captures the enhanced dissipation effect of transport noise on the reaction--diffusion dynamics. The reader is referred to \eqref{eq:linear_PDE_intro}--\eqref{eq:oscillation_energy} for a discussion on physical aspects and the mathematical formulation.

\begin{theorem}[Enhanced dissipation by transport noise]
\label{t:enhanced_dissipation_reaction}
Let \eqref{reaction-full} be a complex balanced chemical reaction network, and consider the reaction nonlinearities given by \eqref{b-1}. Fix a compact set $K\subseteq \q \R^{\ell}_{> 0}$,
$q>\frac{d(h-1)}{2}\vee 2$, $N \geq 1$, 
$$
 \varepsilon\in (0,1),\qquad
\chi\in (0,\infty), \quad \text{ and }\quad b \in (1,\infty).
$$
Assume that for each $M\in \q\R^\ell_{> 0}$, there are no boundary equilibria $\wt{v}_\infty$ such that $\q\wt{v}_\infty=M$ (see Definition \ref{def:complex_boundary_equilibria}).
If $d\geq 5$, then additionally assume that $q>(d-2)(h-1)$. 
Then, there exist $\nu>0$ and $\theta\in \ell^2$ such that $\#\{k\,:\,\theta_k\neq 0\}<\infty$ for which the following assertion holds: for $\delta\in (1,2)$, $q>\frac{d}{d-\delta}$, and initial data $v_0$ satisfying
$$
v_0\in \InD^q , \qquad  \|v_0\|_{L^q(\T^d;\R^\ell)}\leq N, \qquad \text{ and }\qquad 
\q \Big(\int_{\T^d} v_0(x)\,\dd x\Big) \in K,
$$
let $(v,\tau)$ denote the unique $(p,\a_{p,\delta},\s,q)$-solution  for some $p\geq \frac{2}{2-\delta}\vee q$ and with $\a_{p,\delta}=p(1-\frac{\delta}{2})-1$ to the {\normalfont{stochastic RDEs}} \eqref{eq:reaction_diffusion}. Then, there exists a stopping time $\mu\leq \tau$, and a random variable $\const$ such that
\begin{equation}
\label{eq:enhanced_dissipation_reaction1}
\P(\mu=\infty)>1-\varepsilon 
\qquad \text{ and }\qquad
\E \const^b\lesssim_{K,q,N,\varepsilon,\chi,b}1,
\end{equation}
for which we have enhanced dissipation:
\begin{equation}
\label{eq:enhanced_dissipation_reaction2}
\Big\|v(t,x)-\int_{\T^d} v(t,x)\,\dd x \Big\|_{L^2(\T^d;\R^\ell)}
\leq \const e^{-\chi t} \|v_0\|_{L^2(\T^d;\R^\ell)}\ \text{ a.s.\ for all }t<\mu.
\end{equation}
\end{theorem}

Choosing $\chi\gg \max_{1\leq i\leq \ell}\nu_i$, the estimate \eqref{eq:enhanced_dissipation_reaction2} ensures that the transport noise in \eqref{eq:reaction_diffusion} enhances the decay of the spatial fluctuations compared to the case of pure diffusion, see \eqref{eq:basic_decay_L2}. 
The additional integrability condition for $d\geq 5$ appears to be purely technical, and we do not pursue its removal here.
The proof of Theorem \ref{t:enhanced_dissipation_reaction} is given in Subsection \ref{ss:proof_enhanced}.

Finally, although the random constant $\const$ may depend on $v_0$ and not only on its $L^q$-norm, the corresponding moment estimate is uniform over initial data in the class considered in Theorem \ref{t:enhanced_dissipation_reaction}. More precisely,
$$
\E \Big[\one_{\O_\infty} \sup_{t>0} \Big( e^{\chi t }\Big\|v(t,x)-\int_{\T^d} v(t,x)\,\dd x \Big\|_{L^2}\Big)^b \Big]
\lesssim_{K,q,N,\varepsilon,\chi,b} \|v_0\|_{L^2}^b,
$$
and $\O_\infty=\{\mu=\infty\}$ satisfies $\P(\O_\infty)>1-\varepsilon$. 

\smallskip

We conclude this subsection by discussing the extension of our results in the presence of boundary equilibria. The proof strategy is outlined in Subsection \ref{ss:proof_strategy}.

\begin{remark}[The case of boundary equilibria]
\label{rem:boundary_equilibria}
Theorems \ref{t:global} and \ref{t:enhanced_dissipation_reaction} are formulated for chemical reaction networks without boundary equilibria. However, the same arguments extend to networks with boundary equilibria provided that the corresponding deterministic system with sufficiently large diffusivity admits a global solution that remains uniformly bounded away from zero.
Indeed, the absence of boundary equilibria is used only to verify condition \ref{cond:eep-estimate-2} in Theorem \ref{t:entropy_dissipation}, which in turn enters our argument only through Proposition \ref{prop:global_high_viscosity}. A uniform positive lower bound for the deterministic solution provides an alternative way to verify condition \ref{cond:eep-estimate-1} and, therefore, allows the proof to proceed in the presence of boundary equilibria.
\end{remark}

\begin{figure}[htpb]
    \centering
    \begin{tikzpicture}[
        >=latex,
        mathnode/.style={
            align=center,
            text width=4.9cm,
            inner sep=3pt,
            font=\small
        },
        mainnode/.style={
            align=center,
            text width=5.3cm,
            inner sep=4pt,
            font=\small
        }
    ]

    \node[mainnode] (N0) {
        Uniform entropy--entropy dissipation\\
        \emph{Theorem \ref{t:entropy_dissipation}} (\S\ref{s:entropy_review})
    };

    \node[mathnode, below=0.9cm of N0] (N2) {
        Global well-posedness of\\ deterministic RDEs with high diffusivity\\
        \emph{Proposition \ref{prop:global_high_viscosity}} (\S\ref{ss:global_with_high_diffusivity})
    };

    \node[mathnode, left=0.25cm of N2] (N1) {
        Scaling limit for\\ stochastic RDEs\\ with cutoff\\
        \emph{Theorem \ref{t:scaling_limit_cutoff}} (\S\ref{ss:cutoff2})
    };

    \node[mathnode, right=0.25cm of N2] (N3) {
        Close-to-equilibrium implies\\ global for stochastic RDEs\\
        \emph{Theorem \ref{t:small_implies_global}} 
        (\S\ref{s:global_close_equilibria})
    };

    \node[mainnode, below=1.0cm of N2] (N4) {
        Global classical solutions by transport noise\\
        \textbf{Theorem \ref{t:global}}
        (\S\ref{ss:proof_global_well_posedness})
    };

    \node[mathnode] at (N1 |- N4) (N5) {
        Uniform estimates \\ for fluctuations \\
        \emph{Propositions \ref{prop:quenched_estimate_w}
        and \ref{prop:iteration_lemma}} (\S\ref{s:enhanced})
    };

    \node[mainnode, below=1.0cm of N4] (N6) {
        Enhanced dissipation\\
        \textbf{Theorem \ref{t:enhanced_dissipation_reaction}} (\S\ref{ss:proof_enhanced})
    };

    \draw[->] (N0) -- (N2);

    \draw[->] (N1.south) -- ([xshift=-1.15cm]N4.north);
    \draw[->] (N2) -- (N4);
    \draw[->] (N3.south) -- ([xshift=1.15cm]N4.north);

    \draw[->] (N4) -- (N6);
    \draw[->] (N5.south) -- ([xshift=-1.15cm]N6.north);

    \end{tikzpicture}

    \caption{Proof architecture. The main results (Theorems \ref{t:global} and \ref{t:enhanced_dissipation_reaction}) are highlighted in bold. Relevant subsections are indicated in parentheses.}
    \label{fig:proof_architecture}
\end{figure}
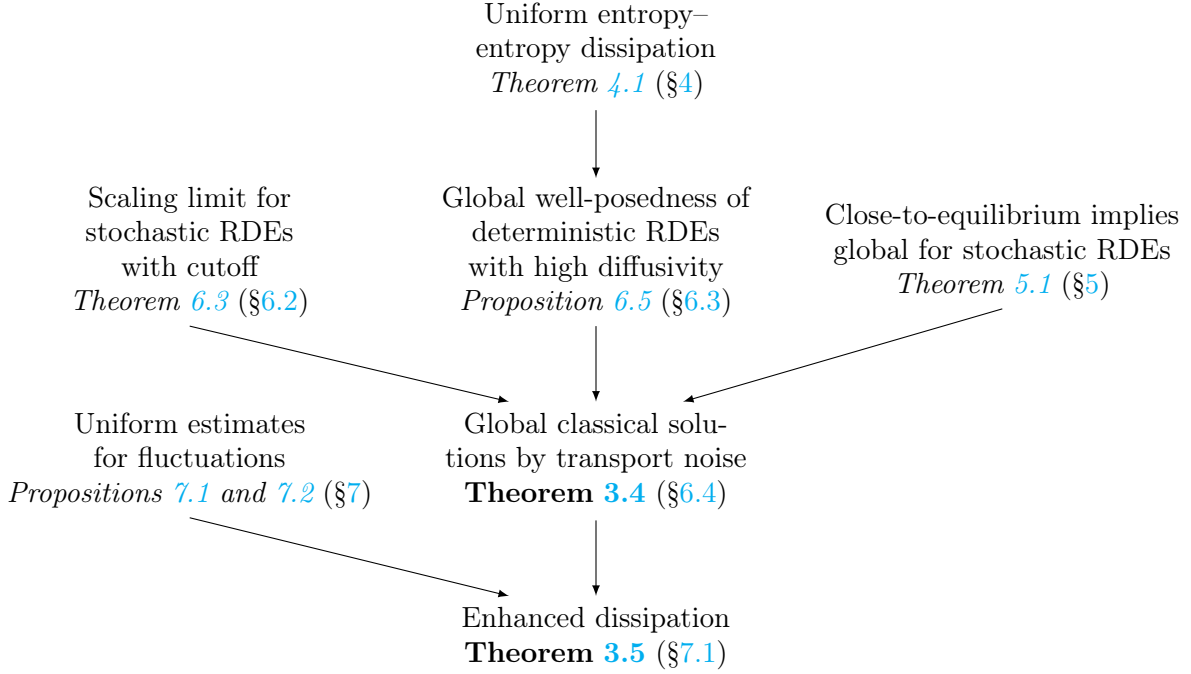

\subsection{Proof strategy}
\label{ss:proof_strategy}
The proof architecture is summarized in Figure \ref{fig:proof_architecture}. We follow the diagram from top to bottom. As a first ingredient, we show that the constants in the entropy--entropy dissipation estimates for complex balanced reaction networks developed in \cite{DFT17,FT18_normalized_entropy} can be made uniform with respect to initial data satisfying the conditions of Theorems \ref{t:global} and \ref{t:enhanced_dissipation_reaction}.  Obtaining this uniformity requires tracking the constants in the entropy--entropy dissipation argument. For this, we prove that the unique positive equilibrium of a complex balanced chemical reaction network depends continuously on the mass vector; see Lemma \ref{l:continuous_cinfty}. These estimates are then used to prove global well-posedness and exponential convergence to equilibrium for deterministic RDEs with sufficiently large diffusivity; see Proposition \ref{prop:global_high_viscosity}.

To prove Theorem \ref{t:global}, we also need two other ingredients. First, Theorem \ref{t:scaling_limit_cutoff} adapts the finite-time scaling-limit approach of \cite{A22} in which, for suitable noise coefficients, stochastic RDEs with cutoff (see \eqref{eq:reaction_diffusion_cut_off} below) are approximated by deterministic RDEs with enhanced diffusivity. Here, the key point is that the approximation is established in a space of the form $L^r(0,T;L^q)$ with $T>0$ large, and parameters $(r,q)$ satisfying the subcritical condition
$$
-\frac{2}{r}-\frac{d}{q}>-\frac{2}{h-1};
$$
see the discussion below Proposition \ref{prop:LWP} and \cite[Subsection 2.3]{A22}.
Second, Theorem \ref{t:small_implies_global} shows that closeness to a positive equilibrium implies global well-posedness with high probability. Its proof combines Moser-type arguments in the spirit of \cite{A22,DG15_boundedness} with the spectral-gap approach of \cite{tang2018close}. Importantly, the result is uniform with respect to the normalized transport-noise coefficients and allows nonlinearities of arbitrary polynomial growth. Combining the scaling limit, the deterministic high-diffusivity dynamics, and the close-to-equilibrium argument allows us to upgrade the finite-time result of Theorem \ref{t:scaling_limit_cutoff} to global well-posedness in Theorem \ref{t:global}. Details are given in Subsection \ref{ss:proof_global_well_posedness}.

Finally, the lower part of Figure \ref{fig:proof_architecture} concerns enhanced dissipation. Starting from the global control provided by Theorem \ref{t:global}, we study the energy of the spatial fluctuations
$$
\big\|v(t,\cdot)-\overline{ v}(t)\big\|_{L^2}^2.
$$ 
The argument builds on quantitative estimates for transport noise from \cite{FGL24_quantitative,L21}, adapted here to the nonlinear reaction--diffusion setting in an $L^r(L^q)$-framework. Two points are important when dealing with reaction nonlinearities of high polynomial growth. First, one needs to control the $L^{2(h-1)}(0,1;L^{d(h-1)})$-norm, which has the same parabolic scaling as the critical initial-data space $L^{\frac{d}{2}(h-1)}$ and is therefore critical; see Remark \ref{rem:criticality_pathwise_estimates}. This is achieved through the interpolation estimate of Lemma \ref{l:interpolation_Lr0} using the uniform bounds obtained in Proposition \ref{prop:global_cut_off_entropy}; see also \eqref{eq:boundedness_Lr0_proof}. Second, the reaction term appears in the centered form $f(v)-\overline{f(v)}$, whose zero-mean structure yields the cancellation needed to control the deterministic convolution. Proposition \ref{prop:quenched_estimate_w} provides pathwise $L^2$-estimates for the spatial fluctuations, while Proposition \ref{prop:iteration_lemma} gives a quantitative moment energy contraction that can be iterated in time. Their combination with Theorem \ref{t:global} yields Theorem \ref{t:enhanced_dissipation_reaction}.

\section{Uniform entropy--entropy dissipation estimates}
\label{s:entropy_review}
Fix a complex balanced chemical reaction network. 
For a measurable function $v=(v_i)_{i=1}^{\ell}:\T^d\to [0,\infty)^{\ell}$ and a positive equilibrium $v_{\infty}=(v_{\infty,i})_{i=1}^{\ell}\in (0,\infty)^{\ell}$, we define the relative entropy functional $\Ent(v|v_{\infty})$ as 
\begin{equation}
\label{eq:entropy}
\Ent (v|v_{\infty}) \coleq 
\sum_{1\leq i\leq \ell} \int_{\T^d} \Big(v_i \log \frac{v_i}{v_{\infty,i}}-v_i+v_{\infty,i}\Big)\,\dd x,
\end{equation}
where $0\log 0 \coleq 0$. Let us recall that the entropy functional is a standard Lyapunov functional for a nonlinear reaction--diffusion system, i.e., for smooth solutions of \eqref{b-2}--\eqref{b-1}, we have  
$\frac{\dd }{\dd t} \Ent(v(t)|v_{\infty})\leq 0$. 
The entropy dissipation functional $\D(v)=-\frac{\dd}{\dd t}\Ent(v(t)|v_{\infty}) \geq 0$ associated to \eqref{eq:entropy} is given by 
\begin{equation}
\begin{aligned}
\label{eq:dissipation}
&\D(v) \coleq 
\sum_{1\leq i\leq \ell} \nu_i \int_{\T^d} \frac{|\nabla v_i|^2}{v_i}\,\dd x
+ \sum_{1\leq r\leq m } k_r \sol_{\infty}^{\y_r} \int_{\T^d} \Psi\Big(\frac{\sol^{\y_r}}{\sol_{\infty}^{\y_r}},\frac{\sol^{\y_r'}}{\sol_{\infty}^{\y_r'}}\Big)\, \dd x\geq 0,\\
&\text{where }\quad 
\Psi(x,y) \coleq x\log(x/y)-x+y\geq 0,
\end{aligned}
\end{equation}
and was first derived in the above formulation in \cite{DFT17}. In the latter work, the entropy and the dissipation functionals have been used to study the convergence to equilibria of solutions to chemical reaction networks via the entropy--entropy dissipation inequality. We shall use such results in the deterministic and stochastic setting. 

\begin{theorem}[Entropy--entropy dissipation inequality \cite{DFT17,FT18_normalized_entropy}]
\label{t:entropy_dissipation}
Fix a complex balanced chemical reaction network. Let $N \geq 1$, $v_\mathrm{inf} \geq 0$, and $K \subseteq \q \R_{> 0}^{\ell}$ be a compact subset of $\R^{\ell'}$ such that at least one of the following conditions holds: 
\begin{enumerate}[label={\rm(C\arabic*)}]
\item\label{cond:eep-estimate-1} $v_\mathrm{inf} > 0$; 
\item\label{cond:eep-estimate-2} for all $M \in K$, there are no boundary equilibria $\wt{v}_\infty$ satisfying $\q \wt{ v}_\infty = M$.
\end{enumerate}
Then, there exists a constant $\lambda>0$ depending only on $N$, $v_\mathrm{inf}$, and $K$ such that 
\begin{align*}
\D(v) \geq \lambda\, 
\Ent(v|v_{\infty}) 
\end{align*}
holds for all measurable functions $v:\T^d\to [0,\infty)^{\ell}$ satisfying $\Ent(\overline{v}|v_{\infty})\leq N$, $\q \overline{v} = M \in K$, and $\overline{v_j} \geq v_\mathrm{inf}$ for all $j \in \{ 1, \dotsc, \ell\}$, where we recall the notation $\overline{v_j} = \int_{\T^d} v_j(x)\,\dd x$. 
\end{theorem}

The proof follows from the same line of argument as the proof of \cite[Theorem 1.1]{FT18_normalized_entropy} together with the continuity of the equilibrium $v_\infty$ as a function of the mass vector $M \in K$. We collect the necessary ingredients in the Lemmas \ref{l:continuous_cinfty}--\ref{l:continuous_H0} and we present the proof of Theorem \ref{t:entropy_dissipation} subsequently. 

\begin{lemma}
\label{l:continuous_cinfty} 
Assume that the reaction network \eqref{reaction-full} is complex balanced. For $M\in \q\R^\ell_{>0}$, let $v_\infty(M)$ be the positive equilibrium given by Lemma \ref{lem1}. Then, the mapping $M\mapsto v_\infty(M)$ from $\q \R^{\ell}_{> 0}$ to $\R^\ell_{>0}$ is continuous.
\end{lemma}

\begin{proof}
Let $v_\infty^* \in \R_{>0}^\ell$ be a complex balanced equilibrium and $M^* \in \R^{\ell'}$ such that $\q v_\infty^* = M^*$. We first recall from Lemma \ref{lem1} that for all $M \in \q \R^{\ell}_{> 0}$, there exists a unique complex balanced equilibrium $v_\infty \in \R_{>0}^\ell$ such that $\q v_\infty = M$. 

We now prove that $v_\infty \rightarrow v_\infty^*$ for $M \rightarrow M^*$. The crucial ingredient is the existence of a vector $s^\perp \in (\mathrm{ran} \, \W^\top)^\perp$ satisfying 
\begin{align}
	\label{eq:s-perp}
	v_{\infty,i} = v_{\infty,i}^* e^{s_i^\perp}
\end{align} 
for all $i = 1, \dots, \ell$; see \cite[Theorem 2.3 (2.)]{YC18} and also \cite{feinberg1995existence,HJ72}. 
Recalling \eqref{eq:ranWT=kerQ}, there exist unique $r^\perp \in (\mathrm{ran} \, \W^\top)^\perp = (\mathrm{ker}\, \q)^\perp$ and $r \in \mathrm{ran} \, \W^\top = \mathrm{ker} \, \q$ such that 
\begin{align}
	\label{eq:r-perp_r}
	v_\infty - v_\infty^* = r^\perp + r. 
\end{align} 
Invoking $\q v_\infty = M$, this entails 
\begin{align}
	\label{eq:r-perp}
	\q r^\perp = M - M^*. 
\end{align}
The identities \eqref{eq:s-perp} and \eqref{eq:r-perp_r} imply 
\begin{align}
	\label{eq:r_s-perp_r-perp}
	r_i = v_{\infty,i}^* \big( e^{s_i^\perp} - 1 \big) - r_i^\perp
\end{align}
for all $i = 1, \dots, \ell$. For $M \rightarrow M^*$, we immediately get $r^\perp \rightarrow 0$ from \eqref{eq:r-perp}, and we aim to show that also $s^\perp \rightarrow 0$. To this end, we suppose that 
\begin{align}
	\label{eq:s_i0-perp}
	\big| s_{i_0}^\perp \big| > \frac{6\ell|r^\perp|}{\underline v}
\end{align}
holds with $\underline v \coleq \min_i v_{\infty,i}^*$ for some $1 \leq i_0 \leq \ell$. We will see that this leads to a contradiction, which eventually establishes the claim of this lemma. For convenience, we introduce the functions $f : \R \rightarrow \R$, $g : \R \rightarrow \R$, 
\begin{align*}
	f(x) &\coleq \underline v (e^x - 1) x - |r^\perp| |x|, \\ 
	g(x) &\coleq \frac{\underline v}{2} x^2 - |r^\perp| |x|.
\end{align*}
Starting from \eqref{eq:r_s-perp_r-perp}, we estimate (for $r^\perp$ sufficiently small) 
\begin{align*}
	r_i s_i^\perp &\geq \min_{s_i^\perp \in \R} f(s_i^\perp) \geq \min_{s_i^\perp \in \R} g(s_i^\perp) = - \frac{|r^\perp|^2}{2 \underline v}, 
\end{align*}
where the second estimate only holds for the minima of $f$ and $g$ but not for the functions themselves; here, we use the asymptotic behavior 
\begin{align}
	\label{eq:(e^x-1)x-asymptotics}
	(e^x - 1) x \sim 
	\begin{cases}
		|x| &\text{for}\ x \rightarrow -\infty, \\ 
		x^2 &\text{for}\ x \rightarrow 0, \\ 
		x e^x &\text{for}\ x \rightarrow +\infty.
	\end{cases}
\end{align}
Hence, $\sum_{i \neq i_0} r_i s_i^\perp \geq - \ell (2 \underline v)^{-1} |r^\perp|^2$. Note that $g$ attains its minimum $-|r^\perp|^2/(2 \underline v)$ at $\pm |r^\perp|/\underline v$, while $g(\pm 2 |r^\perp|/\underline v) = 0$. Observe that the function 
\begin{align*}
	h(x) \coleq \frac{|r^\perp|}{6} |x| 
\end{align*}
satisfies $h(\pm 3 |r^\perp|/\underline v) = |r^\perp|^2/(2 \underline v) < g(\pm 3 |r^\perp|/\underline v)$. For $r^\perp$ sufficiently small, \eqref{eq:(e^x-1)x-asymptotics} guarantees $g(\pm 3 |r^\perp|/\underline v) < f(\pm 3 |r^\perp|/\underline v)$ and $f(x) > h(x)$ on $\R \backslash (-3 |r^\perp| / \underline v, 3 |r^\perp| / \underline v)$. We now employ \eqref{eq:s_i0-perp} to derive 
\begin{align*}
	r_{i_0} s_{i_0}^\perp \geq f(s_{i_0}^\perp) > h(s_{i_0}^\perp) > \frac{\ell|r^\perp|^2}{\underline v}.
\end{align*}
Thus, $r \cdot s^\perp > \ell (2 \underline v)^{-1} |r^\perp|^2 > 0$, which contradicts $r \in \mathrm{ran} \, \W^\top$ and $s^\perp \in (\mathrm{ran} \, \W^\top)^\perp$. 
\end{proof}

\begin{lemma}
Let $M\in \q \R^{\ell}_{> 0}$ and $v_\infty$ be the positive complex balanced equilibrium determined by $M$, and let $\hat K>0$ be fixed. Then, recalling the notation from \eqref{b-1}, it holds that
\begin{align*}
    \Xi \coleq \liminf_{\Sigma_{\hat K,M} \ni \vec v \rightarrow v_\infty} \frac{\sum_{r = 1}^m \Big( \sqrt{\frac{\vec v}{v_\infty}}^{\y_r} - \sqrt{\frac{\vec v}{v_\infty}}^{\y_r'} \Big)^2}{\sum_{i = 1}^\ell \Big( \sqrt{\frac{\vec v_i}{v_{\infty,i}}} - 1 \Big)^2} \geq \sigma_{\hat A^{-1}}^2 \hat \sigma_\W^2 > 0, 
\end{align*}
where $\Sigma_{\hat K,M} \coleq \{ \vec v = (\vec v_1,\ldots, \vec v_{\ell})\in [0, \hat K]^\ell \, | \, \q \vec v = M \}$, $\sigma_{\hat A^{-1}}$ is the smallest singular value of the inverse of $\hat A$ defined in \eqref{eq:def-A-matrix} and \eqref{eq:def-hat-a-matrix}, and where
$\hat \sigma_\W$ is the smallest non-zero singular value of $\W$. 
\end{lemma}

\begin{proof}
As in the proof of \cite[Lemma 2.8]{FT18_normalized_entropy}, we note that 
\begin{align*}
\Xi = \liminf_{\Sigma_{\hat K,M} \ni \vec v \rightarrow v_\infty} \Big| \W \frac{\eta}{|\eta|} \Big|_2^2
\end{align*}
holds, where $\eta \coleq \frac{\vec v - v_\infty}{v_\infty}$. We decompose $\eta = \eta^\perp + \eta^\parallel$ with $\eta^\perp \in (\ker \W)^\perp$ and $\eta^\parallel \in \ker \W$. As the rows of $\q$ form a basis of $\ker \W$ (see \eqref{eq:ranWT=kerQ}), there exists some $\gamma$ such that $\eta^\parallel = \q^\top \gamma$. Thanks to the conservation laws $\q \vec v = M = \q v_\infty$, we have $0 = \q \, \text{diag}(v_\infty) \big( \eta^\perp + \q^\top \gamma \big)$ and 
\begin{align}
    \label{eq:represent-gamma}
    \gamma = - \big( \q \, \text{diag} (v_\infty) \q^\top \big)^{-1} \q \, \text{diag}(v_\infty) \eta^\perp.
\end{align}
Indeed, $\q \, \text{diag} (v_\infty) \q^\top \in \mathbb R^{\ell' \times \ell'}$ is invertible since 
\begin{align*}
    \q \, \text{diag} (v_\infty) \q^\top x = 0 \quad \Rightarrow \quad \big\langle \q^\top x, \text{diag}(v_\infty) \q^\top x \big\rangle = 0 \quad \Rightarrow \quad \q^\top x = 0 \quad \Rightarrow \quad x = 0
\end{align*}
for all $x \in \mathbb R^{\ell'}$, where we used the linear independence of the rows of $\q$ in the last step. 
Inserting \eqref{eq:represent-gamma} into $\eta = \eta^\perp + \q^\top \gamma$, we obtain $\eta = A \eta^\perp$ with 
\begin{align}
    \label{eq:def-A-matrix}
    A \coleq I - \q^\top \big( \q \, \mathrm{diag}(v_\infty) \q^\top \big)^{-1} \q \, \mathrm{diag}(v_\infty) \in \mathbb R^{\ell \times \ell}.
\end{align}
This shows that 
\begin{align*}
\Xi \geq \inf_{\eta \in A ( (\ker \W)^\perp ), \, |\eta| = 1} | \W \eta |_2^2.
\end{align*}
The mapping $A : (\ker \W)^\perp \rightarrow A \big( (\ker \W)^\perp \big)$ induced by the matrix $A$ is regular by construction, which is easily seen as follows. If $A \eta^\perp = \eta^\perp + \q^\top \gamma = 0$ for some $\eta^\perp \in (\ker \W)^\perp$, then $\eta^\perp = 0$ due to the orthogonality relation $\eta^\perp \perp \q^\top \gamma$. Next, we define the mapping $\hat A : \mathbb R^{\ell} \rightarrow \mathbb R^\ell$ as a bijective extension of $A$ such that 
\begin{align}
\label{eq:def-hat-a-matrix}
\hat A = A \quad \text{on} \ \ (\ker \W)^\perp.
\end{align}
For instance, one can choose basis vectors $w_1, \dotsc, w_{\ell-\ell'}$ and $w_{\ell-\ell'+1}, \dotsc, w_\ell$ of $(\ker \W)^\perp$ and $\ker \W$, respectively, and define $\hat A w_j \coleq A w_j$ for $j = 1, \dotsc, \ell-\ell'$ and $\hat A w_j \coleq w_j'$ for $j = \ell-\ell'+1, \dots, \ell$, where $w_{\ell-\ell'+1}', \dotsc, w_\ell'$ extend the linear independent vectors $A w_1, \dotsc, A w_{\ell-\ell'}$ to a basis of $\mathbb R^\ell$. As usual, we also write $\hat A \in \mathbb R^{\ell \times \ell}$ for the corresponding matrix representation.  
We observe that 
\begin{align*}
    | \eta^\perp |_2^2 = \big| \hat A^{-1} \eta \big|_2^2 \geq \sigma_{\hat A^{-1}}^2
\end{align*}
holds for all $\eta \in A \big( (\ker \W)^\perp \big)$ satisfying $|\eta| = 1$ with the smallest singular value $\sigma_{\hat A^{-1}}$ of $\hat A^{-1}$. As a consequence, 
\begin{align*}
    \Xi \geq \inf_{\eta \in A ( (\ker \W)^\perp ), \, |\eta| = 1} | \W \eta^\perp |_2^2 
    \geq \sigma_{\hat A^{-1}}^2 \inf_{\eta^\perp \in (\ker \W)^\perp, \, |\eta^\perp|=1} |\W \eta^\perp|_2^2 
    = \sigma_{\hat A^{-1}}^2 \hat \sigma_\W^2
\end{align*}
with the smallest non-zero singular value $\hat \sigma_\W$ of $\W$. Indeed, let $\W = U \Sigma T^\top$ be a singular value decomposition of $\W$, where the columns $u_i$ of $U \in \mathbb R^{m \times m}$ and the columns $t_i$ of $T \in \mathbb R^{\ell \times \ell}$ form orthonormal bases. Let $\eta^\perp = \sum_i \alpha_i t_i \in (\ker \W)^\perp$ satisfy $|\eta^\perp| = 1$. Then, $\sum_i \alpha_i^2 = 1$ and every $t_i$ corresponds to a non-zero singular value $\sigma_i$; otherwise, we would have $\W t_i = 0$ and $\alpha_i = 0$. Consequently, 
\begin{align*}
    |\W \eta^\perp|_2^2 = \Big| \sum_i \alpha_i U \Sigma T^\top t_i \Big|_2^2 = \Big| \sum_i \alpha_i \sigma_i u_i \Big|_2^2 = \sum_i \alpha_i^2 \sigma_i^2 \geq \hat \sigma_\W^2
\end{align*}
and the claim follows. 
\end{proof}

\begin{lemma}
\label{l:continuous_H0}
Let $v_\mathrm{inf} \geq 0$ and $K \subseteq \q \R_{> 0}^{\ell}$ be a compact subset.
Assume that either condition \ref{cond:eep-estimate-1} or condition \ref{cond:eep-estimate-2} of Theorem \ref{t:entropy_dissipation} holds.
Then, the constant $H_1 > 0$ in \cite[Eq.\ (11)]{FT18_normalized_entropy} is subject to 
\begin{align*}
    H_1 = \inf_{\vec v \in \Sigma_{v_\mathrm{inf},\hat K,M}} \frac{\sum_{r = 1}^m \Big( \sqrt{\frac{\vec v}{v_\infty}}^{\y_r} - \sqrt{\frac{\vec v}{v_\infty}}^{\y_r'} \Big)^2}{\sum_{i = 1}^\ell \Big( \sqrt{\frac{\vec v_i}{v_{\infty,i}}} - 1 \Big)^2} \geq H_0 > 0, 
\end{align*}
where $\Sigma_{v_\mathrm{inf},\hat K,M} \coleq \{ \vec v \in [v_\mathrm{inf}, \hat K]^\ell \, | \, \q \vec v = M \}$, and where $K \ni M \mapsto H_0 \in \R_{>0}$ is continuous. 
\end{lemma}
\begin{proof}
For fixed $M \in K$, the equilibrium $v_\infty$ and the matrix $A$ are also fixed and there exists a neighborhood $U \subset \Sigma_{v_\mathrm{inf},\hat K,M}$ of $v_\infty$ such that 
\begin{align*}
    F(\vec v, v_\infty) \coleq \frac{\sum_{r = 1}^m \Big( \sqrt{\frac{\vec v}{v_\infty}}^{\y_r} - \sqrt{\frac{\vec v}{v_\infty}}^{\y_r'} \Big)^2}{\sum_{i = 1}^\ell \Big( \sqrt{\frac{\vec v_i}{v_{\infty,i}}} - 1 \Big)^2} \geq \frac{\sigma_{\hat A^{-1}}^2 \hat \sigma_\W^2}{2}
\end{align*}
holds for all $\vec v \in U$. Moreover, $F(\vec v, v_\infty)$ is uniformly positive for all $\vec v \in \Sigma_{v_\mathrm{inf},\hat K,M} \backslash U$, which can be seen as follows. According to \cite[eq.\ (30)]{FT18_normalized_entropy}, $F(\vec v, v_\infty) = 0$ with $\vec v \in \Sigma_{v_\mathrm{inf},\hat K,M} \backslash U$ implies that $\vec v$ is a complex balanced equilibrium. But in both cases \ref{cond:eep-estimate-1} and \ref{cond:eep-estimate-2}, the set $\Sigma_{v_\mathrm{inf},\hat K,M} \backslash U$ does not contain any complex balanced equilibrium. Hence, 
\begin{align*}
    H_1 \geq H_0 \coleq \min \bigg\{ \frac{\sigma_{\hat A^{-1}}^2 \hat \sigma_\W^2}{2}, \inf_{\vec v \in \Sigma_{v_\mathrm{inf},\hat K,M} \backslash U} F(\vec v, v_\infty) \bigg\} > 0.
\end{align*}
Since $F(\vec v,v_\infty)$ and $A\in\R^{\ell\times\ell}$ defined in \eqref{eq:def-A-matrix} depend continuously on $v_\infty\in\R^\ell_{>0}$, and since $\Sigma_{v_\mathrm{inf},\hat K,M}$ and $v_\infty$ depend continuously on $M\in K$, the assertion follows. In particular, by the compactness of $K$, $H_0$ can be chosen uniformly positive for $M\in K$.
\end{proof}

\begin{proof}[Proof of Theorem \ref{t:entropy_dissipation}]
We start with establishing $L^1(\T^d)$-bounds on $v$. As in \cite{FT18_normalized_entropy}, we recall the elementary estimate $x\log(x/y) - x + y \geq x/2 - y$ for $x \geq 0$ and $y > 0$. Applied to \eqref{eq:entropy} and using the assumption $\Ent(\overline{v}|v_{\infty})\leq N$, we obtain, for all $i \in \{ 1, \dotsc, \ell\}$,
\begin{align*}
    \overline{v_i} \leq \wh{K} \coleq 2 \Big( N + \sum_{1 \leq i \leq \ell} \max_{M\in K} v_{\infty, i}(M) \Big),
\end{align*}
where the finiteness of $\wh K$ follows from Lemma \ref{l:continuous_cinfty} and the compactness of $K$.

Suppose first that \ref{cond:eep-estimate-2} holds. Then, we can employ the proof of Thm.\ 1.1 on p.\ 17 in \cite{FT18_normalized_entropy} together with the preceding lemmas in \cite{FT18_normalized_entropy} almost verbatim up to the derivation of the constant $H_1$ therein, which we have to replace by our constant $H_0$ from Lemma \ref{l:continuous_H0}. In detail, the positive constants $K_1$ and $K_2$ in \cite{FT18_normalized_entropy} are easily seen to depend continuously on $v_\infty \in \R^\ell_{>0}$, which, in connection with our Lemma \ref{l:continuous_cinfty} above, guarantees that $K \ni M \mapsto K_i(M)=K_i\in \R_{>0}$ is continuous for $i = 1,2$. From the compactness of $K$, we deduce that $\max_{M \in K} K_1$ and $\min_{M \in K} K_2$ are positive constants, which are equally admissible in \cite{FT18_normalized_entropy} and which are independent of $M \in K$. Likewise, we estimate $H_1 \geq H_0$ according to Lemma \ref{l:continuous_H0} above, and we note that $\min_{M \in K} H_0$ is a valid positive constant in the proof of Thm.\ 1.1 in \cite{FT18_normalized_entropy} since $K$ is compact. 

Now, suppose that \ref{cond:eep-estimate-1} holds. Then, the same arguments are applicable up to a minor adaptation in the proof of \cite[Lemma 2.8]{FT18_normalized_entropy}, where it is assumed that no boundary equilibria exist, i.e.\ that \ref{cond:eep-estimate-2} holds. But this hypothesis is only used to ensure that any complex balanced equilibrium $\wt{v}_\infty \in \R^\ell_{\geq 0}$ coincides with $v_\infty \in \R^\ell_{>0}$. In the case of \ref{cond:eep-estimate-1}, this conclusion trivially holds as we only consider functions $v:\T^d\to [0,\infty)^{\ell}$ satisfying $\overline v \in [v_\mathrm{inf}, \hat K]^\ell$. 
\end{proof}

\section{Global well-posedness of RDEs close to equilibria with $L^{\infty}$-noise}
\label{s:global_close_equilibria}
This section is devoted to the proof of the following result.

\begin{theorem}[Global well-posedness with data close to the equilibrium]
\label{t:small_implies_global}
Let \eqref{reaction-full} be a complex balanced chemical reaction network. 
Fix $\delta\in (1,2)$. Let $q,p\in (2,\infty)$ be such that 
$$
q\geq \frac{d(h-1)}{2}\vee \frac{d}{d-\delta}, \quad p\geq \frac{2}{2-\delta} \vee q \quad \text{ and } \quad \a_{p,\delta}=p(1-\frac{\delta}{2})-1.
$$
Fix a compact set $K\subseteq \q\R^\ell_{> 0}$.
Let $(v,\tau)$ be the unique solution to stochastic RDEs \eqref{eq:reaction_diffusion} with initial data $v_0\in L^q(\T^d;\R^\ell_{\geq 0})$ such that $\q\overline{v_0}\in K$ provided by Proposition \ref{prop:LWP}. 
Let $v_\infty\in \R_{>0}^\ell$ be the complex balanced equilibrium satisfying $\q \overline{v_0}=\q v_\infty$ (Lemma \ref{lem1}). 
Then, there exists a constant $R(q,p,\delta,K)>0$ such that the following assertion holds: 
for any $\varepsilon\in (0,1)$, there exists a constant $\eta(q,p,\delta,K,\varepsilon)>0$ such that, if for some stopping time $\tau_0$, one has 
\begin{equation}
\label{eq:small_implies_global_assumption}
\P(\tau_0<\tau,\, \|v(\tau_0)-v_\infty\|_{L^q}\leq\eta)>1- \varepsilon,
\end{equation}
then there exists a stopping time $\tau_1\in (0,\tau]$ such that $ \P(\tau_1=\infty) >1-\varepsilon$ and 
\begin{equation}
\begin{aligned}
\label{eq:small_implies_global_conclusion}
&\sup_{t\in [\tau_0,\tau_1)}\|v(t)-\veq\|_{L^q}^q
+\max_{1\leq i\leq \ell} \int_{\tau_0}^{\tau_1} \int_{\Tor^d}|v_i-\veqi|^{q-2}|\nabla v_i|^2\,\dd x\,\dd r\\
&+\|v-\veq\|_{L^{q_0}(\tau_0,\tau_1;L^{q_0})}^q + \|v-\veq\|_{L^2(\tau_0,\tau_1;L^2)}^2\leq 
R\|v(\tau_0)-v_\infty\|_{L^q}^2(1+\|v(\tau_0)-v_\infty\|_{L^q}^{q-2}) 
\end{aligned}
\end{equation}
a.s.\ on $\{\tau_0<\tau,\, \|v(\tau_0)-v_\infty\|_{L^q}\leq\eta\}$, where
\begin{equation}\label{eq:def_q0}
q_0=q \frac{d+2}{d} > q.
\end{equation}
\end{theorem}

Note that if \eqref{eq:small_implies_global_assumption} holds, then in particular, 
$$
\P(\tau=\infty)>1-\varepsilon.
$$ 
Let us stress that in Theorem \ref{t:small_implies_global} we do not have to exclude boundary equilibria.

\smallskip

It will be of central importance that the smallness constant $\eta$ is independent of the normalized radially symmetric coefficient $\theta\in \ell^2$. This in particular implies that we do not require any smoothness of the noise beyond its boundedness, and consequently, we assume that the noise in \eqref{eq:reaction_diffusion} is merely bounded. 
Hence, our proof of Theorem \ref{t:small_implies_global} is in the spirit of Moser-type iteration as in \cite[Section 5]{A22} (see also \cite{DG15_boundedness}) and the spectral gap methods as in \cite{tang2018close}. The latter allows one to exploit the closeness to the equilibria $v_\infty$ to obtain exponential stability of the linearized system. We point out that Theorem \ref{t:small_implies_global} not only extends the results in \cite{tang2018close} to the stochastic setting, but also removes unnecessary conditions on the growth of the nonlinearity.
In particular, $h$ in Theorem \ref{t:small_implies_global} can be arbitrarily large.

\smallskip

The proof of Theorem \ref{t:small_implies_global} is given in Subsection \ref{ss:proof_small_implies_global} below. To this end, we collect some preliminary results in Subsection \ref{ss:preliminaries_global_small}.

\subsection{Intermediate estimates}
\label{ss:preliminaries_global_small}
Clearly, $w=v -\veq$ solves
\begin{equation}
\label{eq:reaction_diffusion_w}
\left\{
\begin{aligned}
\dd w_i -\nu_i \Delta w_i \,\dd t &= f_i(w+v_\infty)\, \dd t+\sqrt{c_d \nu}\sum_{k,\alpha} \theta_k (\sigma_{k,\alpha} \cdot\nabla) w_i\circ \dd W^{k,\alpha}_{t}   & \text{ on }&\Tor^d,\\
w_{i}(0)&=w_{0,i} &\text{ on }&\Tor^d,
\end{aligned}
\right.
\end{equation}
where $w_0 \coleq v_{0}-\veq$.
In the following, we set $q\geq \frac{d(h-1)}{2}\vee 2$ and we define $q_0$ as in \eqref{eq:def_q0}.

\subsubsection{$L^2$-estimates via spectral gap}
This subsection aims to prove the following $L^2$-estimates.
\begin{lemma}[Pathwise $L^2$-estimates]
\label{l:L2_estimates_small_implies}
Let $q\geq \frac{d(h-1)}{2}\vee 2$. Then, there exists a constant $C>0$ depending on $d$ and the parameters in the reaction network \ref{reaction-full} for which the following assertion holds for all $v_0\in L^q(\T^d;\R^\ell_{\geq 0})$. Let $(v,\tau)$ be the solution to \eqref{eq:reaction_diffusion}. Then, for all stopping times $\tau_0$ such that $0\leq \tau_0<\tau$ a.s., we have, a.s.\ for all $t<\tau$, 
\begin{align*}
\sup_{r\in [\tau_0,t)}\|w(r)\|_{L^2}^2  
+\int_{\tau_0}^t \| w\|_{L^2}^2\,\dd r 
&\leq C\|w(\tau_0)\|_{L^2}^2 +C\int_{\tau_0}^t\int_{\T^d} |w|^{q_0}\,\dd x \,\dd r.
\end{align*} 
\end{lemma}

In order to do this, we utilize the arguments in \cite{tang2018close}. Note that, by writing 
$$
f_i(w+\veq)=f_i(w+\veq)-f_i(\veq) \eqcol \nabla f_i (\veq) \cdot w+ R_i(w)
$$
where $R_i(w) \coleq f_i(w+v_\infty) - \nabla f_i(v_\infty)\cdot w$, we have for some $\beta\in (0,1)$, see \cite[Lemma 4.4]{tang2018close},
\begin{equation*}
    |R_i(w)|\leq C(|w|^{1+\beta}+|w|^{h}) \ \text{ for all } w \in \R^\ell.
\end{equation*}

We have the following spectral gap estimate, which is taken from \cite[Lemma 3.3]{tang2018close}.

\begin{lemma}[Spectral gap for the linearized operator around positive equilibrium]
\label{l:spectral_gap}
Assume that the chemical reaction network \eqref{reaction-full} is complex balanced. Fix a compact set $K\subseteq \q \R_{>0}^\ell$. Then, there exists a constant $\lambdalin>0$, which depends continuously on $\nu_i$, the stoichiometric coefficients and the compact set $K$ such that, for any equilibria $v_\infty\in \R^\ell_{>0}$ and $z\in H^1(\T^d;\R^\ell)$ satisfying $\q v_\infty \in K$ and $\q \overline{z}=0$, we have 
\begin{equation}\label{spectral-gap-inequality}
\sum_{1\leq i\leq \ell} \Big( \nu_i \frac{\|\nabla z_i\|_{L^2}^2}{\veqi} - \int_{\T^d} (\nabla f_i(\veq)\cdot z)\frac{ z_i}{\veqi}\,\dd x \Big)\geq \lambdalin\sum_{1\leq i\leq \ell} \frac{\|z_i\|_{L^2}^2}{\veqi}.
\end{equation}
\end{lemma}

\begin{proof}
The existence of a constant $\lambdalin>0$ such that the spectral inequality \eqref{spectral-gap-inequality} holds was given in \cite[Lemma 3.3]{tang2018close}, where the continuous dependence of $\lambdalin$ on $\nu_i$, the stoichiometric coefficients, and $v_\infty$ follows from the proof. Finally, the uniformity of $\lambdalin$ with respect to equilibria $v_\infty\in\R^\ell_{>0}$ such that $\q v_\infty\in K$ follows from the continuity of $M\mapsto v_\infty(M)$, proved in Lemma \ref{l:continuous_cinfty}, and the compactness of $K$.
\end{proof}

\begin{proof}[Proof of Lemma \ref{l:L2_estimates_small_implies}]
From Lemma \ref{lem:Lebesgue_initial_data_continuity}, we can apply It\^o's formula to the normalized energy functional 
$$
w\mapsto\frac{1}{2}\sum_{1\leq i\leq \ell} \frac{\|w_i\|_{L^2}^2}{\veqi}.
$$
We obtain, a.s.\ for all $t\in (\tau_0,\tau)$,
\begin{align*}\frac{1}{2}
\sum_{1\leq i\leq \ell }\frac{\|w_i(t)\|_{L^2}^2}{\veqi} 
&=\frac{1}{2} \sum_{1\leq i\leq \ell }\frac{\|w_i(\tau_0)\|_{L^2}^2}{\veqi} 
-\sum_{1\leq i\leq \ell}
\int_{\tau_0}^t\int_{\T^d}  \Big(\nu_i \frac{|\nabla w_i|^2}{\veqi}
-[\nabla f_i(\veq)\cdot w] \frac{w_i}{\veqi} \Big)\,\dd x\,\dd r\\
&+\sum_{1\leq i\leq \ell} \int_{\tau_0}^t\int_{\T^d} R_i(w) \frac{w_i}{\veqi}\,\dd x\,\dd r \\
&\leq\frac{1}{2} \sum_{1\leq i\leq \ell }\frac{\|w_i(\tau_0)\|_{L^2}^2}{\veqi} -\frac{\lambda_{{\rm lin}}}{2} \sum_{1\leq i\leq \ell }\int_{\tau_0}^t\frac{\|w_i(r)\|_{L^2}^2}{\veqi}\, \dd r 
+C\int_{\tau_0}^t\int_{\T^d} |w|^{q_0}\,\dd x \,\dd r, 
\end{align*}
where the last inequality follows from Lemma \ref{l:spectral_gap} and Young's inequality
$$
|R_i(w)w_i|\leq C(|w|^{2+\beta}+ |w|^{h+1})\leq \frac{\lambda_{{\rm lin}}}{2}\sum_{1\leq i\leq \ell}\frac{|w_i|^2}{\veqi} + C|w|^{q_0}
$$
as $q_0 = q(1+2/d) \ge h + 1$ and $q\geq 2$. The conclusion of Lemma \ref{l:L2_estimates_small_implies} follows as $\veqi>0$ for all $i\in \{1,\dots,\ell\}$ and Lemma \ref{l:continuous_cinfty}.
\end{proof}

\subsubsection{$L^q$-estimates}
In this subsection, we prove the following result.

\begin{lemma}[Pathwise $L^q$-estimate]
\label{l:Lq_estimate_small}
Let $q\geq \frac{d(h-1)}{2}\vee 2$ and $K\subseteq \q \R^\ell_{>0}$. Then, there exists a constant $C_q>0$ depending on $q$, $d$, $K$ and the parameters in the reaction network \ref{reaction-full} for which the following assertion holds for all $v_0\in L^q(\T^d;\R^\ell_{\geq 0})$. Let $(v,\tau)$ be the solution to \eqref{eq:reaction_diffusion}. Then, for all stopping times $\tau_0$ such that $0\leq \tau_0<\tau$ a.s., we have, 
a.s.\ for all $t<\tau$, 
\begin{equation*}
\begin{aligned}
\sup_{r\in [\tau_0,t)}\|w(r)\|_{L^q}^q 
&+\max_{1\leq i\leq  \ell}\int_{\tau_0}^t \int_{\T^d} |w_i|^{q-2}|\nabla w_i|^2\,\dd x \,\dd r +
\|w\|_{L^{q_0}(\tau_0,t;L^{q_0})}^q\\
&\leq C_q \|w(\tau_0)\|_{L^q}^q+ \|w\|_{L^{q_0}(\tau_0,t;L^2)}^2+ C_q \int_{\tau_0}^t \int_{\T^d} (|w|^2+|w|^{q_0})\,\dd x \,\dd r .
\end{aligned}
\end{equation*}
\end{lemma}

To prove Lemma \ref{l:Lq_estimate_small}, we employ the following result.

\begin{lemma}[Nonlinear estimates in the $L^q$-setting I]
\label{l:nonlinear_estimate_easy}
Fix $q\geq \frac{d(h-1)}{2}\vee 2$ and let $q_0$ be as in \eqref{eq:def_q0}. 
Then, there exists a positive constant $C_0$ independent of $t>0$ such that
\begin{align}
\label{eq:interpolation_subcriticality_apriori_estimates_0}
\|u\|_{L^{q_0}((0,t)\times \T^d)}&\leq C_0\|u\|_{L^{q_0}(0,t;L^2)}\\
\nonumber
&+ C_0 \|u\|_{L^{\infty}(0,t;L^q)}^{1-\vartheta}\Big(\int_0^t \int_{\T^d} |u|^{q-2}|\nabla u|^2 \,\dd x\,\dd r \Big)^{\vartheta/q}
\end{align}
for all $u\in L^2_{\loc}([0,\infty), W^{1,q})\cap L^\infty_{\loc}([0,\infty);L^q)$ and $\vartheta=\frac{d}{d+2}$.
\end{lemma}
\begin{remark}
Note that, due to the assumed regularity of $u$, all the terms on the right-hand side of \eqref{eq:interpolation_subcriticality_apriori_estimates_0} are well-defined as
\begin{equation}
\label{eq:finiteness_of_norms_interpolation}
\int_0^t \int_{\T^d} |u|^{q-2}|\nabla u|^2 \,\dd x\,\dd r\leq \|u\|_{L^\infty(0,t;L^q)}^{q-2}\|\nabla u\|_{L^2(0,t;L^{q})}^2.
\end{equation}
Lemma \ref{l:nonlinear_estimate_easy} is an improvement of  \cite[Lemma 4.5]{A22} as the constants are uniform in $t>0$. The latter fact is crucial to obtain the results in the following section.
\end{remark}
\begin{proof}
The proof is similar to the one in \cite[Lemma 4.5]{A22}.
However, some care is needed as we want to obtain estimates that are uniform in $t>0$. The key idea is to interpolate the energy space and to use the scaling of the Lebesgue spaces. More precisely, 
\begin{align*}
&\|u\|_{L^{q_0}((0,t)\times \T^d)}= \big\||u|^{q/2}\big\|_{L^{\frac{2q_0}{q}}((0,t)\times \T^d)}^{2/q}\\
&\leq C_q \Big\||u|^{q/2}-\int_{\T^d} |u|^{q/2}\,\dd x \Big\|_{L^{\frac{2q_0}{q}}((0,t)\times \T^d)}^{2/q}
+C_q\|u\|_{L^{q_0}(0,t;L^{q/2})}\\
&\leq C_q \Big\||u|^{q/2}-\int_{\T^d} |u|^{q/2}\,\dd x \Big\|_{L^{\frac{2q_0}{q}}((0,t)\times \T^d)}^{2/q}
+C_q'\|u\|_{L^{q_0}(0,t;L^{2})}+ \frac{1}{2}
\|u\|_{L^{q_0}((0,t)\times \T^d)}.
\end{align*}
Let us emphasize that, if $q<4$, then one can use the trivial embedding $L^2(\T^d)\embed L^{q/2}(\T^d)$ instead of the interpolation argument in the last line.
Thus, it remains to interpolate the first term on the right-hand side of the above estimate. 
For the latter, recall that, by standard interpolation inequalities, for $\vartheta=\frac{d}{d+2}$ and recalling $\frac{2q_0}{q} = \frac{2(d+2)}{d}$,
\begin{align*}
L^{\infty}(0,t;\dot{L}^2(\T^d))\cap L^2(0,t;\dot{H}^1(\T^d))
\embed L^{2/\vartheta}(0,t;\dot{H}^{\vartheta}(\T^d))
\embed L^{2q_0/q}(0,t;L^{2q_0/q}(\T^d)).
\end{align*}
Thus, it follows from the above that
\begin{align*}
\Big\||u|^{q/2}-\int_{\T^d} |u|^{q/2}\,\dd x \Big\|_{L^{2q_0/q}((0,t)\times \T^d)}^{2/q}
&\lesssim\Big\||u|^{q/2}-\int_{\T^d} |u|^{q/2}\,\dd x\Big\|_{L^{\infty}(0,t;L^{2})}^{2(1-\vartheta)/q}
 \|\nabla [|u|^{q/2}]\|_{L^{2}(0,t;L^{2})}^{2\vartheta/q}\\
&\lesssim \|u\|_{L^{\infty}(0,t;L^{q})}^{1-\vartheta}
 \Big(\int_0^t \int_{\T^d} |u|^{q-2}|\nabla u|^2 \,\dd x \,\dd r \Big)^{\vartheta/q}.
\end{align*}
Thus, \eqref{eq:interpolation_subcriticality_apriori_estimates_0} follows by putting together the previous estimates.
\end{proof}

\begin{proof}[Proof of Lemma \ref{l:Lq_estimate_small}]
For simplicity, we split the proof into two steps.

\smallskip

\noindent \emph{Step 1: There exists a constant $C_q>0$ such that, a.s.\ for all $t\in (\tau_0,\tau)$,}
\begin{align*}
\sup_{r\in [\tau_0,t)}\|w(r)\|_{L^q}^q 
&+\sum_{i=1}^{\ell}\int_{\tau_0}^t \int_{\T^d} |w_i|^{q-2}|\nabla w_i|^2\,\dd x \,\dd r \\
&\leq C_q \|w(\tau_0)\|_{L^q}^q+ C_q \int_{\tau_0}^t \int_{\T^d} (|w|^2+|w|^{q_0})\,\dd x \,\dd r .
\end{align*}
From Lemma \ref{lem:Lebesgue_initial_data_continuity} and applying the It\^o formula in either \cite[Section 3]{Kry13} or \cite[Appendix A]{DHV16} (see also \cite[Theorem 4.1(2)]{A22} or \cite[Lemma 2]{DG15_boundedness} for details), it holds a.s.\ for all $t>0$ that 
\begin{align}
\label{eq:Itoformula_Lq_winfty}
\|w_i(t)\|_{L^q}^q
&+q(q-1) \nu_i \int_{\tau_0}^t \int_{\T^d} |w_i|^{q-2} |\nabla w_i|^2\,\dd x \, \dd r \\
\nonumber
&=\|w_i(\tau_0)\|_{L^q}^q
+ q\int_{\tau_0}^t  \int_{\T^d} |w_i|^{q-2} w_i f_i (w+v_\infty)\, \dd x \,\dd r .
\end{align}
As $f_i(v_\infty)=0$, it follows that
\begin{align*}
|f_i(w+v_\infty)- f_i(v_\infty)|\leq C (1+|w+v_\infty|^{h-1}+|v_\infty|^{h-1})|w|
\leq C (|w|+|w|^{h}),
\end{align*}
where $C$ depends only on $N$, $q$, and $K$ as $\sup_{\q v_\infty\in K} |v_\infty|<\infty$ by Lemma \ref{l:continuous_cinfty}. 
Using the above in \eqref{eq:Itoformula_Lq_winfty}, one readily obtains the claim of Step 1.

\smallskip

\noindent\emph{Step 2: There exists a constant $C_q>0$ such that, a.s.\ for all $t\in (\tau_0,\tau)$,}
\begin{align*}
\|w\|_{L^{q_0}(\tau_0,t;L^{q_0})}^q
\leq C_q \|w(\tau_0)\|_{L^q}^q+ \|w\|_{L^{q_0}(\tau_0,t;L^2)}^2+ C_q \int_{\tau_0}^t \int_{\T^d} (|w|^2+|w|^{q_0})\,\dd x \,\dd r .
\end{align*}
From Lemma \ref{l:nonlinear_estimate_easy} and Young’s inequality, together with the bounds from Step 1, we obtain
\begin{align*}
&\|w\|_{L^{q_0}(\tau_0,t;L^{q_0})}^q\\
&\leq C_q \|w(\tau_0)\|_{L^q}^q
+C_q \|w\|_{L^{q_0}(\tau_0,t;L^2)}^q + C_q \|w(\tau_0)\|_{L^q}^q+ C_q \int_{\tau_0}^t \int_{\T^d} (|w|^2+|w|^{q_0})\,\dd x \,\dd r .
\end{align*}
Now, the claimed estimate of Step 2 follows by noticing that 
\begin{align*}
\|w\|_{L^{q_0}(\tau_0,t;L^2)}^q
&\lesssim_{q,d} \|w\|_{L^{q_0}(\tau_0,t;L^2)}^2 + \|w\|_{L^{q_0}(\tau_0,t;L^2)}^{q_0}
\end{align*}
and the trivial embedding $L^{q_0}(\T^d)\embed L^2(\T^d)$ as $q\geq 2$.
The claimed bound in Lemma \ref{l:Lq_estimate_small} follows by summing the ones in Steps 1 and 2.
\end{proof}

\subsection{Proof of Theorem \ref{t:small_implies_global}}
\label{ss:proof_small_implies_global}

\begin{proof}[Proof of Theorem \ref{t:small_implies_global}]
We begin by collecting a useful observation. Let $(\tau_n)_n$ be the sequence of stopping times as in Definition \ref{def:solution}\ref{it:solution1}. In particular, $\tau_n\uparrow \tau$ a.s.\ Thus, if \eqref{eq:small_implies_global_assumption} holds, then the same holds with $\tau_0$ replaced by $\tau_0\wedge \tau_n$ with $n$ sufficiently large. 
Therefore, without loss of generality, we may assume $\tau_0\in (0,\tau)$ a.s.\ and hence 
$$
\P(\tau_0<\tau,\, \|v(\tau_0)-v_\infty\|_{L^q}\leq\eta)=
\P( \|v(\tau_0)-v_\infty\|_{L^q}\leq\eta)>1- \varepsilon.
$$
We now divide the proof into several steps.

\smallskip

\noindent\emph{Step 1: There exists a constant $C_q>0$ such that, a.s.\ for all $t\in (\tau_0,\tau)$,}
\begin{align*}
\|w\|_{L^{q_0}(\tau_0,t;L^{q_0})}^q
\leq C_q \|w(\tau_0)\|_{L^q}^2(1+\|w(\tau_0)\|_{L^q}^{q-2})+ \|w\|_{L^{q_0}(\tau_0,t;L^{q_0})}^{q_0}.
\end{align*}
The idea is to sum Lemmas \ref{l:L2_estimates_small_implies} and \ref{l:Lq_estimate_small} with appropriate constants. 
To begin, recall the following standard interpolation inequality,
$$
\|f\|_{L^{q_0}(0,t;L^2)}
\lesssim
\|f\|_{L^{\infty}(0,t;L^2)}+ \|f\|_{L^{2}((0,t)\times \T^d)},
$$
where the implicit constant is independent of $t>0$. The above and Lemma \ref{l:L2_estimates_small_implies} imply
\begin{align}
\label{eq:step_1_proof_implies_global_L21}
\|w\|_{L^{q_0}(\tau_0,t;L^2)}^2
\leq C_1\|w(\tau_0)\|_{L^2}^2 +C_1\|w\|_{L^{q_0}((\tau_0,t)\times\T^d)}^{q_0}.
\end{align}
Next, we recall the following consequence of Lemma \ref{l:Lq_estimate_small}, 
\begin{equation}
\label{eq:step_1_proof_implies_global_L22}
\begin{aligned}
\|w\|_{L^{q_0}(\tau_0,t;L^{q_0})}^q
&\leq C_1 (\|w(\tau_0)\|_{L^2}^2 + \|w(\tau_0)\|_{L^q}^q)+ C_1\|w\|_{L^{q_0}(\tau_0,t;L^2)}^2\\
&\quad + C_1\|w\|_{L^2((\tau_0,t)\times\T^d)}^2 + C_1 \|w\|_{L^{q_0}(\tau_0,t;L^{q_0})}^{q_0}.
\end{aligned}
\end{equation}
Now, summing \eqref{eq:step_1_proof_implies_global_L21} with the inequality \eqref{eq:step_1_proof_implies_global_L22}, where the latter is multiplied by $1/(2C_1)$, then using Lemma \ref{l:L2_estimates_small_implies} to estimate the term $\|w\|_{L^2((\tau_0,t)\times\T^d)}^2$, we have
\begin{align*}
&\tfrac{1}{2C_1}\|w\|_{L^{q_0}(\tau_0,t;L^{q_0})}^q +
\|w\|_{L^{q_0}(\tau_0,t;L^2)}^2\\
&\leq C_0\|w(\tau_0)\|_{L^2}^2+\tfrac{1}{2} \|w(\tau_0)\|_{L^q}^q+ \tfrac{1}{2}\|w\|_{L^{q_0}(\tau_0,t;L^2)}^2+\big(C_0+\tfrac{1}{2}\big) \|w\|_{L^{q_0}(\tau_0,t;L^{q_0})}^{q_0}.
\end{align*}
This, together with $\|w(\tau_0)\|_{L^2} \lesssim \|w(\tau_0)\|_{L^q}$, proves the claim of Step 1.

\smallskip

\noindent\emph{Step 2: There exists a constant $\eta\in (0,1)$ such that, a.s.\ on $\{\|w(\tau_0)\|_{L^q}\leq \eta\}$, it holds that}
\begin{equation}
\label{eq:bound_tau0mu_w}
\|w\|_{L^{q_0}(\tau_0,\tau;L^{q_0})}^q\leq R \|w(\tau_0)\|_{L^q}^2(1+\|w(\tau_0)\|_{L^q}^{q-2}),
\end{equation}
\emph{where the constant $R$ depends only on $q$, $d$, and $K$.}
Let  
\begin{gather}
\label{eq:def_Oeta}
\O_\eta \coleq \{\|w(\tau_0)\|_{L^q}\leq \eta\}, \\ 
\label{eq:def_XtE0}
\mathscr{X}_t \coleq \|w\|_{L^{q_0}(\tau_0,t;L^{q_0})}^q \quad \text{ and } \quad
\mathscr{E}_0 \coleq \|w(\tau_0)\|_{L^q}^2(1+\|w(\tau_0)\|_{L^q}^{q-2}).
\end{gather}
Set $\gamma=q_0/q>1$. Then, Step 1 can be reformulated as
\begin{align}
\label{eq:mathscr_X_star_bound}
\mathscr{X}_t\leq 
 C \mathscr{E}_0+ C\, \mathscr{X}_t^\gamma \qquad \Longleftrightarrow \qquad \psi_C (\mathscr{X}_t)\leq  \mathscr{E}_0,
\end{align}
where $\psi_C (x)=C^{-1}x-x^\gamma$ for $x\geq 0$, see \cite[Section 4]{A22} for a similar situation. Clearly, 
$\psi_C$ attains a unique maximum at the point $x_*=(\gamma C)^{-1/(\gamma-1)}$ with value $\psi_*= \frac{x_*}{C}(1-\frac{1}{\gamma})$. Now, as $\mathscr{X}_0=0$, it follows from the above estimate and a standard connectedness argument that 
$$
\mathscr{E}_0 \leq \psi_* \qquad \Longrightarrow \qquad\sup_{\tau_0<t<\tau} \mathscr{X}_t \leq x_* \text{ a.s.\ on }\O_\eta.
$$
Finally, to prove the estimate, it is enough to note that, if $x\in [0,x_*]$,
$$
\psi_C(x)\geq \frac{x}{C}-x^{\gamma}\geq \frac{x}{C}-x (x_*)^{\gamma-1}= \frac{x}{C}\Big(1-\frac{1}{\gamma}\Big)\text{ a.s.\ on }\O_\eta.
$$
Thus, as $\sup_{\tau_0<t<\tau} \mathscr{X}_t \leq x_*$ a.s.\ on $\O_\eta$ and \eqref{eq:mathscr_X_star_bound}, it follows that  
$$
\mathscr{X}_t\leq C \Big(\frac{\gamma}{\gamma-1} \Big)\,\psi_C(\mathscr{X}_t)\leq  C \Big(\frac{\gamma}{\gamma-1} \Big)\, \mathscr{E}_0.
$$
From Fatou's lemma
$$
\sup_{\tau_0<t<\tau}\mathscr{X}_t =\|w\|^q_{L^{q_0}(\tau_0,\tau;L^{q_0})} \text{ a.s.\ on }\O_\eta,
$$
as desired.
 
\smallskip

\noindent\emph{Step 3: Conclusion}. 
Let $\eta$ and $\O_\eta$ be as in Step 2. 
Combining \eqref{eq:bound_tau0mu_w} with the bounds in Lemmas \ref{l:L2_estimates_small_implies} and \ref{l:Lq_estimate_small}, it follows that 
\begin{equation}
\label{eq:apriori_estimate_mu}
\sup_{r\in [\tau_0,\tau)}\|w(r)\|_{L^q}^q 
+\max_{1\leq i\leq \ell}\int_{\tau_0}^\tau \int_{\T^d} |w_i|^{q-2}|\nabla w_i|^2\,\dd x \,\dd r+\int_{\tau_0}^\tau \int_{\T^d} |w|^2\,\dd x \,\dd s  \leq C \mathscr{E}_0,
\end{equation}
a.s.\ on $\O_\eta$, 
where $\mathscr{E}_0$ is as in \eqref{eq:def_XtE0}. Clearly, it suffices to consider the case $\mathscr{E}_0>0$. In the latter situation, to conclude, it remains to prove that 
\begin{equation}
\label{eq:tau_infty_Oeta}
\tau= \infty \ \text{ a.s.\ on }\{\|w(\tau_0)\|_{L^q}\leq \eta\}.
\end{equation}
Indeed, if the above holds, then we can let 
$$
\tau_1 \coleq \inf \Big\{t\in [\tau_0,\tau)\,:\, \sup_{r\in [\tau_0,t)}\|w(r)\|_{L^q}^q 
+\max_{1\leq i\leq \ell}\int_{\tau_0}^t \int_{\T^d} |w_i|^{q-2}|\nabla w_i|^2\,\dd x \,\dd r\geq 2C \mathscr{E}_0\Big\}
$$
with $\inf\emptyset \coleq \tau$ to obtain the claim of Theorem \ref{t:small_implies_global} as $\tau_1=\tau$ a.e.\ on $\O_\eta$ by construction.

The claim \eqref{eq:tau_infty_Oeta} follows by partially repeating the argument in \cite[Theorem 3.2]{AV24_dissipative} by means of the blow-up criteria in \cite[Corollary 2.11]{AV23_RDE_local}. We give some details in the subcritical case $q>\frac{d(h-1)}{2}$. The critical case follows similarly (see Step 2 in \cite[Theorem 3.2]{AV24_dissipative}).
In the case $q>\frac{d(h-1)}{2}$, from \cite[Corollary 2.11(1)]{AV23_RDE_local} and Lemma \ref{lem:Lebesgue_initial_data_continuity}, for all stopping times $\tau_0\in [0,\tau)$ a.s., the following blow-up criterion holds:
$$
\P\Big(\tau<\infty,\, \sup_{t\in [\tau_0,\tau)} \|v(t)\|_{L^q}<\infty\Big)=0.
$$
Thus, from the above, it follows that 
\begin{align*}
\P(\{\tau<\infty\}\cap \O_\eta)
&\stackrel{\eqref{eq:apriori_estimate_mu}}{=} \P\Big(\Big\{\tau<\infty,\,\sup_{t\in [\tau_0,\tau)}\|v(t)\|_{L^q}<\infty\Big\} \cap \O_\eta \Big)
\\
&\leq 
\P\Big(\tau<\infty,\,\sup_{t\in [\tau_0,\tau)}\|v(t)\|_{L^q}<\infty \Big)
=0.
\end{align*}
This proves \eqref{eq:tau_infty_Oeta}.
\end{proof}

\section{Scaling limits for RDEs with cutoff -- Proof of Theorem \ref{t:global}}
\label{s:scaling_limit}
The ultimate aim of this section is to prove Theorem \ref{t:global}. In order to do that, we will combine the scaling limit results proven in \cite{A22} with Theorem \ref{t:small_implies_global} and the entropy--entropy dissipation relation of Theorem \ref{t:entropy_dissipation} to prove the global well-posedness with high probability of solutions to \eqref{eq:reaction_diffusion} as stated in Theorem \ref{t:global}. More precisely, we start in Subsection \ref{ss:cutoff1} with stochastic RDEs with cutoff, then in Subsection \ref{ss:cutoff2}, we recall and simplify the approach in \cite[Section 6]{A22} for the scaling limit of \eqref{eq:reaction_diffusion_cut_off}, where the noise coefficients $\theta\in \ell^2$ vanish in a suitable sense but keeping constant intensity, see Theorem \ref{t:scaling_limit_cutoff} below. In Subsection \ref{ss:global_with_high_diffusivity}, we prove global well-posedness and suitable bounds for the deterministic RDEs with high diffusivity \eqref{eq:reaction_diffusion_deterministic_enhanced_diffusivity}. Finally, in Subsection \ref{ss:proof_global_well_posedness}, we provide the proof of Theorem \ref{t:global} by combining these results.

Before going into the details, it is worth pointing out that the entropy--entropy dissipation relation as in Theorem \ref{t:entropy_dissipation} is only used to ensure the exponential decay in \eqref{eq:global_high_viscosity3} for solutions of \emph{deterministic} RDEs with high diffusivity in Subsection \ref{ss:global_with_high_diffusivity}, see also Remark \ref{rem:boundary_equilibria}.

\subsection{Stochastic RDEs with cutoff}\label{ss:cutoff1}
As is typical in scaling limits for SPDEs with transport noise, part of this subsection is devoted to the study of stochastic RDEs \eqref{eq:reaction_diffusion} with cutoff. Let $\phi\in C^{\infty}(\mathbb R)$ such that $\phi|_{[0,1]}=1$, $\phi|_{[2,\infty)}=0$. Then, for $q,r\in (1,\infty)$ and $R>0$, we define the cutoff
\begin{equation}
\label{eq:def_cutoff}
\phi_{R,r}(t,v) \coleq \phi\big(R^{-1}\|v\|_{L^r(0,t;L^q)}\big).
\end{equation}
With this cutoff, we consider the following stochastic RDEs
\begin{equation}
\label{eq:reaction_diffusion_cut_off}
\left\{
\begin{aligned}
\dd \vcuti - \nu_i \Delta \vcuti \,\dd t
& = \phi_{R,r}(\vcut) f_i(\vcut)\,\dd t \\
&+\sqrt{c_d \nu }\sum_{k,\alpha} \theta_k (\sigma_{k,\alpha}\cdot\nabla) \vcuti\circ \dd W^{k,\alpha}_t
& \text{ on }&\Tor^d,\\ 
\vcuti(0)&=v_{0,i}  &\text{ on }&\Tor^d.
\end{aligned}\right.
\end{equation}

The main results of this section on stochastic RDEs with cutoff \eqref{eq:reaction_diffusion_cut_off} are as follows. The definition of $(p,\a,\s,q)$-solutions to \eqref{eq:reaction_diffusion_cut_off} is almost identical to that in Definition \ref{def:solution}.

\begin{proposition}[Global existence and uniqueness of stochastic RDEs with cutoff]
\label{prop:global_cut_off}
Let \eqref{reaction-full} be a complex balanced chemical reaction network, and consider the reaction--diffusion system \eqref{b-1}--\eqref{b-11}. Let $M\in \q \R^\ell_{> 0}$ and $v_\infty$ be the positive complex balanced equilibrium, and fix a normalized radially symmetric coefficient $\theta$ (see \eqref{eq:theta_normalized_symmetric}).
Suppose that $q,p\in (2,\infty)$, $\a\in (0,\frac{p}{2}-1)$, and $\delta\in [1,2)$ satisfy
\begin{equation}
\label{eq:global_cut_off0}
q>\frac{d(h-1)}{2}\vee \frac{d}{d-\delta}
\qquad \text{ and }\qquad
\frac{1+\a}{p}\leq 1-\frac{\delta}{2}.
\end{equation}
Then, there exists some $r_0\in (1,\infty)$ depending only on $q$, $\delta$, $h$, and $d$, for which the following assertion holds provided $r\geq r_0$. For all  
$$
v_0\in B^{2-\delta-2(1+\a)/p}_{q,p}(\T^d;\R^{\ell}_{\geq 0}),
$$
there exists a unique global $(p,\a,\delta,q)$-solution $\vcut$ to \eqref{eq:reaction_diffusion_cut_off} and
\begin{align}
\label{eq:global_cut_off1}
\vcut&\in  L^p_{\loc}([0,\infty),w_{\a};H^{2-\delta,q}(\T^d;\R^\ell)) \cap C([0,\infty);B^{2-\delta-2(1+\a)/p}_{q,p}(\T^d;\R^\ell)), \\
\vcut&\in   C((0,\infty);B^{2-\delta-2/p}_{q,p}(\T^d;\R^\ell))\text{ a.s., }\\
\label{eq:positivity_statement_cutoff}
\vcut(t)&\geq 0 \text{ componentwise a.e.\ on }\T^d \text{ for all }t>0.
\end{align}
Finally, the global solution $\vcut$ depends continuously on the initial data with the same mass vector $\q \overline{ v_0}$ (hence, the same equilibrium $v_\infty$). More precisely, if $(v_0^{(n)})_n\subseteq B^{2-\s-2(1+\a)/p}_{q,p}(\T^d;\R^\ell_{\geq 0})$ satisfies $\q\overline{v_0^n }=\q\overline{v_0}$ for all $n\geq 1$ and $\lim_{n\to\infty}\|v_0^{(n)} - v_0\|_{B^{2-\delta-2(1+\a)/p}_{q,p}(\T^d;\R^{\ell}_{\geq 0})} = 0$, then for all $r\in (1,\infty)$, $T<\infty$, and $\g>0$,
\begin{align}
\label{eq:global_cut_off_continuity1}
\lim_{n\to \infty}\P\Big(\sup_{t\in [0,T]} \|\vcut^n(t)-\vcut(t)\|_{B^{0}_{q,p}}\geq \g \Big)&= 0 ,\\
\label{eq:global_cut_off_continuity2}
\lim_{n\to \infty}\P\Big(\|\vcut^n-\vcut\|_{L^{r}(0,T;L^q)}\geq \g \Big)
&= 0,
\end{align}
where $\vcut^n$ is the global $(p,\a,\delta,q)$-solution to \eqref{eq:reaction_diffusion_cut_off} with initial data $v_0^{(n)}$.
\end{proposition}

The proof of Proposition \ref{prop:global_cut_off} is given at the end of this subsection.
As in Proposition \ref{prop:LWP}, the conditions \eqref{eq:global_cut_off0} imply that $B^{2-\delta-2(1+\a)/p}_{q,p}$ is subcritical for stochastic RDEs with cutoff \eqref{eq:reaction_diffusion_cut_off}, and that the initial data have non-negative smoothness. In particular, the usual choice $\a_{p,\delta}=p(1-\frac{\delta}{2})-1$ satisfies the latter condition. Interestingly, as we are assuming $\a>0$, we have the following elementary embedding at our disposal $B^{1-2/p}_{q,p}\embed L^q$ and therefore \eqref{eq:global_cut_off1} implies 
\begin{equation}
\label{eq:continuity_in_Lq_cutoff}
\vcut\in C((0,\infty);L^q(\T^d;\R^\ell)) \text{ a.s. }
\end{equation}
In particular,  the condition \eqref{eq:positivity_statement_cutoff} is well-defined.

\smallskip

The proof of Proposition \ref{prop:global_cut_off} follows similarly to that of \cite[Section 4]{A22} by employing the stochastic maximal $L^p$-regularity estimates in \cite{AV21_SMR_torus}. 
For the readers' convenience, we provide some details here. 
We begin with the following nonlinear estimate, from which we will derive the main a priori estimates for \eqref{eq:reaction_diffusion_cut_off}. Here, the subcriticality assumption $q>\frac{d(h-1)}{2}\vee \frac{d}{d-\delta}$ is central (see the discussion below Proposition \ref{prop:LWP}).

\begin{lemma}[Nonlinear estimates and subcriticality]
\label{l:subcriticality_Lq}
Fix $q>\frac{d(h-1)}{2}\vee \frac{d}{d-\delta}$. Then, there exist $\varepsilon\in (0,1)$ and $C>0$ such that, for all $i\in \{1,\dots,\ell\}$ and $v\in H^{2-\delta,q}(\T^d;\R^\ell)$ satisfying $v\geq 0$ componentwise a.e.\ on $\T^d$, 
\begin{equation}
\label{eq:subcriticality_Lq1}
\|f_i(v)\|_{H^{-\delta,q}(\T^d)}\leq C( 1+\|v\|_{L^{q}(\T^d;\R^\ell)}^{h-1+\varepsilon}\|v\|_{H^{2-\delta,q}(\T^d;\R^\ell)}^{1-\varepsilon}).
\end{equation}
\end{lemma}

The subcriticality of the space $L^q$ with $q$ as above is encoded in the sublinear growth of the right-hand side of \eqref{eq:subcriticality_Lq1} with respect to the $H^{2-\delta,q}$-norm.

\begin{proof}
Fix $i\in \{1,\dots,\ell\}$.
Since $q>\frac{d}{d-\delta}$ by assumption, there exists $\eta\in (1,\infty)$ such that $\frac{d}{\eta}=\delta+\frac{d}{q}$. Hence, by the Sobolev embedding $L^\eta(\T^d)\embed H^{-\delta,q}(\T^d)$,
\begin{equation}
\label{eq:estimate_nonlinearity_subcritical}
\|f_i(v)\|_{H^{-\delta,q}(\T^d)}
\lesssim \|f_i(v)\|_{L^\eta(\T^d)}
\lesssim 1+\|v\|_{L^{h\eta}(\T^d;\R^\ell)}^{h}.
\end{equation}
Note that, again by Sobolev embeddings, 
$$
H^{\rho,q}(\T^d)\embed L^{h\eta}(\T^d) \quad \text{ if } \quad \rho-\frac{d}{q}\geq -\frac{d}{h\eta} .  
$$
The previous is satisfied by letting 
$$
\rho\coloneq \max\Big\{-\frac{\delta}{h}+\frac{d}{q}\Big(1-\frac{1}{h}\Big),\, 0\Big\}.
$$
This, together with the interpolation inequality $\|u\|_{H^{\rho,q}(\T^d)}\lesssim \|u\|_{L^q(\T^d)}^{(2-\delta-\rho)/(2-\delta)}\|u\|_{H^{2-\delta,q}(\T^d)}^{\rho/(2-\delta)}$, yields
$$
\|f_i(v)\|_{H^{-\delta,q}(\T^d)}
\lesssim 1+ 
\|v\|_{L^q(\T^d;\R^\ell)}^{[h(2-\delta-\rho)/(2-\delta)]}\|v\|_{H^{2-\delta,q}(\T^d;\R^\ell)}^{(h\rho)/(2-\delta)}.
$$
The estimate \eqref{eq:subcriticality_Lq1} follows by noticing that $h\rho< 2-\delta$ if and only if $q>\frac{d(h-1)}{2}$.
\end{proof}

With the above result at our disposal, we are able to prove Proposition \ref{prop:global_cut_off}.

\begin{proof}[Proof of Proposition \ref{prop:global_cut_off}]
For brevity, we divide the proof into several steps.

\smallskip

\emph{Step 1: (Local well-posedness with cutoff). For each $v_0\in B^{2-\s-2(1+\a)/p}_{q,p}$, there exists a $(p,\a,\s,q)$-solution $(\vcut,\tau)$ to \eqref{eq:reaction_diffusion_cut_off}. Moreover, the following blow-up criterion holds: For all $T<\infty$,}
\begin{equation}
\label{eq:blow_up_criteria_cutoff}
\P\Big(\tau<T,\, \max_{1\leq i\leq \ell}\int_0^\tau (\phi_{R,r}(\vcut))^p \| f_i (\vcut)\|_{H^{-\s,q}}^p\, w_{\a}\,\dd t 
<\infty\Big)=0.
\end{equation}
The proof of Step 1 follows as in \cite[Proposition 4.2]{A22} by using Proposition \ref{prop:LWP}, a stopping time argument, and \cite[Theorem 4.10(1)]{AV19_QSEE_2}.

\smallskip

\emph{Step 2: (A-priori estimates and global well-posedness). It holds that $\tau=\infty$ a.s.\ Moreover, for all $T<\infty$, there exists $C_0>0$ independent of $v_0$ such that}
\begin{equation}
\label{eq:apriori_estiamte_cutoff_global}
\E\sup_{t\in [0,T]}\|\vcut(t)\|_{B^{2-\s-2(1+\a)/p}_{q,p}}^p+ 
\E \int_0^T \|\vcut\|_{H^{2-\s,q}}^p\,w_\a \,\dd t \leq C_0^p (1+\|v_0\|_{B^{2-\s-2(1+\a)/p}_{q,p}}^p).
\end{equation} 
Fix $T<\infty$. Let $\tau_n$ be the stopping time given by
$$
\tau_n \coleq \inf\Big\{t\in [0,\tau\wedge T)\,:\, \|\vcut(t)\|_{B^{2-\s-2(1+\a)/p}_{q,p}}^p+ \int_0^t \|\vcut\|_{H^{2-\s,q}}^p\,w_\a \,\dd s 
\geq n\Big\}
$$
with $\inf\emptyset \coleq \tau\wedge T$. As $\#\{\theta_k \neq 0\}<\infty$, we can apply the stochastic maximal $L^p$-regularity estimates in \cite[Theorem 5.2]{AV21_SMR_torus}. From the latter, there exists a constant $C>0$ independent of $v_0$ and $n\geq 1$ such that 
\begin{align}
\label{eq:smr_applied_to_vcut}
&\E\sup_{t\in [0,\tau_n]}\|\vcut(t)\|_{B^{2-\s-2(1+\a)/p}_{q,p}}^p+ 
\E \int_0^{\tau_n} \|\vcut\|_{H^{2-\s,q}}^p \,w_\a \,\dd t \\
\nonumber
&\qquad \leq C^p \|v_0\|_{B^{2-\s-2(1+\a)/p}_{q,p}}^p
+C^p\max_{1\leq i\leq \ell} \E \int_0^{\tau_n} (\phi_{R,r}(\vcut))^p\|f_i(\vcut)\|_{H^{-\s,q}}^p\,w_\a \,\dd t.
\end{align}
Let $\varepsilon\in (0,1)$ be as in Lemma \ref{l:subcriticality_Lq}. Suppose that $r\geq [(h+1-\varepsilon)p]/\varepsilon$ and let
$$
\mu_{R,r} \coleq \inf\big\{t\in [0,\tau\wedge T)\,:\, \|\vcut\|_{L^r(0,t;L^q)}\geq 2R\big\} \quad \text{ with }\quad \inf\emptyset \coleq \tau\wedge T,
$$
and $\mu_{R,r}^n =\tau_n\wedge \mu_{R,r}$. 
Now, from Lemma \ref{l:subcriticality_Lq} and $\mathrm{supp}\,\phi\subseteq [0,2]$, for all $n\geq 1$, we have
\begin{align}
\label{eq:cutoff_nonlinearity_bound_f}
&\max_{1\leq i\leq \ell} \, \E \int_0^{\tau_n} (\phi_{R,r}(\vcut))^p\|f_i(\vcut)\|_{H^{-\s,q}}^p\,w_\a \,\dd t\\
&\qquad\qquad
\leq C^p T+ C\int_0^{\mu_{R,r}^n} \|\vcut\|_{L^{q}}^{(h-1+\varepsilon)p}\|\vcut\|_{H^{2-\delta,q}}^{(1-\varepsilon)p}\,w_\a \,\dd t \\
\nonumber
&\qquad\qquad
\leq C^pT+C \|\vcut\|_{L^r(0,\mu_{R,r}^n,w_\a;L^q)}^{(h-1+\varepsilon)p}
\|\vcut\|_{L^p(0,\mu_{R,r}^n,w_\a;H^{2-\s,q})}^{(1-\varepsilon)p}\\
\nonumber
&\qquad\qquad\stackrel{(i)}{\leq} C^pT+ C_T(2R)^{(h-1+\varepsilon)p}\|\vcut\|_{L^p(0,\tau_n,w_\a;H^{2-\s,q})}^{(1-\varepsilon)p}\\
\nonumber
&\qquad\qquad\stackrel{(ii)}{\leq} C_{R,T,p} +\frac{1}{2} \|\vcut\|_{L^p(0,\tau_n,w_\a;H^{2-\s,q})}^{p},
\end{align}
where in $(i)$ we used the H\"older inequality and $w_\a(t)\leq T^\a$ for $t\in (0,T)$, and in $(ii)$ the Young inequality.
Combining the above with \eqref{eq:smr_applied_to_vcut}, by Fatou's lemma, we obtain 
\begin{equation}
\label{eq:apriori_estiamte_cutoff_global_tau}
\E\sup_{t\in [0,\tau\wedge T)}\|\vcut(t)\|_{B^{2-\s-2(1+\a)/p}_{q,p}}^p+ 
\E \int_0^{\tau\wedge T} \|\vcut\|_{H^{2-\s,q}}^p \,w_\a\,\dd t \leq  C_{R,T,p} (1+ \|v_0\|_{B^{2-\s-2(1+\a)/p}_{q,p}}^p).
\end{equation}
By letting $n\to \infty$ in \eqref{eq:cutoff_nonlinearity_bound_f} with $\tau_n$, we have 
$$
\max_{1\leq i\leq \ell} \, \E \int_0^{\tau\wedge T} (\phi_{R,r}(\vcut))^p\|f_i(\vcut)\|_{H^{-\s,q}}^p\,w_\a \,\dd t
\leq C_1( 1+ \|v_0\|_{B^{2-\s-2(1+\a)/p}_{q,p}}^p).
$$ 
From the above, the arbitrariness of $T<\infty$, and \eqref{eq:blow_up_criteria_cutoff} yield $\tau=\infty$. Hence, the a priori estimate \eqref{eq:apriori_estiamte_cutoff_global} follows from the former fact and \eqref{eq:apriori_estiamte_cutoff_global_tau}.

\smallskip

To conclude the proof of Proposition \ref{prop:global_cut_off}, it remains to show \eqref{eq:global_cut_off_continuity1}--\eqref{eq:global_cut_off_continuity2}. 
%%%
As usual, it is based on a Lipschitz-type estimate for solutions to \eqref{eq:reaction_diffusion_cut_off}. However, due to the nonlocality of the nonlinearity in \eqref{eq:reaction_diffusion_cut_off} and the $L^p(L^q)$-type spaces used in the convergence result \eqref{eq:global_cut_off_continuity1}--\eqref{eq:global_cut_off_continuity2}, we cannot apply the known stochastic Gr\"onwall's lemma. 
Instead, we prove the main estimate exploiting the \emph{subcriticality} of the setting $(p,\a,s,q)$ in \eqref{eq:global_cut_off0}.
For notational convenience, we let 
\begin{equation}
\mreg(t) \coleq 
C([0,t];B^{2-\delta-2(1+\a)/p}_{q,p})\cap L^p(0,t,w_{\a};H^{2-\s,q}),
\end{equation}
where $t\in [0,T]$ and $T<\infty$ are fixed.
Next, we first prove a Lipschitz-type estimate for the nonlinearities. The proof of \eqref{eq:global_cut_off_continuity1}--\eqref{eq:global_cut_off_continuity2} is postponed to Step 4.

\smallskip

\emph{Step 3: (Lipschitz estimate). There exists $\varepsilon_0\in (0,1)$ such that, for all $t>0$ and $v,v'\in \mreg(t)$, it holds that}
\begin{align*}
&\| \phi_{R,r}(v) f_i(v)- \phi_{R,r}(v') f_i(v')\|_{L^p(0,t,w_{\a};H^{-\s,q})}\\
&\qquad \lesssim (1+\|v\|_{\mreg(t)}^{h-1}+\|v'\|^{h-1}_{\mreg(t)})\|v-v'\|_{L^{p}(0,t;B_{q,p}^{2-\s-2(1+\a)/p})}^{1-\varepsilon_0}\|v-v'\|_{\mreg(t)}^{\varepsilon_0},
\end{align*}
\emph{where the implicit constant depends only on $T$.}
We begin by writing 
$$
\phi_{R,r}(v)f_i(v)-
\phi_{R,r}(v')f_i(v') =(
\phi_{R,r}(v)-
\phi_{R,r}(v'))f_i(v)+ 
\phi_{R,r}(v')(f_i(v)-f_i(v')).
$$
From  Lemma \ref{l:subcriticality_Lq}, it readily follows that 
\begin{align*}
\|\phi_{R,r}(v)-
\phi_{R,r}(v'))f_i(v)\|_{L^p(0,t,w_{\a};H^{-\s,q})}
&\leq C \sup_{(0, t)}|\phi_{R,r}(v)-
\phi_{R,r}(v'))| \|f_i(v)\|_{L^p(0,t,w_{\a};H^{-\s,q})} \\
&\leq C \|v-v'\|_{L^r(0,t;L^q)}
(1+\|v\|_{\mathrm{MR}_{p,\a}^{\s,q}(t)}^{h}).
\end{align*}
Let $\eta$ be as in \eqref{eq:estimate_nonlinearity_subcritical}. 
Moreover, as in \eqref{eq:estimate_nonlinearity_subcritical} and the embedding below it, we have
\begin{align*}
&\|\phi_{R,r}(v')(f_i(v)-f_i(v'))\|_{L^p(0,t,w_\a;H^{-\delta,q})}\\
&\qquad
\lesssim \big\|(1+\|v\|_{L^{h\eta}}^{h-1}+\|v'\|_{L^{h\eta}}^{h-1})\|v-v'\|_{L^{h\eta}}\big\|_{L^p(0,t,w_\a)}\\
&\qquad\lesssim (1+\|v\|_{L^{ph}(0,t,w_\a;L^{h\eta})}^{h-1}+\|v'\|_{L^{ph}(0,t,w_\a;L^{h\eta})}^{h-1})\|v-v'\|_{L^{ph}(0,t,w_\a;L^{h\eta})}\\
&\qquad\lesssim (1+\|v\|_{\mreg(t)}^{h-1}+\|v'\|_{\mreg(t)}^{h-1})\|v-v'\|_{L^r(0,t;L^q)}^{1-\varepsilon}\|v-v'\|_{\mreg(t)}^{\varepsilon},
\end{align*}
where $\delta$, $\varepsilon$ and $r$ are as in Lemma \ref{l:subcriticality_Lq}, \eqref{eq:estimate_nonlinearity_subcritical} and below \eqref{eq:smr_applied_to_vcut}, respectively. 
To conclude Step 3, it suffices to note that, as $2-\s-2(1+\a)/p\geq 0$ by \eqref{eq:global_cut_off0}, there exists $\varepsilon_1\in (0,1)$ such that, for all $t\in [0,T]$ and $v\in \mreg(t)$,
\begin{equation}
\label{eq:interpolation_spaces_Lq}
\|v\|_{L^r(0,t;L^q)}\lesssim_T \|v\|_{L^p(0,t;B^{2-\s-2(1+\a)/p}_{q,p})}^{1-\varepsilon_1}\|v\|_{\mreg(t)}^{\varepsilon_1},
\end{equation}
as it readily follows by interpolation (see \cite[Chapter 5]{BeLo}).

\smallskip

\emph{Step 4: Proof of \eqref{eq:global_cut_off_continuity1}--\eqref{eq:global_cut_off_continuity2}.}
Fix $R>1$ and $n\geq 1$. Let $(v_0^{(n)})_n$ be a sequence in $B^{2-\s-2(1+\a)/p}_{q,p}$ such that $v_0^{(n)} \to v_0$ in $B^{2-\s-2(1+\a)/p}_{q,p}$. Let $\vcut^{(n)}$ be the $(p,\a,s,q)$-solution to \eqref{eq:reaction_diffusion_cut_off} with initial data $v_0^{(n)}$. Define
$
\O_{R}^n =\big\{\max\{\|\vcut\|_{\mathrm{MR}_{p,\a}^{\s,q}(T)}\,,\,\|\vcut^{(n)}\|_{\mathrm{MR}_{p,\a}^{\s,q}(T)}\}\leq R\big\}
$
and 
$$
\chi_{R}^n \coleq \inf \big\{t\in [0,T]\,:\, \max\{\|\vcut\|_{\mathrm{MR}_{p,\a}^{\s,q}(t)}\,,\,\|\vcut^n\|_{\mathrm{MR}_{p,\a}^{\s,q}(t)}\}\geq R\big\}\quad \text{ with }\quad \inf\emptyset \coleq T.
$$
Clearly, $\{\chi_R^n < T\}\subseteq \O\setminus\O_R^n$. Define $w_0^{(n)} \coleq v_0^{(n)}-v_0$ and $\wcut^{(n)} \coleq \vcut-\vcut^{(n)}$. 
Combining the stochastic maximal regularity in \cite[Theorem 5.2]{AV21_SMR_torus} and the estimate of Step 3, for all $t\in [0,T]$,
\begin{align*}
\E\|\wcut^{(n)}\|_{\mreg(\chi_{R}^n\wedge t)}^p
\leq C  \|w_0^{(n)}\|_{B^{2-\s-2(1+\a)/p}_{q,p}}^p
+ C_R\E \big[ \|\wcut^{(n)}\|_{L^{p}(0,\chi_{R}^n\wedge t;B_{q,p}^{2-\s-2(1+\a)/p})}^{p(1-\varepsilon_0)}
\|\wcut^{(n)}\|_{\mreg(\chi_{R}^n\wedge t)}^{p\varepsilon_0}\big]&\\
\leq C  \|w_0^{(n)}\|_{B^{2-\s-2(1+\a)/p}_{q,p}}^p
+\frac{1}{2}\E\|\wcut^{(n)}\|_{\mreg(\chi_{R}^n\wedge t)}^p+  
C_R\E \|\wcut^{(n)}\|_{L^{p}(0,\chi_{R}^n\wedge t;B_{q,p}^{2-\s-2(1+\a)/p})}^{p}&.
\end{align*}
The (deterministic) Gr\"onwall's lemma applied to $X_t=\sup_{s\leq \chi_{R}^n \wedge t }\E\|\wcut^{(n)}(s)\|_{B_{q,p}^{2-\s-2(1+\a)/p}}^p$ ensures that 
$
X_T\leq C_R  \|w_0^{(n)}\|_{B^{2-\s-2(1+\a)/p}_{q,p}}^p
$ 
and, therefore,
\begin{equation}
\label{eq:estimate_wn_on_chiRn}
\E\|\wcut^{(n)}\|_{\mreg(\chi_{R}^n)}^p
\leq C_R  \|w_0^{(n)}\|_{B^{2-\s-2(1+\a)/p}_{q,p}}^p.
\end{equation}
For all $\varepsilon>0$, the above implies
\begin{align*}
\P(\|\wcut^{(n)}\|_{\mreg(T)}\geq \varepsilon)
&
\leq \P\big(\|\wcut^{(n)}\|_{\mreg(T)}\geq \varepsilon,\, \chi_{R}^n \geq T\big)
+\P(\chi_R^n <T)\\
&
\leq \frac{1}{\varepsilon^p} 
\E\|\wcut^{(n)}\|_{\mreg(\chi_R^n)}^p 
+\frac{1}{R^p}(\E\|\vcut\|_{\mreg(T)}^p+\E\|\vcut^{(n)}\|_{\mreg(T)}^p)\\
&
\leq \frac{C}{\varepsilon^p}  \|w^{(n)}_0\|_{B^{2-\s-2(1+\a)/p}_{q,p}}^p
+\frac{C}{R^p}\|v_0\|_{B^{2-\s-2(1+\a)/p}_{q,p}}^p,
\end{align*}
where, in the last step, we used \eqref{eq:estimate_wn_on_chiRn} and the uniform estimate in \eqref{eq:apriori_estiamte_cutoff_global}. 

In light of \eqref{eq:interpolation_spaces_Lq}, the conclusion now follows by letting $n\to \infty$, and afterwards $R\to \infty$. 
\end{proof}

\subsection{Scaling limit for stochastic RDEs with cutoff}
\label{ss:cutoff2}
In this subsection, a scaling limit result for stochastic RDEs with cutoff \eqref{eq:reaction_diffusion_cut_off} is proven, essentially using the ideas in \cite[Theorem 6.1]{A22} (see eq.\ (6.21) there). Here, we limit ourselves to giving comments on the modifications needed to apply the strategy in \cite[Theorem 6.1]{A22}.
To begin, for notational convenience, for fixed $q>\frac{d(h-1)}{2}\vee 2$ and $N\geq 1$, and for a compact set $K\subseteq \q \R^\ell_{> 0}$, we let 
\begin{equation}
\label{eq:def_K}
\mathcal{O}_{q,N} \coleq 
\big\{v_0\in L^q(\T^d;\R^\ell_{\geq 0})\,:\,  \|v_0\|_{L^q(\T^d;\R^\ell)}\leq N\ \text{ and } \ \q\overline{v_0}\in K \big\}.
\end{equation}
In the above, with a slight abuse of notation, we did not display the dependence on $K$. Moreover, we consider the RDEs with cutoff and enhanced diffusivity:
\begin{equation}
\label{eq:reaction_diffusion_cut_off_deterministic}
\left\{
\begin{aligned}
&\partial_t \vcdi 
 = (\nu+\nu_i) \Delta \vcdi +\phi_{R,r}(\vcd) f_i(\vcd)
& \text{ on }&\Tor^d,\\ 
&\vcdi(0)=v_{0,i}  &\text{ on }&\Tor^d.
\end{aligned}\right.
\end{equation}

\begin{theorem}[Scaling limit for stochastic RDEs with cutoff -- Finite interval]
\label{t:scaling_limit_cutoff}
Let \eqref{reaction-full} be a complex balanced chemical reaction network, and consider the reaction nonlinearities given by \eqref{b-1}.
Let $(p,q,\delta)$ be as in Theorem \ref{t:global}. 
Fix $T<\infty$. 
Assume that for each $v_0\in 
\mathcal{O}_{q,N}$, there exists a unique 
\begin{equation}
\label{eq:weak_solution_regularity_class_cutoff}
\vcd\in L^2(0,T;H^{1-\g})\cap C([0,T];H^{-\g})\cap L^\infty(0,T;L^q)\cap L^q(0,T;L^\xi)
\end{equation}
for any $\g\in (0,1)$ and for any $\xi<\infty$ if $d=2$, or $\xi=\frac{qd}{d-2}$ otherwise, that is a unique weak solution (in the PDE sense) to \eqref{eq:reaction_diffusion_cut_off_deterministic}. 
Let $(\theta^{(n)})_{n}\subseteq \ell^2$ be a sequence such that  
\begin{equation}
\label{eq:scaling_limit_noise_coefficients}
\|\theta^{(n)}\|_{\ell^2}\equiv 1 \ \text{ for all }\ n \ \quad \text{ and }\ \quad \lim_{n\to \infty }\|\theta^{(n)}\|_{\ell^\infty}=0.
\end{equation}
Then there exists $r<\infty$ sufficiently large, depending only on $q,p,\delta$ and $ d$, for which the following assertion holds. 
Let $(\vcut^{(n)})_n$ be the $(p,\a_{p,\delta},\delta,q)$-solution to the stochastic RDEs with cutoff \eqref{eq:reaction_diffusion_cut_off} with $\theta=\theta^{(n)}$ and initial data $v_0\in \mathcal{O}_{q,N}$.  
Then
\begin{equation}
\label{eq:scaling_limit_cutoff1}
\limsup_{n \to \infty} \sup_{v_0\in 
\mathcal{O}_{q,N}} \P\big(\|\vcut^{(n)}-\vcd\|_{L^r(0,T;L^q(\T^d;\R^\ell))}\geq \varepsilon\big)=0 \ \ \text{ for all }\ \varepsilon\in (0,1).
\end{equation}
\end{theorem}

A typical choice of noise coefficients satisfying  \eqref{eq:scaling_limit_noise_coefficients} is given by
\begin{equation}
\label{eq:choice_noise_coefficients_scaling}
\theta_k= \frac{\Theta_k}{\|\Theta\|_{\ell^2}}\quad \text{ where }\quad \Theta_k = |k|^{-\g} \one_{N\leq |k|\leq 2N} \ \text{ for } \ k\in \Z^d_0,
\end{equation}
and where $\g>0$ is fixed; see, e.g., \cite[Subsection 1.1]{L21}.
Before going further, let us mention that weak solutions to \eqref{eq:reaction_diffusion_cut_off_deterministic} with regularity as in \eqref{eq:weak_solution_regularity_class_cutoff} are understood in the usual PDE sense: For all $\psi\in C^\infty(\T^d)$, $i\in \{1,\dots,\ell\}$, and for all $t\in [0,T]$,
\begin{align}
\langle \vcdi(t),\psi\rangle 
=
\langle v_{0,i},\psi\rangle 
+ \int_0^t \Big(-(\nu+\nu_i)\langle \nabla \vcdi, \nabla \psi\rangle + \phi_{R,r}(\vcd)\int_{\T^d} f_i(\vcd)\,\psi\,   \dd x \Big)\,\dd s.
\end{align}
In the above, we denoted by $\langle \cdot ,\cdot \rangle$ the duality pairing between $H^{-\g}$ and $H^{\g}$. Let us emphasize that the above weak formulation is well-defined as $f_i(\widehat{v}_{\mathrm{cut}})\in L^1(0,T;L^1)$ due to \eqref{eq:weak_solution_regularity_class_cutoff}, see Step 1 in \cite[Proposition 5.3]{A22}.

\begin{proof}[Proof of Theorem \ref{t:scaling_limit_cutoff} -- Sketch]
The proof of \eqref{eq:scaling_limit_cutoff1} is essentially given in 
\cite{A22}. However, here we comment on the modifications needed in order to handle complex balanced chemical reactions instead of reaction--diffusion with mass conservation, and how to weaken the regularity on the initial data in \cite[eq.\ (6.21)]{A22}.

\smallskip

\emph{Step 1: (Uniform control of the mass). There exists a constant $M_0>0$ depending only on $N$, $d$, and the compact set $K\subseteq \q \R^\ell_{> 0}$ such that, for all $v_0\in \mathcal{O}_{q,N}$, the weak solution to \eqref{eq:reaction_diffusion_cut_off} satisfies
\begin{equation}
\label{eq:control_ass_cut_equation}
\|\vcut(t)\|_{L^1(\T^d;\R^\ell)}\leq M_0 \ \ \text{ a.s.\ for all }\ t>0.
\end{equation}
In particular, $M_0$ is independent of the radially symmetric noise coefficient $\theta$.}

By a standard mollification argument, one can check that, for all $v_0\in L^q(\T^d;\R^\ell_{\geq 0})$, there exists a sequence $(v_{0,n})_{n}\subseteq C^{\g}(\T^d;\R^\ell_{\geq 0})$ such that $v_{0,n}\to v_0$, it satisfies the same $L^q$-bound, and $\q\overline{v_{0,n}}=\q \overline{v_0}$. Hence, in the following, we assume that $v_0\in C^\g(\T^d;\R^\ell_{\geq 0})$ and from Proposition \ref{prop:global_cut_off}, the corresponding solution satisfies \eqref{eq:global_cut_off1} with $\delta=1$. 

By the regularity of $\vcut$ in Proposition \ref{prop:global_cut_off} and \eqref{eq:Q_anniled_f}, one gets
$$
\q \overline{\vcut(t)}=\q \overline{v_\infty} \ \ \text{ a.s.\ for all }\ t>0.
$$
By Lemma \ref{l:continuous_cinfty} and the Csisz\'ar--Kullback--Pinsker inequality (see, e.g., \cite{AMTU01} or \cite[Lemma 2.6]{DFT17}), to prove \eqref{eq:control_ass_cut_equation}, it suffices to show 
\begin{equation}
\label{eq:entropy_dissipation_cutoff}
\Ent(\vcut(t)|v_\infty)\leq \Ent (v_0|v_\infty).
\end{equation}

To prove the above, by a standard approximation argument together with the It\^o formula (see either \cite[Section 3]{Kry13} or \cite[Appendix A]{DHV16}) and Fatou's lemma, one can check that, a.s.\ for all $t>0$, 
\begin{equation}
\label{eq:dissipation_inequality_relation_cutoff}
\int_{0}^t \big(\D_{{\rm diff}}(\vcut)+ \phi_{R,r}(\vcut)\D_{{\rm reac}}(\vcut)\big)\,\dd s
\leq 
\Ent(v_0|v_\infty)-
\Ent(\vcut(t)|v_\infty),
\end{equation}
where $\D_{{\rm diff}}(\vcut)$ and $\D_{{\rm reac}}(\vcut)$ denote the diffusion and the reaction dissipation terms, i.e., for any $v\in H^1(\T^d;\R^\ell_{\geq 0})$,
\begin{align*}
\D_{{\rm diff}}(v)=
\sum_{1\leq i\leq \ell} \nu_i \int_{\T^d} \frac{|\nabla v_i|^2}{v_i}\,\dd x\ \  \text{ and }\ \  
\D_{{\rm reac}}(v)=
\sum_{1\leq r\leq m } k_r \sol_{\infty}^{\y_r} \int_{\T^d} \Psi\Big(\frac{\sol^{\y_r}}{\sol_{\infty}^{\y_r}},\frac{\sol^{\y_r'}}{\sol_{\infty}^{\y_r'}}\Big)\, \dd x,
\end{align*}
where $\Psi$ is as in \eqref{eq:dissipation}. As $
\D_{{\rm diff}}(v)\geq 0$ and 
$\D_{{\rm reac}}(v)\geq 0$, \eqref{eq:entropy_dissipation_cutoff} follows from \eqref{eq:dissipation_inequality_relation_cutoff}.

\smallskip

\emph{Step 2: Proof of \eqref{eq:scaling_limit_cutoff1} by density.}
In light of Step 1, the proof of \cite[Theorem 6.1]{A22} follows verbatim. The only changes are needed in the proof of \cite[Theorem 4.1(2)]{A22}, where one uses the bound \eqref{eq:control_ass_cut_equation} to bound lower order terms, and a standard approximation argument together with \eqref{eq:global_cut_off_continuity1} to pass to the limit in the a priori bounds in \cite[Theorem 4.1(2)]{A22} (see the proof of Lemma \ref{lem:Lebesgue_initial_data_continuity} for an analogous result). Hence, \cite[eq.\ (6.21)]{A22} also holds in this case. To prove \eqref{eq:scaling_limit_cutoff1}, it remains to note that, for all $\g\in (0,1)$ and $\varepsilon>0$, 
\begin{equation}
\label{eq:equality_of_the_sup_proof_scaling}
\sup_{v_0\in 
\mathcal{O}_{q,N}}
\P\big(\|\vcut^{(n)}-\vcd\|_{L^r(0,T;L^q(\T^d;\R^\ell))}\geq \varepsilon\big)
\leq \sup_{v_0\in 
\mathcal{B}_{q,N,\g}}
\P\big(\|\vcut^{(n)}-\vcd\|_{L^r(0,T;L^q(\T^d;\R^\ell))}\geq \varepsilon/2\big)
\end{equation}
where 
$
\mathcal{B}_{q,N,\g} \coleq \mathcal{O}_{q,N} \cap C^{\g}(\T^d;\R^\ell).
$
Clearly, \eqref{eq:equality_of_the_sup_proof_scaling} follows from \eqref{eq:global_cut_off_continuity2} and the approximation argument via mollification mentioned at the beginning of Step 1 of the current proof.
\end{proof}

We collect some observations in the following

\begin{remark}\
\begin{itemize}
    \item Using stochastic Meyers' estimates \cite[Appendix A]{A25_anomalous}, one can prove that \eqref{eq:scaling_limit_cutoff1} holds with the space $L^r(0,T;L^q)$ replaced by $C([0,T];L^q)$ at the expense of additional regularity with respect to the initial data. 
    \item It does not seem possible to obtain the scaling limit of Theorem \ref{t:scaling_limit_cutoff} on the whole line, i.e., $T=\infty$. The key problem is in the compactness argument behind \cite[Theorem 6.1]{A22}, and the possibly complicated long-time behavior of the stochastic RDEs with cutoff \eqref{eq:reaction_diffusion_cut_off}, which might differ from the one of \eqref{eq:reaction_diffusion} in the case $\|\vcut\|_{L^r(0,t;L^q)}\geq 2R$ for $t\gg 1$ (and removing the reaction term as $\phi_{R,r}(\vcut)=0$). Thus, in the latter case, one cannot use the entropy--entropy dissipation relation of Theorem \ref{t:entropy_dissipation}, which effectively uses the nonlinear reactions for the convergence to the equilibrium $v_\infty$.
\end{itemize}    
\end{remark}

\subsection{Global well-posedness of deterministic RDEs with high diffusivity}
\label{ss:global_with_high_diffusivity}
The aim of this subsection is to prove the global well-posedness of the following deterministic RDEs with large diffusion coefficients:
\begin{equation}
\label{eq:reaction_diffusion_deterministic_enhanced_diffusivity}
\left\{
\begin{aligned}
&\partial_t \vdi 
 = \mu_i \Delta \vdi +f_i(\vd)
& \text{ on }&\Tor^d,\\ 
&\vdi(0)=v_{0,i} &\text{ on }&\Tor^d.
\end{aligned}\right.
\end{equation}
Below, a $(p,\a,\delta,q)$-solution to the \emph{deterministic} reaction--diffusion equations is understood as in Definition \ref{def:solution} with zero noise. Let us point out that in this case, due to the lack of the stochastic term, the time weight $\a$ can be chosen from the interval $[0,p-1)$ differently from what is assumed in Definition \ref{def:solution} (in the latter, the constraint $\a<\frac{p}{2}-1$ is required to make sense of the stochastic integral). We refer to, e.g., \cite[Section 5]{pruss2016moving} or \cite{CriticalQuasilinear} for the $L^p$-approach to evolution equations with time weights.

\begin{proposition}[Global well-posedness with high diffusivity]
\label{prop:global_high_viscosity}
Let \eqref{reaction-full} be a complex balanced chemical reaction network. Fix $q\geq \frac{d(h-1)}{2}\vee 2$, $p\in [q,\infty)$, and let $\kappa_p=\frac{p}{2}-1$. Let $N\geq 1$ and a compact set $K\subseteq \q \R^\ell_{> 0}$ be given. 
Then, there exists a constant $\mu_0\geq 1$ depending only on $N$, $q$, $p$, $d$, and $K$, for which the following assertion holds: If 
$$
\min_{1\leq i\leq \ell}\mu_i\geq \mu_0,
$$
then, for all initial data $v_0$ satisfying
$$
v_0\in \InD^q , \qquad  \|v_0\|_{L^q(\T^d;\R^\ell)}\leq N, \qquad \text{ and }\qquad 
\q \overline{ v_0} \in K,
$$
there exists a \emph{global (unique)} $(p,\a_p,1,q)$-solution $\wh{v}$ to the deterministic RDEs \eqref{eq:reaction_diffusion_deterministic_enhanced_diffusivity} such that 
\begin{align}
\label{eq:global_high_viscosity1}
\vd
&\in W^{1,p}_{\loc}([0,\infty),w_{\a_p};W^{-1,q} )\cap L^{p}_{\loc}([0,\infty),w_{\a_p};W^{1,q} ).
\end{align}
If additionally, for each $M\in \q\R^\ell_{> 0}$, there are no boundary equilibria $\wt{v}_\infty$ such that $\q\wt{v}_\infty=M$ (see Definition \ref{def:complex_boundary_equilibria}), then, for any $q_0\in (1,q)$, there exist $L, \lambda>0$ such that
\begin{align}
\label{eq:global_high_viscosity2}
\|\vd-v_\infty\|_{L^\infty(\R_+;L^q)}+\max_{1\leq i\leq \ell}\int_{ \R_+}\int_{\T^d} |\vdi-v_{\infty,i}|^{q-2}|\nabla \vdi|^2\,\dd x \,\dd t &\leq L,\\
\label{eq:global_high_viscosity21}
\sup_{t>0} \big(e^{\lambda t}\|\vd(t)-v_\infty\|_{L^{1}}\big)&\leq L .
\end{align}
\end{proposition}

In the setting of Proposition \ref{prop:global_high_viscosity}, by interpolating \eqref{eq:global_high_viscosity2} and \eqref{eq:global_high_viscosity21}, it follows that, for each $q_0\in (1,q)$, there exist $L_0,\lambda_0>0$ depending only on $\lambda$, $L$, and $q_0$, for which 
\begin{equation}
\label{eq:global_high_viscosity3}
\|\vd(t)-v_\infty\|_{L^{q_0}}\leq L_0 e^{- \lambda_0 t} \ \  \text{ for all }\ t>0.
\end{equation}

\begin{remark}
    Proposition \ref{prop:global_high_viscosity} is an improvement of \cite[Theorem 1.2]{morgan2020boundedness} because here it only requires the boundedness of initial data in $L^q$-norm with $q \geq \frac{d(h-1)}{2}\vee 2$ instead of the boundedness in $L^\infty$-norm as in \cite{morgan2020boundedness}. Interestingly, we cover the critical regime $q=\frac{d(h-1)}{2}\vee 2$.
\end{remark}

The proof of Proposition \ref{prop:global_high_viscosity} follows the line of \cite[Proposition 5.1]{A22}. However, we provide some details for establishing the a priori estimates \eqref{eq:global_high_viscosity2}--\eqref{eq:global_high_viscosity3}, where we employ the entropy--entropy dissipation inequality Theorem \ref{t:entropy_dissipation}. As shown in Figure \ref{fig:proof_architecture} (see also Remark \ref{rem:boundary_equilibria}), this is the only place where the latter are used in the proof of Theorems \ref{t:global} and \ref{t:enhanced_dissipation_reaction}.

\begin{proof}
Arguing as in \cite[Proposition 5.1]{A22} (see also Proposition \ref{prop:LWP} and Lemma \ref{lem:Lebesgue_initial_data_continuity}), as $p\geq q$, for all 
$$
v_0\in L^q(\T^d;\R^\ell)\subseteq B^0_{q,p}(\T^d;\R^\ell),
$$   
there exists a \emph{local} $(p,\a_p,1,q)$-solution $(\vd,\taud)$ to \eqref{eq:reaction_diffusion_deterministic_enhanced_diffusivity} satisfying \eqref{eq:global_high_viscosity1} with $[0,\infty)$ replaced by $[0,\taud)$. Moreover, from Lemma \ref{lem:Lebesgue_initial_data_continuity}, it holds that
\begin{align*}
\vd&\in C([0,\taud);L^q(\T^d;\R^\ell_{\geq 0})),\\
|\vdi|^{q-2}|\nabla \vdi |^2&\in L^1([0,\taud)\times \T^d)
\end{align*}
for all $i\in \{1,\dots,\ell\}$. Thus, it remains to prove $\taud=\infty$ and the estimates \eqref{eq:global_high_viscosity2}--\eqref{eq:global_high_viscosity3}. 
We divide the proof into several steps. Here we focus on proving \eqref{eq:global_high_viscosity2}--\eqref{eq:global_high_viscosity21}, while the assertion \eqref{eq:global_high_viscosity1} can be proven analogously avoiding the use of the entropy--entropy dissipation relation in Step 1, and using only the entropy bound \eqref{eq:entropy_decay_strict_inequality} below, which yields an $L^1((0,T)\times \T^d)$-bound on $v$ with corresponding constant depending on $T$ via the Csisz\'ar--Kullback--Pinsker inequality, see the proof of \eqref{eq:control_ass_cut_equation} for a similar situation.

\smallskip

\emph{Step 1: There exist constants $L_1,\lambda_1>0$ depending only on $N$, $q$, $p$, $d$, and $K$ such that}
\begin{equation*}
    \|\vd(t)- v_\infty\|_{L^1}\leq L_1 e^{-\lambda_1 t}\ \  \text{ for all }\ t\in [0,\taud).
\end{equation*}
Let $\Ent$ be the entropy as defined in \eqref{eq:entropy}. From the discussion at the beginning of Section \ref{s:entropy_review}, we have 
\begin{equation}
\label{eq:entropy_dissipation_relation_deterministic_proof}
\frac{\dd}{\dd t }\Ent(\vd(t)| \veq)=- \wh{\D} (\vd(t))\leq 0 \  \text{ for all }\ t\in [0,\taud).
\end{equation}
where $\wh{\D}$ is as in \eqref{eq:dissipation} with $\nu_i$ replaced by $\mu_i$. 
Hence,
\begin{equation}
\label{eq:entropy_decay_strict_inequality}
\Ent(\vd(t)|\veq)\leq \Ent(v_0|\veq)\lesssim 1+N\  \text{ for all }\ t\in [0,\taud)
\end{equation} 
as $v_0\in 
\mathcal{O}_{q,N}$, where the latter is defined in \eqref{eq:def_K}. 

Next, we apply Theorem \ref{t:entropy_dissipation}. To obtain constants that are uniform in $v_0$, we first need to prove a uniform bound on $\Ent(\overline{\vd(t)}|\veq)$ for $t\in [0,\taud)$. To this end, note that  
$$
\Ent(\vd(t)|\veq)
=
\Ent(\vd(t)|\overline{\vd(t)})+
\Ent(\overline{\vd(t)}|\veq)
$$
and $
\Ent(\vd(t))|\overline{\vd(t)})\geq 0$ for $t\in [0,\taud)$. Combining this with \eqref{eq:entropy_decay_strict_inequality}, we obtain $\Ent(\overline{\vd(t)}|\veq)\lesssim 1+ N$. Therefore, the conclusion of Step 1 follows from \eqref{eq:entropy_dissipation_relation_deterministic_proof} and Theorem \ref{t:entropy_dissipation}.

\smallskip

\emph{Step 2: There exists a constant $L_2>0$ depending only on $N$, $q$, $p$, $d$, and $K$ such that}
\begin{equation*}
\|\vd-v_\infty\|_{L^\infty(0,\taud ;L^q)}+\int_{ 0}^{\taud }\int_{\T^d} |\vdi-v_{\infty,i}|^{q-2}|\nabla \vdi|^2\,\dd x \,\dd t \leq L_2.
\end{equation*}
The proof of Step 2 is similar to that of Lemma \ref{l:Lq_estimate_small}, where one uses
$$
\int_{0}^{\taud}\int_{\T^d} |\wdd|\,\dd x \,\dd s\lesssim \Ent(v_0|\veq)\lesssim_{K,N}1,
$$
as it follows from Step 1. Note that in Lemma \ref{l:Lq_estimate_small}, we cannot use the above as Step 1 relies on the absence of boundary equilibria via Theorem \ref{t:entropy_dissipation}; the latter is not assumed in Lemma \ref{l:Lq_estimate_small} (in the latter, linear terms are estimated via the spectral gap of the linearized operator around the equilibrium $\veq$; see Lemma \ref{l:L2_estimates_small_implies}).

\smallskip

\emph{Step 3: Conclusion: $\taud=\infty$}. From the estimate in Step 2, the conclusion follows as in Step 3 in the proof of 
Theorem \ref{t:small_implies_global}; see Subsection \ref{ss:proof_small_implies_global}.
\end{proof}

\subsection{Proof of Theorem \ref{t:global}}
\label{ss:proof_global_well_posedness}
With all the previous results at our disposal, we can finally provide the proof of the global well-posedness result of Theorem \ref{t:global}.

\begin{proof}[Proof of Theorem \ref{t:global}]
In light of Theorems \ref{t:small_implies_global},  \ref{t:scaling_limit_cutoff} and Proposition \ref{prop:global_high_viscosity}, the proof of Theorem \ref{t:global} is standard. For the reader's convenience, we include some details.

We begin by collecting some facts. First, by compatibility of $(p,\a_{p,\delta},\delta,q)$-solutions to \eqref{eq:reaction_diffusion}, it suffices to consider fixed parameters $q>\frac{d(h-1)}{2}\vee 2$, $\delta\in (1,2)$ and $p> \frac{2}{2-\delta}\vee q$, where as usual, $\a_{p,\delta}=p(1-\frac{\delta}{2})-1$. Second, $N\geq 1$ and 
$\varepsilon\in (0,1)$ are given constants. Moreover, for $\InD^q$ as in \eqref{eq:initial_data_Xin}, we define 
\begin{equation}
\label{eq:def_mathcal_OInD}
\mathcal{O}_{q,N}^{\InD}
=\{v_0\in \InD^q\,:\,  \|v_0\|_{L^q(\T^d;\R^\ell)}\leq N\ \text{ and } \ \q\overline{ v_0}\in K  \}.
\end{equation}
Finally, let $r\in (1,\infty)$ be so large that the conclusion of Theorem \ref{t:scaling_limit_cutoff} holds, and let $\eta\in (0,1)$ be such that Theorem \ref{t:small_implies_global} holds with $q$ replaced by $q_*>2$ such that $q_*\in (\frac{d(h-1)}{2}\vee \frac{d}{d-\delta},q)$, and $(p,\delta)$ as above.

\smallskip

It follows from Proposition \ref{prop:global_high_viscosity} and \eqref{eq:global_high_viscosity2} that there exist constants 
\begin{equation}
\label{eq:parameters_choice_global}
\nu_0(q,K,p)>0\quad \text{ and } \quad  R_0(q,K,N,r),T_0(q,K,N,r,q_*)>0
\end{equation}
such that, for all $v_0\in \mathcal{O}^{\InD}_{q,N}$ and $\nu\geq \nu_0$, the unique $(p,\a_{p,\delta},\delta,q)$-solution $\vd$ to the RDEs \eqref{eq:reaction_diffusion_deterministic_enhanced_diffusivity} with increased diffusivity $\mu_i=\nu_i+\nu$ is \emph{global} in time and satisfies 
\begin{equation}
\label{eq:smallness_global_holds}
\|\vd\|_{L^r(0,T_0;L^q)}\leq R_0-1\quad \text{ and }\quad
\sup_{t\geq T_0-1}\|\vd-\veq\|_{L^{q_*}}\leq \frac{\eta}{2}.
\end{equation}
Let us stress that all the above choices are independent of  $\theta\in \ell^2$ satisfying \eqref{eq:theta_normalized_symmetric}, while for the intensity of the noise, we enforce $\nu\geq \nu_0$.

\smallskip

For clarity, we now divide the proof into two steps. Below, $(v,\tau)$ is the $(p,\a_{p,\delta},\delta,q)$-solution to \eqref{eq:reaction_diffusion} provided by Proposition \ref{prop:LWP} with initial data $v_0\in \mathcal{O}_{q,N}^{\InD}$.

\smallskip

\noindent \emph{Step 1: (Classical solutions up to time $T_0$). There exists a normalized radially symmetric $\theta\in \ell^2$ independent of $v_0\in \mathcal{O}_{q,N}^{\InD}$ such that the solution $(v,\tau)$ lives up to time $T_0$ with large probability:
\begin{equation}
\label{eq:step1_global_classical1}
\P(\tau\geq T_0)>1-\frac{\varepsilon}{2}.    
\end{equation}
Moreover, }
\begin{equation}
\label{eq:step1_global_classical2}
\P\Big(\tau\geq T_0,\,\|v-\vd\|_{L^r(0,T_0;L^q)}\leq\frac{\eta}{2}\Big)>1-\frac{\varepsilon}{2}.
\end{equation}

From Theorem \ref{t:scaling_limit_cutoff} applied with the choice of $(\theta^{(n)})_{n\geq 1}$ as in \eqref{eq:choice_noise_coefficients_scaling}, it follows that there exists a normalized radially symmetric $\theta\in \ell^2$ such that $\#\{k\,:\,\theta_k\neq 0\}<\infty$, and uniformly in $v_0\in \mathcal{O}_{q,N}^{\InD}$, the global $(p,\a_{p,\delta},\delta,q)$-solution to the stochastic RDEs \eqref{eq:reaction_diffusion_cut_off} with cutoff $\vcut$ with $R=R_0$ satisfies
\begin{equation}
\label{eq:large_probability_LrLq_close}
\P\Big(\|\vcut-\vd\|_{L^r(0,T_0;L^q)}\leq\frac{\eta}{2}\Big)>1-\frac{\varepsilon}{2}.
\end{equation}
Here, we also used that \eqref{eq:reaction_diffusion_cut_off_deterministic} implies that $\vd$ is also a solution to \eqref{eq:reaction_diffusion_cut_off_deterministic} with $R=R_0$. Thus, \cite[Corollary 5.5]{A22} ensures that $\vd$ is indeed the unique weak solution to \eqref{eq:reaction_diffusion_cut_off_deterministic}.

As $\eta\in (0,1)$ and $\vd$ satisfies the first bound in \eqref{eq:smallness_global_holds}, we have
\begin{equation}
\label{eq:large_probability_LrLq}
\P(\|\vcut\|_{L^r(0,T_0;L^q)}\leq R_0)>1-\varepsilon/2.
\end{equation}
Let $\tau_0$ be the stopping time given by 
$$
\tau_0 \coleq \inf\big\{t\in [0,T_0]\,:\,\|\vcut\|_{L^r(0,t;L^q)}\geq R_0\big\} \quad \text{ with } \quad \inf\emptyset \coleq T_0.
$$
Clearly, 
$$
\P(\tau_0\geq T_0)\geq \P(\|\vcut\|_{L^r(0,T_0;L^q)}\leq R_0)\stackrel{\eqref{eq:large_probability_LrLq}}{\geq }1-\varepsilon/2
$$
and 
\begin{equation}
\phi_{R,r}(\cdot,\vcut)=1  \ \ \text{ on }\ \  [0,\tau_0)\times\O.  
\end{equation}
In particular, on the stochastic interval $[0,\tau_0)\times \O$, the cutoff 
$\phi_{R,r}(\vcut)$ is not present, and $(\vcut|_{[0,\tau_0)\times \O},\tau_0)$ is a local $(p,\a_{p,\delta},\delta,q)$-solution to the stochastic RDEs \eqref{eq:reaction_diffusion} \emph{without} cutoff. From the maximality of a $(p,\a_{p,\delta},\delta,q)$-solution to the stochastic RDEs \eqref{eq:reaction_diffusion_cut_off} (see Definition \ref{def:solution}), we have 
\begin{equation}
\label{eq:tau_0tau_equality_global_proof}
\tau\geq \tau_0 \ \text{ a.s. }\quad \text{ and }\quad 
v=\vcut \ \text{ a.e.\ on }\ [0,\tau_0)\times \O.
\end{equation}
This proves \eqref{eq:step1_global_classical1}. The assertion \eqref{eq:step1_global_classical2} is a consequence of \eqref{eq:large_probability_LrLq_close} and \eqref{eq:tau_0tau_equality_global_proof}. Indeed, from the definition of $\tau_0$ and \eqref{eq:smallness_global_holds}, it follows that   
$$
\{\|\vcut-\vd\|_{L^r(0,T_0;L^q)}\leq \eta/2\}\subseteq \{\tau_0\geq T_0\}.
$$

\smallskip

\noindent\emph{Step 2: Conclusion}.  
Let $\nu$ and $\theta$ be fixed as above; see below \eqref{eq:parameters_choice_global} and Step 1, respectively.
We begin by noticing that \eqref{eq:step1_global_classical2} and the second bound in \eqref{eq:smallness_global_holds} imply
\begin{equation}
\begin{aligned}
\label{eq:smallness_Lq00}
\P(\O_0)>1-\varepsilon, \quad \text{ with }\quad 
\O_0 &\coleq \{\tau\geq T_0,\, \|\vcut-\vd\|_{L^r(0,T_0;L^{q_*})}\leq \eta/2\}\\
\text{ and }\quad \O_0&\subseteq \{\|\vcut-v_\infty\|_{L^r(T_0-1,T_0;L^{q_*})}
\leq\eta\},
\end{aligned}
\end{equation}
where $q_*\in (\frac{d(h-1)}{2}\vee \frac{d}{d-\delta}\vee 2,q)$ was fixed above \eqref{eq:parameters_choice_global}. 
As below \eqref{eq:tau_0tau_equality_global_proof}, 
\begin{align}
\label{eq:smallness_Lq000}
\tau\geq T_0 \ \text{ a.s.\ } \O_0\quad \text{ and }\quad \vcut=v \ \text{ a.e.\ on }[0,T_0)\times \O_0.
\end{align}
Hence, $\|\vcut-\veq\|_{L^r(T_0-1,T_0;L^{q_*})}\leq \eta$ on $\O_0$. Therefore, for all $\om \in \O_0$, there exists $t(\om)\in (T_0-1,T_0)$ such that 
\begin{equation}
\label{eq:smallness_Lq0}
\|\vcut(t(\om))-\veq\|_{L^{q_*}}\leq \eta.    
\end{equation}
Let $\tau_*$ be defined as
\begin{equation}
\label{eq:definition_tau_star}
\tau_*
\coloneq
\inf\big\{
t\in[T_0-1,T_0)\,:\,
\|\vcut(t)-v_\infty\|_{L^{q_*}}\leq\eta
\big\},
\qquad
\inf\emptyset\coloneq T_0.
\end{equation}
As $\a_{p,\delta}>0$ due to $p>\frac{2}{2-\delta}$ (see the conditions at the beginning of the proof), the assertion \eqref{eq:continuity_in_Lq_cutoff} ensures that $\tau_*$ is a stopping time. 
Additionally, from \eqref{eq:smallness_Lq00}--\eqref{eq:smallness_Lq0}, we can infer
\begin{equation}
\label{eq:small_implies_global_end_proof}
\P\big(
\tau_*<\tau,\,
\|v(\tau_*)-v_\infty\|_{L^{q_*}}\leq\eta
\big)
\geq\P(\O_0)>1-\varepsilon.
\end{equation}
The claim of Theorem \ref{t:global} now follows from Theorem \ref{t:small_implies_global} and the choice of $\eta\in (0,1)$ made at the beginning of the proof.
\end{proof}

From the proof of Theorem \ref{t:global}, we can extract the following result, which will be used in the proof of Theorem \ref{t:enhanced_dissipation_reaction}.

\begin{corollary}[Uniform estimates on the half-line]
\label{cor:uniform_estimates_half_line}
Let the assumptions of Theorem \ref{t:global} be satisfied, and fix $q_*>2$ such that $q_*\in (\frac{d(h-1)}{2}\vee \frac{d}{d-\delta},q)$. Suppose further that $r\geq q_0$, where $q_0=q \frac{d+2}{d}$. Let $(\theta^{(n)})_n$ and $\mathcal{O}_{q,N}^{\InD}$ be as in \eqref{eq:choice_noise_coefficients_scaling} and \eqref{eq:def_mathcal_OInD}, respectively. 
Then, for any $\varepsilon\in (0,1)$, there exist $\nu_0$ and $n_0$ such that, if $\nu\geq \nu_0$ and $n\geq n_0$, then, for any $(p,\a_{p,\delta},\delta,q)$-solution $(v,\tau)$ to \eqref{eq:reaction_diffusion} with initial data $v_0\in 
\mathcal{O}_{q,N}^{\InD}$ and $\theta=\theta^{(n)}$, there exists a stopping time $\tau_1\in (0,\tau]$ such that 
$$
\P(\tau_1=\infty)>1-\varepsilon\quad \text{ and }\quad
\|v-\veq\|_{L^2(0,\tau_1;L^{2})}+\|v-\veq\|_{L^r(0,\tau_1;L^{q_*})}\leq \mathscr{C} \ \text{ a.s.}, 
$$
where $\mathscr{C}$ depends only on $p$, $\delta$, $q$, $q_*$, $d$, $K$, $r$, $N$, and $\varepsilon$. 
\end{corollary}

\begin{proof}
Continuing Step 2 in the proof of Theorem \ref{t:global}, let
$$
\O_*\coleq
\{\tau_*<\tau,\,\|v(\tau_*)-\veq\|_{L^{q_*}}\leq\eta\}.
$$
In light of \eqref{eq:small_implies_global_end_proof}, we have
$\P(\O_*)>1-\varepsilon$. Hence, from Theorem
\ref{t:small_implies_global} applied with $q$ replaced by $q_*$,
we also obtain the existence of a stopping time
$\sigma\in(0,\tau]$ such that $\P(\sigma=\infty)>1-\varepsilon
$ and
\begin{equation}
\label{eq:tau_1_is_very_large_in_time}
\sup_{t\in[\tau_*,\sigma)}
\|v(t)-\veq\|_{L^{q_*}}
+
\|v-\veq\|_{L^{q_1}(\tau_*,\sigma;L^{q_1})}
+
\|v-\veq\|_{L^{2}(\tau_*,\sigma;L^{2})}
\leq M_0
\quad\text{a.s.\ on }\O_*,
\end{equation}
where $q_1=q_*\frac{d+2}{d}$, and $M_0$ is a constant depending
only on $p$, $\delta$, $q$, $q_*$, $d$, $K$, $r$, $N$, and
$\varepsilon$. As $r\geq q_0>q_1>q_*$, by interpolation and
H\"older's inequality, the previous estimate ensures that
\begin{equation}
\label{eq:estimate_on_mu_star_corollary_0}
\|v-\veq\|_{L^r(\tau_*,\sigma;L^{q_*})}
\leq M_1
\quad\text{a.s.\ on }\O_*,
\end{equation}
where $M_1$ is a constant depending only on $p$, $\delta$, $q$,
$q_*$, $d$, $K$, $r$, $N$, and $\varepsilon$.

Next, we focus on proving an estimate like
\eqref{eq:estimate_on_mu_star_corollary_0} also holds on a
stochastic interval of the form $[0,\mu_*)\times\O$. 
To this end, let $\mu_*$ be the stopping time
$$
\mu_*
\coleq
\inf\{t\in[0,T_0\wedge\tau):\,
\|v-\veq\|_{L^r(0,t;L^{q_*})}\geq R_1\},
\qquad
\inf\emptyset\coleq T_0\wedge\tau,
$$
where
$
R_1=R_0+T_0^{1/r}\sup_{\q\veq\in K}|\veq|<\infty,
$
$R_0$ is as in \eqref{eq:smallness_global_holds}, and where the
finiteness of the supremum is ensured by Lemma
\ref{l:continuous_cinfty} and the compactness of $K$. Clearly,
\begin{equation}
\label{eq:estimate_on_mu_star_corollary}
\|v-\veq\|_{L^r(0,\mu_*;L^{q_*})}\leq R_1.
\end{equation}
By Step 1 in the proof of Theorem \ref{t:global} and
\eqref{eq:smallness_Lq00}, it follows that
$$
\P(\mu_*\geq T_0)>1-\varepsilon.
$$
Recall from \eqref{eq:definition_tau_star} that $\tau_*\leq T_0$.
Therefore,
$$
\P(\{\mu_*\geq\tau_*\}\cap\O_*)
\geq
\P(\{\mu_*\geq T_0\}\cap\O_*)
>1-2\varepsilon.
$$
Combining this with $\P(\sigma=\infty)>1-\varepsilon$, we obtain
$$
\P(\{\mu_*\geq\tau_*\}\cap\O_*\cap\{\sigma=\infty\})
>1-3\varepsilon.
$$
Thus, by \eqref{eq:estimate_on_mu_star_corollary_0} and
\eqref{eq:estimate_on_mu_star_corollary},
\begin{equation*}
\P\left(
\|v-\veq\|_{L^r(0,\infty;L^{q_*})}
\leq (M_1^r+R_1^r)^{1/r}
\right)>1-3\varepsilon.
\end{equation*}
Now, letting
$$
\mu
\coleq
\inf\left\{
t\in[0,\tau):\,
\|v-\veq\|_{L^r(0,t;L^{q_*})}
\geq 2(M_1^r+R_1^r)^{1/r}
\right\},
\qquad
\inf\emptyset\coleq\tau,
$$
we have
$
\|v-\veq\|_{L^r(0,\mu;L^{q_*})}
\leq 2(M_1^r+R_1^r)^{1/r}
$
a.s.\
Defining $\mu$ as above with the additional stopping condition
$
\|v-\veq\|_{L^2(0,t;L^2)}\geq \mathscr{C},
$
a routine modification of the preceding argument gives
$\P(\mu=\infty)>1-3\varepsilon$. The claim follows by setting
$\tau_1=\mu$ and replacing $\varepsilon$ by $\varepsilon/3$.
\end{proof}

\section{Enhanced dissipation with high probability -- Proof of Theorem \ref{t:enhanced_dissipation_reaction}}
\label{s:enhanced}
This section is devoted to the proof of Theorem \ref{t:enhanced_dissipation_reaction}. In light of the result of Corollary \ref{cor:uniform_estimates_half_line}, in this section, we will focus on the stochastic RDEs \eqref{eq:reaction_diffusion} with (modified) cutoff:
\begin{equation}
\label{eq:reaction_diffusion_cut_off_Lp}
\left\{
\begin{aligned}
\dd \vmodi - \nu_i \Delta \vmodi \,\dd t
& = \phi_{R,r}(\cdot,v_\infty,\vmod) f_i(\vmod)\,\dd t +\sqrt{c_d \nu }\sum_{k,\alpha} \theta_k (\sigma_{k,\alpha}\cdot\nabla) \vmodi\circ \dd W^{k,\alpha}_t & \text{ on }&\Tor^d,\\ 
\vmodi(0)&=v_{0,i}  &\text{ on }&\Tor^d.
\end{aligned}\right.
\end{equation}
where, for $q\in (\frac{d(h-1)}{2}\vee 2, \infty)$, $q_*\in (\frac{d(h-1)}{2}\vee 2,q)$, $R>0$, and $r\in (1,\infty)$,
$$
\phi_{R,r}(t,v_\infty,\vmod) \coleq \phi\big(R^{-1}[ \|\vmod-v_\infty\|_{L^r(0,t;L^{q_*}(\T^d;\R^\ell))}+\|\vmod-v_\infty\|_{L^2(0,t;L^{2}(\T^d;\R^\ell))}]\big) 
$$
with $\phi\in C^{\infty}_{{\rm }}(\R)$ such that $\phi|_{[0,1]}=1$ and $\phi|_{[2,\infty)}=0$. As usual, $v_\infty$ is the unique positive equilibrium to the RDEs \eqref{eq:reaction_diffusion} for which $\q \int_{\T^d} v_0\,\dd x =\q v_\infty$. 
Note that the cutoff \eqref{eq:reaction_diffusion_cut_off_Lp} differs from the one used in Section \ref{s:scaling_limit} and better accommodates the long-time behavior of the solution to the stochastic RDEs \eqref{eq:reaction_diffusion}. 
Clearly, an inspection of the proof of Proposition \ref{prop:global_cut_off} shows that \eqref{eq:reaction_diffusion_cut_off_Lp} admits a global $(p,\a,\delta,q)$-solution for all 
$q,p\in (2,\infty)$, $\a\in (0,\frac{p}{2}-1)$, and $\delta\in [1,2)$ satisfying \eqref{eq:global_cut_off0}.

\smallskip

The following are the remaining ingredients needed to prove Theorem \ref{t:enhanced_dissipation_reaction}. In the first one, we prove a pathwise estimate of the difference between the solution $\vmod$ to \eqref{eq:reaction_diffusion_cut_off_Lp} and its average $\int_{\T^d} \vmod(t,x)\,\dd x$. The structure of the noise described in Subsection \ref{ss:noise_description} plays a fundamental role in both results.
For simplicity, we employ the following shorthand notation: 
$$
\wt{\vmod}(t,x) \coleq \vmod(t,x)-\overline{\vmod(t,\cdot)}\qquad \text{ and }\qquad \overline{\vmod(t,\cdot)} = \int_{\T^d} \vmod(t,x)\,\dd x .
$$

\begin{proposition}[Pathwise $L^2$-estimate for the spatial fluctuations]
\label{prop:quenched_estimate_w}
Under the assumptions of Theorem \ref{t:enhanced_dissipation_reaction}, there exist $r<\infty$ and $q_*\in (\frac{d(h-1)}{2}\vee 2,q)$, for which the following assertion holds. There exists some $C_0>0$ depending only on $R$, $q$, $d$, $h$, $p$, $\delta$, and $K$ such that, for all $0\leq s<t <\infty$ subject to $|s-t| \leq 1$, the global $(p,\a_{p,\delta},\delta,q)$-solution $\vmod$ to \eqref{eq:reaction_diffusion_cut_off_Lp} satisfies
$$
\sup_{r\in [s,t]}\|\wt{\vmod}(r)\|_{L^2}^2
+
\int_s^t\|\wt{\vmod}(r)\|_{L^2}^2\,\dd r 
\leq C_0  \|\wt{\vmod}(s)\|_{L^2}^2 \ \text{ a.s. }
$$
\end{proposition}

The previous result gives a uniform pathwise control of the solution. The following provides a quantitative decay of the energy of $\wt{\vmod}$ at large times. 

\begin{proposition}[Uniform energy decay for the spatial fluctuations]
\label{prop:iteration_lemma}
Under the assumptions of Theorem \ref{t:enhanced_dissipation_reaction}, there exist $r<\infty$ and $q_*\in (\frac{d(h-1)}{2}\vee 2,q)$, for which the following assertion holds. For all $\eta\in (0,1)$, there exists some $\varepsilon\in (0,1)$ depending only on $R$, $q$, $d$, $h$, $p$, $\delta$, $K$ and $\eta$ such that the global $(p,\a_{p,\delta},\delta,q)$-solution $\vmod$ to \eqref{eq:reaction_diffusion_cut_off_Lp} satisfies, for all $n\geq 0$,
$$
\E \|\wt{\vmod}(n+1)\|_{L^2}^2 \leq \eta\, \E\|\wt{\vmod}(n)\|_{L^2}^2
$$
provided $\nu \geq \varepsilon^{-1}$ and $\|\theta\|_{\ell^\infty}\leq \varepsilon$.
\end{proposition}

The proof of Propositions \ref{prop:quenched_estimate_w} and \ref{prop:iteration_lemma} also shows why the statement \eqref{eq:enhanced_dissipation_reaction2} in Theorem \ref{t:enhanced_dissipation_reaction} cannot hold in general with $\wt{v}$ replaced by $v-v_\infty$; see the discussion following \eqref{eq:basic_decay_enhanced}.

\subsection{Proof of Theorem \ref{t:enhanced_dissipation_reaction}}
\label{ss:proof_enhanced}
We first show how to derive Theorem \ref{t:enhanced_dissipation_reaction} from the above results. The proofs of Propositions \ref{prop:quenched_estimate_w} and \ref{prop:iteration_lemma} are provided subsequently in Subsections \ref{ss:proof_quenched_estimate_w} and \ref{ss:proof_iteration_lemma}, respectively.

\begin{proof}[Proof of Theorem \ref{t:enhanced_dissipation_reaction}]
From Corollary \ref{cor:uniform_estimates_half_line} and the fact that the sequence of noise coefficients in \eqref{eq:choice_noise_coefficients_scaling} satisfies \eqref{eq:scaling_limit_noise_coefficients}, it is enough to show the corresponding estimate for the $(p,\a_{p,\delta},\s,q)$-solution $\vmod$  to the stochastic RDEs \eqref{eq:reaction_diffusion_cut_off_Lp} with cutoff. 
Without loss of generality, we assume $v_0\not\equiv \overline{v_0}$; otherwise $\wt{\vmod}(t)\equiv 0$ and the claim of Theorem \ref{t:enhanced_dissipation_reaction} follows trivially. For notational convenience, we let
$$
\wt{\vmod}_0 \coleq v_0-\overline{v_0}.
$$
Due to Propositions \ref{prop:quenched_estimate_w} and \ref{prop:iteration_lemma}, the proof now follows the line of \cite[Theorem 2.3]{L21}. For the reader's convenience, we provide a sketch. 
Let $\chi$, $b$, and $C_0$ be as in Theorem \ref{t:enhanced_dissipation_reaction} and Proposition \ref{prop:quenched_estimate_w}, respectively.
Without loss of generality, we assume $C_0\geq 1$.
Due to Proposition \ref{prop:iteration_lemma}, we can choose $(\nu,\theta)$ such that 
\begin{equation}
\label{eq:choice_chi_chiN}
\chi'>\chi\Big(1+\frac{b}{2}\Big)+\frac{b}{2}\log C_0 \quad \text{ where }\quad \chi' \coleq -\frac{1}{2}\log \eta>0.
\end{equation} 
Iterating the estimate in Proposition \ref{prop:iteration_lemma}, it follows that 
$$
\E\|\wt{\vmod}(n)\|_{L^2}^2 \leq \eta^n \|\wt{\vmod}_0\|_{L^2}^2=e^{-2\chi' n}\|\wt{\vmod}_0\|_{L^2}^2 \quad \text{ for all }n\geq 1.
$$
Moreover, from Proposition \ref{prop:quenched_estimate_w}, we get
\begin{equation}
\label{eq:decay_estimate}
\E\sup_{t\in [n,n+1]}\|\wt{\vmod}(t)\|_{L^2}^2 \leq C_0 \E\|\wt{\vmod}(n)\|_{L^2}^2 \leq C_0e^{-2\chi' n} \|\wt{\vmod}_0\|_{L^2}^2
.
\end{equation}
Next, we prove that $\P(\limsup_{n} A_n)=0$ where 
$$
A_n \coleq 
\Big\{\om\in \O\,:\, \sup_{t\in [n,n+1]}\|\wt{\vmod}(t,\om)\|_{L^2}^2>e^{-2\chi n} \|\wt{\vmod}_0\|_{L^2}^2 \Big\}.
$$
The above follows from the Borel--Cantelli lemma and Chebyshev's inequality,
\begin{equation*}
\sum_{n\geq 1} \P(A_n)
\leq \sum_{n\geq 1} \frac{e^{2\chi n}}{\|\wt{\vmod}_0\|_{L^2}^2} \E\big(\sup_{t\in [n,n+1]}\|\wt{\vmod}(t,\om)\|_{L^2}^2 \big) 
\leq  C_0 \sum_{n\geq 1} e^{-2(\chi'-\chi)n}<\infty.
\end{equation*}
For $\om\not\in \limsup_{n} A_n$, let 
\begin{equation}
\label{eq:definition_N_Borel_Cantelli_argument}
N(\om) \coleq 1+\sup\Big\{n \geq 1\,:\,\sup_{t\in [n,n+1]}\|\wt{\vmod}(t,\om)\|_{L^2}^2>e^{-2\chi n} \|\wt{\vmod}_0\|_{L^2}^2 \Big\} \ \  \text{ with }\ \  \sup\emptyset \coleq 0,
\end{equation}
while $N(\om) \coleq \infty$ otherwise. From the above construction and $\P(A_n)\leq C_0 e^{-2(\chi'-\chi)n}$, for $k\geq 2$,
\begin{equation*}
\P(N\geq k)
= \P\big( \textstyle{\bigcup_{ n\geq {k-1}}^\infty}\, A_n\big)\leq C_0\sum_{n\geq k-1} e^{-2(\chi'-\chi)n}\lesssim e^{-2(\chi'-\chi)k}.
\end{equation*}
From the above and \eqref{eq:choice_chi_chiN}, it readily follows that 
\begin{equation*}
\E [e^{b(\chi + \log C_0) N}]= \E[C_0^{bN} e^{\chi b N}]<\infty.
\end{equation*} 
From \eqref{eq:definition_N_Borel_Cantelli_argument}, on the one hand, for all $n\geq 1$, $t\in [n,n+1]$, and $n\geq N$, it holds that
$$
\|\wt{\vmod}(t)\|_{L^2}^2\leq\sup_{t\in [n,n+1]} \|\wt{\vmod}(t)\|_{L^2}^2\leq C_0 \|\wt{\vmod}(n)\|_{L^2}^2 
\leq C_0 e^{-2\chi n}  \|\wt{\vmod}_0\|_{L^2}^2\leq C_0e^{2\chi (N+1)} e^{-2\chi t }\|\wt{\vmod}_0\|_{L^2}^2
$$ 
a.s. On the other hand, for $0\leq t\leq N$,
$$
\sup_{t\leq N}\|\wt{\vmod}(t)\|_{L^2}^2\leq C_0^N \|\wt{\vmod}_0\|_{L^2}^2\leq 
C_0^Ne^{2\chi N}e^{-2\chi t} \|\wt{\vmod}_0\|_{L^2}^2
$$
a.s.
Letting 
$
D(\om)=C_0^{N(\om)} e^{\chi (N(\om)+1)},
$ 
the above argument yields \eqref{eq:enhanced_dissipation_reaction1}--\eqref{eq:enhanced_dissipation_reaction2} with $\mu\equiv \infty$ and the solution $v$ to the stochastic RDEs \eqref{eq:reaction_diffusion_cut_off_Lp} with cutoff. 
As mentioned at the beginning of the proof, the latter leads to the claim in the absence of the cutoff, as in the proof of Theorem \ref{t:global} in Subsection 
\ref{ss:proof_global_well_posedness}.
\end{proof}

Before going into the proofs of Propositions \ref{prop:quenched_estimate_w} and \ref{prop:iteration_lemma}, let us point out that, in the case $b=2$, the above argument below \eqref{eq:decay_estimate} can be simplified. Indeed, in the latter situation, the claim of Theorem \ref{t:enhanced_dissipation_reaction} follows with a random constant $D$ given by
$$
\max_{n\geq 0}\, e^{\chi n} \|\wt{\vmod}(n)\|_{L^2}/\|\wt{\vmod}_0\|_{L^2}.
$$

\subsection{Uniform bounds with respect to the noise coefficients}
In this subsection, we prove estimates needed to prove Propositions \ref{prop:quenched_estimate_w} and \ref{prop:iteration_lemma}, and to ensure uniform in $\theta$ estimates on the half-line for solutions to \eqref{eq:reaction_diffusion_cut_off_Lp}. The following is a sharper version of \cite[Theorem 4.1(2)]{A22}, which is suited for our purposes.

\begin{proposition}[$\theta$-uniform estimates with cutoff]
\label{prop:global_cut_off_entropy}
Fix a compact set $K\subseteq \q \R^\ell_{> 0}$, $N\geq1$, $R>0$, and assume that $q,p\in (2,\infty)$ and $\delta\in (1,2)$ satisfy 
$$
q>\frac{d(h-1)}{2}\vee \frac{d}{d-\delta} \quad \text{ and }\quad p>\frac{2}{2-\delta}\vee q.
$$ 
Set $\a_{p,\delta}=p(1-\frac{\delta}{2})-1$. 
Then, there exist constants $r_0,L_0>0$ and $q_*\in (\frac{d(h-1)}{2}\vee 2,q)$ depending only on $q$, $h$, $d$, $p$, $\delta$, $N$, $R$, and $K$ such that for all $r\geq r_0$, normalized radially symmetric coefficients $\theta$ (see \eqref{eq:theta_normalized_symmetric}), and initial data $v_0\in L^q(\T^d;\R^\ell_{\geq 0})$ subject to $\|v_0\|_{L^q}\leq N$ and $\q\overline{v_0}\in K$, the global $(p,\a_{p,\delta},\delta,q)$-solution $\vmod$ to \eqref{eq:reaction_diffusion_cut_off_Lp} satisfies, 
\begin{align}
\label{eq:global_cut_off_estimate1}
\sup_{t>0}\|\vmod(t)-\veq\|_{L^q(\T^d;\R^{\ell})}^q
+\max_{1\leq i\leq \ell} \int_0^\infty\int_{\T^d}(1+ |\vmodi-\veqi|^{q-2})|\nabla \vmodi|^2 \,\dd x\,\dd t
\leq L_0.
\end{align}
\end{proposition}

The key point in the above is the independence of $L_0$ on the noise coefficient $\theta$. 
As in Theorem \ref{t:small_implies_global}, the regularity of a $(p,\a_{p,\delta},\delta,q)$-solution $v$ (similar to \eqref{eq:global_cut_off1} and Definition \ref{def:solution}) is in general not sufficient to ensure the finiteness of the left-hand side of \eqref{eq:global_cut_off_estimate1}. The latter is also part of the proof of Proposition \ref{prop:global_cut_off_entropy}.
Finally, we point out that, as
$$
\int_{\T^d} |\vmodi-\veqi|^{q-2}|\nabla \vmodi|^2\,\dd x
 \eqsim_q\int_{\T^d}  |\nabla [|\vmodi-\veqi|^{q/2}]|^2\,\dd x  ,
$$ 
and $\|\vmodi-\veqi\|_{L^{\xi}}^q
\eqsim_q \big\||\vmodi-\veqi|^{q/2}\big\|_{L^{(2\xi)/q}}^2$, it follows from Sobolev embeddings and \eqref{eq:global_cut_off_estimate1} that, for all $0\leq s <t<\infty$ and $i\in \{1,\dots,\ell\}$,
\begin{align}
\label{eq:v_vinfty_embedding_apriori_estimate}
\|\vmodi-\veqi\|_{L^{q}(s,t;L^{\xi}(\T^d))}^q
&\lesssim_{q,\ell,\xi}\|\vmodi-\veqi \|_{L^{q}(s,t;L^{q}(\T^d))}^q
+ 
\big\|\nabla |\vmodi-\veqi|^{q/2}\big\|_{L^2(s,t;L^2(\T^d))}^{2}
\\
\nonumber
&
\lesssim_{\ell,q}  L_0(1+t-s),
\end{align}
where $\xi\in (1,\infty)$ is arbitrary if $d=2$, or $\xi=\frac{qd}{d-2}$ otherwise.

\smallskip

The key for the proof of Proposition \ref{prop:global_cut_off_entropy} is the following result. 
The following can be seen as a `subcritical' variant of Lemma \ref{l:nonlinear_estimate_easy} and an improved version of \cite[Lemma 4.5]{A22} (see the discussion below Proposition \ref{prop:LWP}).

\begin{lemma}[Nonlinear estimates in the $L^q$-setting II] 
\label{l:interpolation_subcriticality_apriori_estimates}
Let $r\in (1,\infty)$ and fix $q>\frac{d(h-1)}{2}\vee 2$. Set $\delta=q-\frac{d}{2}(h-1)$ and assume that $r<\infty$ satisfies
\begin{equation}
\label{eq:lower_bound_r_interpolation_final}
r> \frac{q}{\delta} (\delta+h-1).
\end{equation}
Then, there exist constants $C_0>0$, $q_*<q$, and $\eta_0<\frac{1}{q+h-1}$ independent of $t>0$ such that
\begin{align}
\label{eq:interpolation_subcriticality_apriori_estimates1}
\|u\|_{L^{q+h-1}((0,t)\times \T^d)}&\leq C_0\|u\|_{L^{q+h-1}(0,t;L^1)}\\
\nonumber
&+ C_0 \|u\|_{L^r(0,t;L^{q_*})}^{1-q \,\eta_0}\Big(\int_0^t \int_{\T^d} |u|^{q-2}|\nabla u|^2 \,\dd x\,\dd s \Big)^{\eta_0}
\end{align}
for all $u\in L^2_{\loc}([0,\infty),W^{1,q})\cap L^\infty_{\loc}([0,\infty);L^q)$.
\end{lemma}

Similar to Lemma \ref{l:subcriticality_Lq}, the subcriticality of the $L^q$-space with $q>\frac{d(h-1)}{2}$ is encoded in the condition $\eta_0<1/(q+h-1)$, which ensures that the nonlinearity $f_i(v)$ is sublinear w.r.t.\ $\int_0^t \int_{\T^d} |u|^{q-2}|\nabla u|^2 \,\dd x\,\dd s$; see the proof of \eqref{eq:global_cut_off_estimate1} below.
As in \eqref{eq:finiteness_of_norms_interpolation}, the assumed regularity of $u$ in Lemma  \ref{l:interpolation_subcriticality_apriori_estimates} ensures that the right-hand side of \eqref{eq:interpolation_subcriticality_apriori_estimates1} is finite. 

\begin{proof}
As in \cite[Lemma 4.5]{A22}, the proof is based on an interpolation argument. Some care is needed as we want to obtain estimates that are uniform in $t>0$. The key idea is to interpolate the energy space and to use the scaling of the Lebesgue spaces. More precisely, 
\begin{align*}
&\|u\|_{L^{q+h-1}((0,t)\times \T^d)}= \big\||u|^{q/2}\big\|_{L^{\frac{2}{q}(q+h-1)}((0,t)\times \T^d)}^{2/q}\\
&\leq C_q \Big\||u|^{q/2}-\int_{\T^d} |u|^{q/2}\,\dd x \Big\|_{L^{\frac{2}{q}(q+h-1)}((0,t)\times \T^d)}^{2/q}
+C_q\|u\|_{L^{q+h-1}(0,t;L^{q/2})}\\
&\leq C_q \Big\||u|^{q/2}-\int_{\T^d} |u|^{q/2}\,\dd x \Big\|_{L^{\frac{2}{q}(q+h-1)}((0,t)\times \T^d)}^{2/q}
+C_q'\|u\|_{L^{q+h-1}(0,t;L^{1})}+ \frac{1}{2}
\|u\|_{L^{q+h-1}((0,t)\times \T^d)}.
\end{align*}
Thus, it remains to interpolate the first term on the right-hand side of the above estimate. 
For the latter, recall that, by standard interpolation inequalities,
\begin{equation}
\label{eq:interpolation_inequality_proof_uniformtheta0}
L^{(2r)/q}(0,t;\dot{L}^{(2q_*)/q}(\T^d))\cap L^2(0,t;\dot{H}^1(\T^d))\embed L^{\frac{2}{q}(q+h-1)}(0,t;\dot{H}^{\vartheta,\xi}(\T^d)),
\end{equation}
where $\vartheta\in (0,1)$ and $\xi\in (1,2)$ satisfy 
\begin{align}
\label{eq:interpolation_subcriticality_apriori_estimates}
\frac{q(1-\vartheta)}{2r}+\frac{\vartheta}{2}=\frac{q}{2(q+h-1)}
\qquad \text{ and }\qquad 
\frac{q(1-\vartheta)}{2q_*}+\frac{\vartheta}{2}=\frac{1}{\xi}.
\end{align}
For the interpolation of Bessel potential spaces with different integrability, see, e.g., \cite[Theorem 6.4.5(7)]{BeLo}.
In particular, 
\begin{equation}
\label{eq:vartheta_interpolation_qstar_q_conditions_proof}
\vartheta = \frac{q}{r-q} \frac{r-(q+h-1)}{q+h-1} \qquad \text{ and }\qquad \lim_{q_*\uparrow q}\xi=2.
\end{equation}
We emphasize that the constant in the embedding \eqref{eq:interpolation_inequality_proof_uniformtheta0} can be made uniform in $t>0$ as the first relation in \eqref{eq:interpolation_subcriticality_apriori_estimates} holds as an equality.

Next, we claim that there exists some $q_*<q$, for which 
$$
H^{\vartheta,\xi}(\T^d) \embed L^{\frac{2}{q}(q+h-1)}(\T^d).
$$
Indeed, by the second identity in \eqref{eq:vartheta_interpolation_qstar_q_conditions_proof} and Sobolev embeddings, it is enough to check that 
$$
\vartheta-
\frac{d}{2}>-\frac{qd}{2(q+h-1)}\qquad \Longleftrightarrow \qquad \vartheta>\frac{d}{2}\frac{h-1}{q+h-1}.
$$
Using \eqref{eq:vartheta_interpolation_qstar_q_conditions_proof}, standard computations show that the latter condition is equivalent to \eqref{eq:lower_bound_r_interpolation_final}.
Thus, it follows from the above and \eqref{eq:interpolation_inequality_proof_uniformtheta0} that
\begin{align*}
\Big\||u|^{q/2}-\int_{\T^d} |u|^{q/2}\,\dd x \Big\|_{L^{\frac{2}{q}(q+h-1)}((0,t)\times \T^d)}^{2/q}
&\lesssim\Big\||u|^{q/2}-\int_{\T^d} |u|^{q/2}\,\dd x\Big\|_{L^{\frac{2r}{q}}(0,t;L^{(2q_*)/q})}^{(2(1-\vartheta))/q}
 \|\nabla [|u|^{q/2}]\|_{L^{2}(0,t;L^{2})}^{(2\vartheta)/q}\\
&\lesssim \|u\|_{L^{r}(0,t;L^{q_*})}^{1-\vartheta}
 \Big(\int_0^t \int_{\T^d} |u|^{q-2}|\nabla u|^2 \,\dd x \,\dd s \Big)^{\vartheta/q}.
\end{align*}
The claimed inequality \eqref{eq:interpolation_subcriticality_apriori_estimates1} can now be obtained by collecting the above estimates and choosing $\eta_0=\vartheta/q$. Clearly, $\eta_0<1/(q+h-1)$ as $h>1$ by assumption.
\end{proof}

With Lemma \ref{l:interpolation_subcriticality_apriori_estimates} at our disposal, we can conclude the proof of Proposition \ref{prop:global_cut_off_entropy}.

\begin{proof}[Proof of Proposition \ref{prop:global_cut_off_entropy}]
From \eqref{eq:global_cut_off_continuity1}--\eqref{eq:global_cut_off_continuity2} in Proposition \ref{prop:global_cut_off}, by mollification, it suffices to obtain \eqref{eq:global_cut_off_estimate1} for $v_0\in C^{\gamma}(\T^d;\R^{\ell}_{\geq 0})$ for some $\gamma\in (0,1)$ with constants independent of $v_0$ and $\g$. 
In particular, the corresponding $(p,\a_{p,\delta},\delta,q)$-solution $\vmod$ is also a $(p,\kappa,1,q)$-solution for any $p\in (2,\infty)$ and $\a\in [0,\frac{p}{2}-1)$, and it possesses the additional regularity 
\begin{align}
\label{eq:regularity_smoothed1}
\vmod&\in C([0,\infty);L^q(\T^d;\R^\ell)), \\
\label{eq:regularity_smoothed2}
\vmod&\in L^p_{\loc}([0,\infty),w_{\a};W^{1,q}(\T^d;\R^\ell))\subseteq L^2_{\loc}([0,\infty);W^{1,q}(\T^d;\R^\ell)). 
\end{align}
This will be used below without further mention. 

Without loss of generality, we assume that $r\geq r_0\vee (q+h-1)$ and $q_*$ are chosen so that the estimate in Lemma \ref{l:interpolation_subcriticality_apriori_estimates} holds.
Set 
$$
\vmodhat=\vmod-v_\infty.
$$ 
Since $f_i(v_\infty)=0$ for all $i\in \{1,\dots,\ell\}$, $\vmodhat$ is a global $(p,\kappa,1,q)$-solution to 
\begin{equation}
\left\{
\begin{aligned}
\dd \vmodhati 
& =\big[ \nu_i \Delta \vmodhati+ \wh{\phi}_{R,r}(\vmodhat) f_i(\vmodhat+v_\infty)\big]\,\dd t 
+\sqrt{c_d \nu }\sum_{k,\alpha} \theta_k (\sigma_{k,\alpha}\cdot\nabla) \vmodhati\circ \dd W^{k,\alpha}_t
& \text{on }&\Tor^d,\\ 
\vmodhat_{i,0}&\coleq v_i(0)-v_{\infty,i}  &\text{on }&\Tor^d,
\end{aligned}\right.
\end{equation}
where 
$
\wh{\phi}_{R,r}(t,\vmodhat)
=\phi\big(R^{-1}\,[\|\vmodhat\|_{L^r(0,t;L^{q_*})}+\|\vmodhat\|_{L^2(0,t;L^2)}]\big)
$
and $\phi$ is as in \eqref{eq:def_cutoff}.

\smallskip

As in the proof of Lemma \ref{l:Lq_estimate_small}, from the It\^o formula (\cite[Section 3]{Kry13} or \cite[Appendix A]{DHV16}) it follows that, a.s.\ for all $t>0$,
\begin{align}
\label{eq:ito_formula_what_final_estimates}
\|\vmodhat_i(t)\|_{L^q}^q
&+q(q-1) \nu_i \int_{0}^t \int_{\T^d} |\vmodhati|^{q-2} |\nabla \vmodhati|^2\,\dd x \, \dd t' \\
\nonumber
&=\|\vmodhat_{0,i}\|_{L^q}^q
+ q\int_{0}^t \wh{\phi}_{R,r}(\vmodhat) \int_{\T^d} |\vmodhat_i|^{q-2} \vmodhati f_i (\vmodhat+v_\infty)\, \dd x \,\dd t' .
\end{align}
Since $\veq$ is an equilibrium, one has $f_i(v_\infty)=0$ and 
\begin{align*}
|f_i(\vmodhat+v_\infty)- f_i(v_\infty)|\leq C (1+|\vmodhat+v_\infty|^{h-1}+|v_\infty|^{h-1})|\vmodhat|
\leq C (|\vmodhat|+|\vmodhat|^{h}),
\end{align*}
where $C$ depends only on $N$, $q$, and $K$ as $\sup_{\q v_\infty\in K} |v_\infty|<\infty$ by Lemma \ref{l:continuous_cinfty}. 
Thus, by Young's inequality, 
\begin{align*}
\Big| \int_{\T^d} |\vmodhati(t)|^{q-2} \vmodhati(t) f_i (\vmodhat(t)+v_\infty)\, \dd x  \Big|
&\leq C \int_{\T^d} (|\vmodhat(t)|^{q} +  |\vmodhat(t)|^{q+h-1})\,\dd x  \\
&\leq C_q   \int_{\T^d}  \big( |\vmodhat(t)|^{2}+  |\vmodhat(t)|^{q+h-1}\big)\,\dd x 
\end{align*}
a.s.\ for all $t>0$.
Let $\chi_{R,r}$ be the stopping time given by
$$
\chi_{R,r} \coleq \inf\{t\in [0,\infty)\,:\, \|\vmodhat\|_{L^r(0,t;L^{q_*})}+\|\vmodhat\|_{L^2(0,t;L^{2})}\geq 2R\}  \quad \text{ with } \quad  \inf\emptyset \coleq \infty.
$$
Lemma \ref{l:interpolation_subcriticality_apriori_estimates} implies, a.s.\ for all $t>0$, 
\begin{align*}
&\Big|\int_{0}^t \wh{\phi}_{R,r}(\vmodhat) \int_{\T^d} |\vmodhati|^{q-2} \vmodhati f_i (\vmodhat+v_\infty)\, \dd x \,\dd t'\Big|\\
& \lesssim_{q,N} \int_0^{\chi_{R,r}\wedge t} \int_{\T^d}  |\vmodhat|^{2}\,\dd x\,\dd t'  
+ \max_{1\leq i \leq \ell} \int_0^{\chi_{R,r}\wedge t} \int_{\T^d}  |\vmodhati|^{q+h-1}\,\dd x\,\dd t' \\
& \lesssim \|\vmodhat\|_{L^2(0,\chi_{R,r}\wedge t;L^2)}^2  + \|\vmodhat\|_{L^{q+h-1}(0,\chi_{R,r}\wedge t;L^2)}^{q+h-1}\\
&+\|\vmodhat\|_{L^r(0,\chi_{R,r}\wedge t;L^{q_*})}^{(q+h-1)\alpha_0}
\Big(\max_{1\leq i \leq \ell} \int_0^{\chi_{R,r}\wedge t} \int_{\T^d} |\vmodhati|^{q-2}|\nabla \vmodhati|^2 \,\dd x\,\dd t' \Big)^{(q+h-1)\eta_0},
\end{align*}
where $\alpha_0=1-q\,\eta_0$.
From the definition of $\chi_{R,r}$, it follows that $\|\vmodhat\|_{L^2(0,\chi_{R,r};L^2)}\leq 2R$ and $\|\vmodhat\|_{L^r(0,\chi_{R,r};L^{q_*})}\leq2R$. 
Moreover, as we assumed $r\geq q+h-1>2$, it also follows from interpolation and the definition of stopping time $\chi_{R,r}$ that 
$$
\|\vmodhat\|_{L^{q+h-1}(0,\chi_{R,r}\wedge t;L^2)}^{q+h-1}
\lesssim \|\vmodhat\|_{L^{2}(0,\chi_{R,r}\wedge t;L^2)}^{q+h-1}
+\|\vmodhat\|_{L^{r}(0,\chi_{R,r}\wedge t;L^2)}^{q+h-1}\lesssim_R 1.
$$
Hence, the above observations and the fact that $\eta_0<1/(q+h-1)$ yield
\begin{align*}
&\Big|\int_{0}^t \wh{\phi}_{R,r}(\vmodhat) \int_{\T^d} |\vmodhati|^{q-2} \vmodhati f_i (\vmodhat+v_\infty)\, \dd x \,\dd t' \Big|\\
&\qquad \quad \leq C_{q,N,R} + \frac{\nu_0}{2}(q-1)\max_{1\leq i\leq \ell}\int_0^{t } \int_{\T^d} |\vmodhati|^{q-2}|\nabla \vmodhati|^2 \,\dd x\,\dd t',
\end{align*}
where $\nu_0=\min_{1\leq i\leq \ell}\nu_i$. The previous displayed formula and \eqref{eq:ito_formula_what_final_estimates} readily imply that the right-hand side of \eqref{eq:ito_formula_what_final_estimates} is bounded uniformly in $t>0$. To conclude the proof of Proposition \ref{prop:global_cut_off_entropy}, it suffices to perform the above argument again, with $q$ replaced by $2$ to bound $\int_0^t \int_{\T^d} |\nabla \vmodi |^2\,\dd x \,\dd t' $ (see Lemma \ref{l:L2_estimates_small_implies} for a similar situation).
\end{proof}

\subsection{Proof of Proposition \ref{prop:quenched_estimate_w}}
\label{ss:proof_quenched_estimate_w}
Before going into the proof of Proposition \ref{prop:quenched_estimate_w}, let us note that $\wt{\vmod}(t,x)$ is a global $(p,\a_{p,\delta},\s,q)$-solution to (defined analogously to Definition \ref{def:solution})
\begin{equation}
\label{eq:reaction_diffusion_cut_off_Lp_w}
\left\{
\begin{aligned}
\dd \wt{\vmod}_i - \nu_i \Delta \wt{\vmod}_i \,\dd t
& = \phi_{R,r}(v_\infty,\vmod) (f_i(\vmod)-\overline{f_i(\vmod)})\,\dd t \\
&\qquad \quad +\sqrt{c_d \nu }\sum_{k,\alpha} \theta_k (\sigma_{k,\alpha}\cdot\nabla) \wt{\vmod}_i\circ \dd W^{k,\alpha}_t
& \text{ on }&\Tor^d,\\ 
\wt{\vmod}_i(0)&=\wt{v}_{0,i} &\text{ on }&\Tor^d.
\end{aligned}\right.
\end{equation}

\begin{proof}[Proof of Proposition \ref{prop:quenched_estimate_w}]
By the It\^o formula, we have, a.s.\ for all $0<s<t<\infty$,
\begin{align*}
\frac{1}{2}\|\wt{\vmod}_i(t)\|_{L^2}^2+ \nu_i \int_s^t \|\nabla \wt{\vmod}_i\|_{L^2}^2\,\dd t'
&= \frac{1}{2}\|\wt{\vmod}_i(s)\|_{L^2}^2 + \int_s^t \phi_{R,r}(v_\infty,\vmod)\int_{\T^d} (f_i(\vmod)-\overline{f_i(\vmod)}) \wt{\vmod}_i\,\dd x \,\dd r\\
&\leq\frac{1}{2} \|\wt{\vmod}_i(s)\|_{L^2}^2 + \int_s^t \Big|\int_{\T^d} (f_i(\vmod)-\overline{f_i(\vmod)}) \wt{\vmod}_i\,\dd x \Big|\, \dd r.
\end{align*}
Fix $r\in (1,2)$ to be decided later if $d=2$, otherwise $r=\frac{2d}{d+2}$. From the Sobolev embedding $L^r\embed H^{-1}$, it follows that, a.s.\ for all $t>0$, 
\begin{align}
\label{eq:estimate_f_fbar_mixing}
\Big|\int_{\T^d} (f_i(\vmod)-\overline{f_i(\vmod)})\,\wt{\vmod}_i \,\dd x \Big|
&\leq \|f_i(\vmod)-\overline{f_i(\vmod)}\|_{L^2} \|\wt{\vmod}_i\|_{L^2}\\
\nonumber
&\stackrel{(i)}{\eqsim} \|\nabla [f_i(\vmod)]\|_{H^{-1}}\|\wt{\vmod}_i\|_{L^2}\\
\nonumber
&\stackrel{(ii)}{\lesssim} \|\nabla [f_i(\vmod)]\|_{L^r}\|\wt{\vmod}_i\|_{L^2}\\
\nonumber
&\stackrel{(iii)}{\lesssim} \|(1+|\vmod|^{h-1})|\nabla \wt{\vmod}|\|_{L^r}\|\wt{\vmod}_i\|_{L^2},
\end{align}
where we used in $(i)$ that $f_i(\vmod)-\overline{f_i(\vmod)}$ has mean zero and, by standard Fourier methods,
$$
\|u\|_{L^2}\eqsim \|\nabla u\|_{H^{-1}}\quad \text{ for all }u\in L^2,
$$
in $(ii)$ the Sobolev embeddings, and in $(iii)$ that $\nabla \vmod=\nabla \wt{\vmod}$. Hence, a.s.\ for all $t>0$, 
\begin{align*}
\Big|\int_{\T^d} (f_i(\vmod)-\overline{f_i(\vmod)})\,\wt{\vmod}_i \,\dd x \Big|
&\lesssim
(1+\|\vmod\|_{L^{r_0(h-1)}}^{h-1})\|\nabla \wt{\vmod}\|_{L^2}\|\wt{\vmod}_i\|_{L^2}\\
&\leq \frac{\min_i \nu_i}{2} \|\nabla \wt{\vmod}\|_{L^2}^2+ C_{\nu_i}
(1+\|\vmod\|_{L^{r_0(h-1)}}^{2(h-1)})\|\wt{\vmod}_i\|_{L^2}^2\\
&\leq \frac{\min_{1\leq i\leq \ell}\nu_i}{2\ell} \|\nabla \wt{\vmod}\|_{L^2}^2+ C_{\nu_i,K}
(1+\|\vmod-v_\infty\|_{L^{r_0(h-1)}}^{2(h-1)})\|\wt{\vmod}_i\|_{L^2}^2,
\end{align*}
where $r_0\in (1,\infty)$ satisfies $\frac{1}{r_0}+\frac{1}{2}=\frac{1}{r}$, and where we used $\sup_{\q\veq \in K}|v_\infty|<\infty$ by Lemma \ref{l:continuous_cinfty}. Thus, $r_0=d$ if $d\geq 3$, or $r_0\downarrow 2$ if $d=2$, and $r\downarrow 1$.
From the above estimate and the Gr\"onwall lemma, it follows that, a.s.\ for all $t>s$,
\begin{equation}
\label{eq:Gronwall_Application_wt}
\|\wt{\vmod}(t)\|_{L^2}^2+ \int_s^t \|\nabla \wt{\vmod}\|_{L^2}^2\,\dd t'
\leq C e^{C\int_s^t (1+\|\vmod-v_\infty\|_{L^{r_0(h-1)}}^{2(h-1)})\,\dd t' }\|\wt{\vmod}(s)\|_{L^2}^2.
\end{equation}
Since $|t-s|\leq 1$, the bound  
\begin{equation}
\label{eq:boundedness_Lr0_proof}
\int_s^t (1+\|\vmod-v_\infty\|_{L^{r_0(h-1)}}^{2(h-1)})\,\dd t' \,\lesssim_{q,N,K} 1
\end{equation} 
follows from Proposition \ref{prop:global_cut_off_entropy} and \eqref{eq:v_vinfty_embedding_apriori_estimate} together with the embedding in Lemma \ref{l:interpolation_Lr0} below.
\end{proof}

\begin{lemma}
\label{l:interpolation_Lr0}
Suppose that $q>\frac{d(h-1)}{2}\vee 2$ and additionally $q>(d-2)(h-1)$ in case $d\geq 5$.
Then, 
\begin{equation}
\label{eq:interpolation_inequality_checking}
L^{q}(0,1;L^{\xi})\cap L^\infty(0,1;L^q)\embed L^{2(h-1)}(0,1;L^{r_0(h-1)})
\end{equation}
provided $r_0=d$, $\xi=\frac{dq}{d-2}$, and $d\geq 3$. Finally, in the case $d=2$, there exist $r_0\in (2,\infty)$ sufficiently small and $\xi\in [1,\infty)$ sufficiently large such that \eqref{eq:interpolation_inequality_checking} holds. 
\end{lemma}

\begin{proof}
To prove \eqref{eq:interpolation_inequality_checking}, we distinguish the following cases:
\begin{itemize}
\item \emph{Case $q\geq 2(h-1)$}. In this case, it suffices to check $\xi\geq r_0(h-1)$ with $r_0$ as above. As $\xi\in [2,\infty)$ can be chosen arbitrarily large if $d=2$, it is enough to discuss the cases $d\geq 3$. Hence, $r_0=d$ and $\xi=\frac{dq}{d-2}\geq r_0(h-1)$ is equivalent to $q>(d-2)(h-1)$. As $q>\frac{d}{2}(h-1)$ by assumption, the latter is automatically satisfied if $d\leq 4$.
\item \emph{Case $q<2(h-1)$}. We begin with the case $d\geq 3$. By interpolation, the space on the left-hand side of \eqref{eq:interpolation_inequality_checking} embeds into $L^{2(h-1)}(0,t;L^{\zeta})$ where 
$
\frac{1-\varphi}{q}+\frac{\varphi}{\xi}=\frac{1}{\zeta} $ and 
$\varphi=\frac{q}{2(h-1)}.
$
Thus, $\zeta\geq d(h-1)$ if and only if 
$$
\frac{1}{\zeta}=\frac{1}{q}-\frac{1}{d(h-1)}\leq \frac{1}{d(h-1)}.
$$
The previous estimate holds as a strict inequality as we are assuming $q>\frac{d(h-1)}{2}\vee 2$. 
If $d=2$, then the previous argument follows similarly by noticing that $\lim_{\xi\to \infty}\frac{1}{\zeta}=\frac{1}{q}-\frac{1}{2(h-1)}<\frac{1}{2(h-1)}$ as $q>(h-1)\vee 2$ by assumption. Thus, the existence of $r_0\in (2,\infty)$ and $\xi<\infty$, for which $\zeta\geq r_0(h-1)$ holds, follows by continuity.
\end{itemize}
This finishes the proof of Lemma \ref{l:interpolation_Lr0}. 
\end{proof}

We conclude this subsection with the following observations.  

\begin{remark}\
\label{rem:criticality_pathwise_estimates}
\begin{itemize}
\item
The proof of Proposition \ref{prop:quenched_estimate_w} breaks down if $\overline{f(\vmod)}$ in \eqref{eq:reaction_diffusion_cut_off_Lp_w} is replaced by $f(v_\infty)$. This is consistent with the discussion in \eqref{eq:linear_PDE_intro}--\eqref{eq:oscillation_energy}. Indeed, transport noise cannot, in general, arbitrarily accelerate convergence to equilibrium, whereas it can arbitrarily enhance the decay of spatial fluctuations.
\item
It is worth noting that the space on the right-hand side of \eqref{eq:interpolation_inequality_checking} has critical scaling. Indeed, its space-time Sobolev index is
$
-\frac{2}{2(h-1)}-\frac{d}{d(h-1)}
=-\frac{2}{h-1},
$
which coincides with that of the critical space $L^{\frac{d}{2}(h-1)}$.
\end{itemize}
\end{remark}

\subsection{Proof of Proposition \ref{prop:iteration_lemma}}
\label{ss:proof_iteration_lemma}
We begin by collecting some facts. From the It\^o--Stratonovich correction \eqref{eq:Ito_stratonovich_change} and Definition \ref{def:solution}, it follows that $v$ satisfies
\begin{align}
\label{eq:mild_formulation_SPDE_wtv}
\wt{\vmod}_i(t)
&= e^{(\nu_i+\nu)(t-s)\Delta}\wt{\vmod}_i(s)+\int_s^t e^{(\nu_i+\nu)(t-t')\Delta} \big[ \phi_{R,r}(v_\infty,\vmod)\big( f_i(\vmod(t'))-\overline{f_i(\vmod(t'))}\big)\big]\,\dd t' \\
\nonumber
&+\sqrt{c_d \nu} \sum_{k,\alpha}\theta_k \int_s^t  e^{(\nu_i+\nu)(t-t')\Delta} [(\sigma_{k,\alpha}\cdot\nabla)\wt{\vmod}_i(t')]\,\dd W^{k,\alpha}_{t'}
\end{align}
a.s.\ for all $0\leq s<t<\infty$ and $i\in \{1,\dots,\ell\}$. 

\smallskip

For the proof of Proposition \ref{prop:iteration_lemma}, we need suitable estimates for the deterministic and stochastic convolutions appearing on the right-hand side of the identity \eqref{eq:mild_formulation_SPDE_wtv}. We start by looking at deterministic convolutions. As it will be clear in the proof of Proposition \ref{prop:iteration_lemma} below, for the deterministic convolutions, it is crucial to allow for $L^1$-integrability in time.

\begin{lemma}
\label{l:smr_L2_L1H}
Let $\mu>0$ and $f\in L^1(a,b;H^{-1})$ for some $0\leq a<b<\infty$. Suppose that $\langle f,\one_{\T^d}\rangle =0$ a.e.\ on $(a,b)$. Then, there exists a constant $C$ independent of $a$, $b$, $\mu$, and $f$ such that
$$
\int_a^b\Big\| \int_a^t e^{\mu (t -s)\Delta} f(s)\,\dd s\Big\|_{L^2}^2\,\dd t \leq \frac{C}{\mu}\Big(\int_a^b \|f(s)\|_{H^{-1}}\,\dd s \Big)^2.
$$
\end{lemma}

The proof is well-known to experts and can be thought of as a maximal regularity estimate for the heat equation in the usual variational setting (see, e.g., \cite[Theorem 2.1, p.\ 31]{pardoux1975equations} or \cite[Theorem 3.4]{TV_large}). As in our situation, it is important to keep track of the dependence on the diffusivity $\mu$, we include a brief proof here.

\begin{proof}
By density, it is enough to show the claimed estimate in the case $f\in C^1([a,b];L^2)$ with zero mean. 
In this case, it is well-known that $t\mapsto w(t) \coleq \int_a^t e^{\mu \Delta(t-s) }f(s)\,\dd s$ is continuously differentiable on $(a,b)$ and satisfies 
$$
\frac{\dd }{\dd t}\|I_{-1}w(t)\|_{L^2}^2 =-2\mu\| I_{-1} \nabla w(t)\|_{L^2}^2
+ \int_{\T^d} [I_{-1} f(t) ][I_{-1} w(t)]\,\dd x
$$
a.e.\ on $(a,b)$, where $I_{-1} \coleq (-\Delta)^{-1/2}$. Using $w(a)=0$ and integrating the above identity over $(a,b)$, we obtain
\begin{align*}
\sup_{a<t<b}\Big(\| I_{-1}w(t)\|_{L^2}^2 
+ 2 \mu\int_a^t \| I_{-1} \nabla w(s)\|_{L^2}^2\,\dd s \Big)
\leq \sup_{a<t<b}\| I_{-1} w(t)\|_{L^2} \int_a^b \| I_{-1} f(s)\|_{L^2}\,\dd s &\\
\leq \frac{1}{2}
\sup_{a<t<b}\| I_{-1} w(t)\|_{L^2}^2 + \frac{1}{2}\Big(\int_a^b \| I_{-1} f(s)\|_{L^2}\,\dd s\Big)^2&. 
\end{align*}
The claim now follows from standard Fourier methods and $\langle f,\one_{\T^d} \rangle=0$ a.e.\ on $(a,b)$.
\end{proof}

The following gives the required estimate on the stochastic convolution as in \eqref{eq:mild_formulation_SPDE_wtv}.

\begin{lemma}
\label{l:estimate_stoch_convol_linfty}
Let $0\leq a<b<\infty$, $\nu>0$, and $\mu\geq 0$. For a progressively measurable process $f\in L^2(\O\times (a,b);H^1(\T^d))$, let  
$$
\mathcal{Z}_{\mu,\nu}^f(t) \coleq 
\sqrt{c_d \nu} \sum_{k,\alpha}\theta_k \int_a^t  e^{(\nu+\mu)(t-s)\Delta} [(\sigma_{k,\alpha}\cdot\nabla) f]\,\dd W^{k,\alpha}_s.
$$
Then, for any $\g\in (0,\frac{2}{d+2})$, there exists a constant $C_\g$ independent of $a$, $b$, $\mu$, $\nu$, and $f$ such that
\begin{equation}
\label{eq:estimate_stoch_convol_linfty0}
\E\int_a^b \|\mathcal{Z}_{\mu,\nu}^f(t)\|_{L^2}^2\,\dd t 
\leq C_\g \|\theta\|_{\ell^\infty}^{2\g} \,\E\int_{a}^b \|\nabla f(t)\|_{L^2}^2\,\dd t .
\end{equation}
\end{lemma}

The above is an improved version of \cite[Theorem 3.4]{L21}, for which we give a self-contained and simplified proof below. It is worth noticing that the latter result is not directly applicable, as the assumptions there are not satisfied in our situation. 
The reader is also referred to \cite[Section 2]{FGL24_quantitative} and \cite{flandoli2026enhanced} for related results.

\begin{proof}
For simplicity, throughout the proof, we write $\mathcal Z$
instead of $\mathcal Z_{\mu,\nu}^f$. Fix $\gamma\in(0,\frac{2}{d+2})$ and set
$$
\beta=(1-\gamma)/\gamma>d/2.
$$
By interpolation, it suffices to prove the following estimates:
\begin{align}
\label{eq:estimate_stoch_convol_linfty1}
\E\int_a^b \| \mathcal{Z}(s)\|_{H^{1}}^2\,\dd s 
&\lesssim  \E\int_{a}^b \|\nabla f\|_{L^2}^2\,\dd s ,\\
\label{eq:estimate_stoch_convol_linfty2}
\E\int_a^b \|\mathcal{Z}(s)\|_{H^{-\beta}}^2\,\dd s 
&\lesssim_\beta \|\theta\|_{\ell^\infty}^2 \,\E\int_{a}^b \|\nabla f\|_{L^2}^2\,\dd s
\end{align}
with implicit constants independent of $a$, $b$, $\mu$, $\nu$, and $f$. Indeed, since
$
\|g\|_{L^2}
\lesssim
\|g\|_{H^{-\beta}}^\gamma
\|g\|_{H^1}^{1-\gamma},
$
H\"older's inequality on $\Omega\times(a,b)$ shows that \eqref{eq:estimate_stoch_convol_linfty1} and \eqref{eq:estimate_stoch_convol_linfty2} imply \eqref{eq:estimate_stoch_convol_linfty0}.

\smallskip

\emph{Step 1: Proof of \eqref{eq:estimate_stoch_convol_linfty1}.}  It is well-known that $\mathcal{Z}$ satisfies the following linear SPDE with additive noise:
$$
\dd \mathcal{Z} = (\nu+\mu)\Delta \mathcal{Z}\,\dd t + \sqrt{c_d \nu}
\sum_{k,\alpha} \theta_k (\sigma_{k,\alpha}\cdot \nabla) f\,\dd W^{k,\alpha}_t\ \  \text{ on } \ (a,b)
$$
with the initial condition $\mathcal{Z}(a)=0$. Now, by the It\^o formula (see, e.g., \cite[Theorem 4.2.5]{LR15})
\begin{align*}
\frac{1}{2}\E \|\mathcal{Z}(b)\|_{L^2}^2 + (\mu+\nu)\, \E\int_a^b \|\nabla \mathcal{Z}(t)\|_{L^2}^2\,\dd t 
&= c_d \nu \sum_{k,\alpha}\theta_k^2\,\E \int_a^b \|(\sigma_{k,\alpha}\cdot \nabla) f\|_{L^2}^2\,\dd t \\
&\lesssim_d \nu\, \E\int_a^b \|\nabla f\|_{L^2}^2\,\dd t,
\end{align*}
where, in the last step, we used $\|\sigma_{k,\alpha}\|_{L^\infty}\leq 1$ for all $k$, $\alpha$, and $\|\theta\|_{\ell^2}=1$.
Thus, \eqref{eq:estimate_stoch_convol_linfty1} clearly follows from the above noticing that $\overline{\mathcal{Z}} =0$ a.e.\ on $(a,b)\times \O$ as $\nabla \cdot \sigma_{k,\alpha}=0$.
 
\smallskip

\emph{Step 2: Proof of \eqref{eq:estimate_stoch_convol_linfty2}.} It\^o's isomorphism yields, for all $t\in (a,b)$,
\begin{align}
\label{eq:estimate_stoch_convol_linfty20}
\E \|\mathcal{Z} (t)\|_{H^{-\beta}}^2 
&\eqsim_\beta 
\sum_{k,\alpha} \nu\theta_k^2\,\E \int_a^t \|e^{(\nu+\mu)(t-s)\Delta}[(\sigma_{k,\alpha}\cdot \nabla)f]\|_{H^{-\beta}}^2\,\dd s \\
\nonumber
&\lesssim  \nu
\sum_{k,\alpha} \theta_k^2\, \E \int_a^t e^{-4(\nu+\mu)\pi^2 (t-s)} \|(\sigma_{k,\alpha}\cdot \nabla)f\|_{H^{-\beta}}^2\,\dd s ,
\end{align}
where we used again that $\overline{(\sigma_{k,\alpha}\cdot \nabla)f }=0$ as $\nabla \cdot \sigma_{k,\alpha}=0$.

Next, let $(e_j)_{j\in \Z^d}$ be the standard Fourier basis of $L^2(\T^d;\C)$, i.e., $e_j(x)=e^{2\pi \im j\cdot x}$. As $(\sigma_{k,\alpha})_{k,\alpha}$ is an orthonormal system of $L^2(\T^d;\C^d)$, by the Parseval inequality, we have
\begin{align*}
\sum_{k,\alpha} \theta_k^2 \|(\sigma_{k,\alpha}\cdot \nabla)f\|_{H^{-\beta}}^2
&\leq \|\theta\|_{\ell^\infty}^2\sum_{k,\alpha}\sum_{j\in \Z^d}(1+|j|^2)^{-\beta} |(e_j,[(\sigma_{k,\alpha}\cdot\nabla)f ])_{L^2}|^2\\
&= \|\theta\|_{\ell^\infty}^2 \sum_{k,\alpha}\sum_{j\in \Z^d}(1+|j|^2)^{-\beta} |(\sigma_{k,\alpha}, e_j \nabla f )_{L^2} |^2\\
&\leq  \|\theta\|_{\ell^\infty}^2
\sum_{j\in \Z^d}(1+|j|^2)^{-\beta} \|e_j \nabla f\|_{L^2}^2\\
&\leq \|\theta\|_{\ell^\infty}^2 \|\nabla f\|_{L^2}^2 \sum_{j\in \Z^d} (1+|j|^2)^{-\beta}\lesssim_\beta\|\theta\|_{\ell^\infty}^2 \|\nabla f\|_{L^2}^2,
\end{align*}
where the last inequality follows from $\beta>\frac{d}{2}$. Combining the above with \eqref{eq:estimate_stoch_convol_linfty20}, by Fubini's theorem (or Young's inequality for convolutions), we obtain
\begin{align*}
\int_a^b \E \|\mathcal{Z} (t)\|_{H^{-\beta}}^2\,\dd t 
&\lesssim_{d,\beta} \nu \|\theta\|_{\ell^\infty}^2 \E \int_a^b\int_a^t e^{-4\pi^2(\nu+\mu) (t-s)} \|\nabla f(s)\|_{L^2}^2\,\dd s\,\dd t \\&\lesssim_{d,\beta} \|\theta\|_{\ell^\infty}^2 \E \int_a^b \|\nabla f(s)\|_{L^2}^2\,\dd s ,
\end{align*}
where we used $\int_{0}^\infty e^{-4\pi^2 (\nu+\mu)t}\,\dd t \lesssim 1/\nu$.
\end{proof}

Finally, we are in a position to prove Proposition \ref{prop:iteration_lemma}.

\begin{proof}[Proof of Proposition \ref{prop:iteration_lemma}]
We begin by noticing that, by Proposition \ref{prop:quenched_estimate_w}, for $t\in [n,n+1]$, 
$$
\E\|\wt{\vmod}(n+1)\|_{L^2}^2 
\leq \E\sup_{s\in [t,n+1]}\|\wt{\vmod}(s)\|_{L^2}^2 
\leq C_0 \E\|\wt{\vmod}(t)\|_{L^2}^2 .
$$
Integrating over $t\in [n,n+1]$ entails 
\begin{equation}
\label{eq:C0_plays_an_improtant_role_independence_theta}
\E\|\wt{\vmod}(n+1)\|_{L^2}^2 \leq C_0\max_{1\leq i\leq \ell}\E\int_n^{n+1}\|\wt{\vmod}_i(t)\|_{L^2}^2\,\dd t.
\end{equation}
Fix $\g\in (0,\frac{2}{d+2})$.
Hence, to conclude the proof of Proposition \ref{prop:iteration_lemma}, it suffices to show
\begin{equation}
\label{eq:energy_estimate_wt_i_claim}
\max_{1\leq i\leq \ell} \,
\E\int_n^{n+1}\|\wt{\vmod}_i(t)\|_{L^2}^2\,\dd t\lesssim (\nu^{-1}+\|\theta\|_{\ell^\infty}^\g) \,\E \|\wt{\vmod}(n)\|_{L^2}^2.
\end{equation}
To prove \eqref{eq:energy_estimate_wt_i_claim}, note that, from the mild formulation \eqref{eq:mild_formulation_SPDE_wtv}, we have  
\begin{equation}
\label{eq:estimate_I_tilde_v}
\Big(\E\int_n^{n+1}\|\wt{\vmod}_i(t)\|_{L^2}^2\,\dd t\Big)^{1/2}\leq I_1 + I_2+ I_3,
\end{equation}
where 
\begin{align*}
I_1 &=\Big(\E \int_n^{n+1}\|e^{(\nu+\nu_i)(s-n)\Delta} \wt{\vmod}_i(n)\|_{L^2}^2\,\dd s\Big)^{1/2},\\
I_2 &= \Big(\E \int_n^{n+1}\Big\|\int_{n}^t e^{(\nu+\nu_i)(t-s)\Delta} \big[\phi_{R,r}(v_\infty,\vmod)\big(f_i(\vmod(s))-\overline{f_i(\vmod(s))}\, \big)\big]\,\dd s\Big\|_{L^2}^2\,\dd t\Big)^{1/2},\\
I_3 &= \Big(\E\int_n^{n+1} \| \mathcal{Z}^{f_i}_{\nu_i,\nu}(t)\|_{L^2}^2\,\dd t\Big)^{1/2} \ \ \  \text{ with }\ \ \ f_i = \one_{[n,n+1]} \wt{\vmod}_i.
\end{align*}
In the above, we used the notation introduced in Lemma \ref{l:estimate_stoch_convol_linfty} for the stochastic convolution.

Next, we estimate each term separately. First, as $\wt{\vmod}_i(n)$ has zero mean, we have 
$$
I_1^2\leq   \int_n^{n+1} e^{- 4\pi^2 (t-n)(\nu_i+\nu)}\E\|\wt{\vmod}_i(n)\|_{L^2}^2\,\dd t 
\lesssim \frac{1}{\nu} \,\E \|\wt{\vmod}_i(n)\|_{L^2}^2 .
$$
Second, by Lemma \ref{l:smr_L2_L1H} and $\phi_{R,r}(v_\infty,\vmod)\leq 1$, 
$$
I_2^2 \lesssim \frac{1}{\nu}\,\E \Big[\Big( \int_n^{n+1} \|f_i(\vmod(s))-\overline{f_i(\vmod(s))}\|_{H^{-1}}\,\dd s \Big)^{2}\Big].
$$
We estimate $ \|f_i(\vmod(s))-\overline{f_i(\vmod(s))}\|_{H^{-1}}$ similarly to the proof of Proposition \ref{prop:quenched_estimate_w}. 
Set $r_0=d$ and $r=\frac{2d}{d+2}$ if $d\geq3$. If $d=2$, choose $r_0>2$ such that \eqref{eq:interpolation_inequality_checking} holds and let $r\in(1,2)$ satisfy $\frac{1}{2}+\frac{1}{r_0}=\frac{1}{r}$.
With the previous choice, by Sobolev embedding $L^r\embed H^{-1}$, a.s., it holds that  
\begin{align}
\label{eq:estimate_f_overlinef_enhanced_dissip}
\int_n^{n+1}\|f_i(\vmod)-\overline{f_i(\vmod)}\|_{H^{-1}}\, \dd t' 
&\leq 
\int_n^{n+1}\|\nabla[f_i(\vmod)]\|_{H^{-1}}\,\dd t' \\
\nonumber
&\leq 
\int_n^{n+1}\|(1+|\vmod|^{h-1})|\nabla \wt{\vmod}|\|_{L^r}\,\dd t' \\
\nonumber
&\lesssim_K 
 \int_n^{n+1}\big[  (1+\|\vmod-v_\infty\|_{L^{r_0(h-1)}}^{h-1})\|\nabla \wt{\vmod}\|_{L^2}\big]\,\dd t' \\
\nonumber
&\lesssim
\big( 1+\|\vmod-v_\infty\|_{L^{2(h-1)}(n,n+1;L^{r_0(h-1)})}^{h-1} \big)
\|\nabla \wt{\vmod}\|_{L^2(n,n+1;L^2)}\\
\nonumber
&\stackrel{(i)}{\lesssim}_{R,N}
\|\nabla \wt{\vmod}\|_{L^2(n,n+1;L^2)} \stackrel{(ii)}{\lesssim} \|\wt{\vmod}(n)\|_{L^2},
\end{align}
where $(i)$ follows as in \eqref{eq:boundedness_Lr0_proof} by Proposition \ref{prop:global_cut_off_entropy} and Lemma \ref{l:interpolation_Lr0}, while $(ii)$ is a result of Proposition \ref{prop:quenched_estimate_w}. Therefore,
$$
I_2^2 \lesssim \nu^{-1 }\, \E \|\wt{\vmod}(n)\|_{L^2}^2.
$$
Finally, from Lemma \ref{l:estimate_stoch_convol_linfty} and Proposition \ref{prop:quenched_estimate_w}, we infer 
$$
I_3^2 \lesssim_\g \|\theta\|_{\ell^\infty}^\g \,\E \int_n^{n+1} \|\nabla \wt{\vmod}(s)\|_{L^2}^2\,\dd s  
\lesssim \|\theta\|_{\ell^\infty}^\g \,\E \|\wt{\vmod}(n)\|_{L^2}^2.
$$
Thus, \eqref{eq:energy_estimate_wt_i_claim} now follows by collecting the previous inequalities and taking $\max_{1\leq i\leq \ell}$.
\end{proof}

We conclude with some comments on the use of the $L^1$-integrability in Lemma \ref{l:smr_L2_L1H} and on the use of Proposition \ref{prop:quenched_estimate_w} in \eqref{eq:estimate_f_overlinef_enhanced_dissip}.

\begin{remark}\
\label{r:L1_estimate_usefulness}
\begin{itemize}
    \item 
Without additional assumptions, it does not seem to be possible to control the deterministic convolution $I_2$ defined below \eqref{eq:estimate_I_tilde_v} by using the analogous version of Lemma \ref{l:smr_L2_L1H} with time integrability $2$. Indeed, in this situation, by partially repeating the argument in \eqref{eq:estimate_f_overlinef_enhanced_dissip}, one needs a uniform in $\theta$ bound on $\|\vmod-v_\infty\|_{L^\infty(n,n+1; L^{d(h-1)})}$. From Proposition \ref{prop:global_cut_off_entropy}, this is only possible under the additional assumption $q>d(h-1)$.
\item It is worth pointing out that the $\theta$-independence of the estimates in Proposition \ref{prop:global_cut_off_entropy} is essentially used in  \eqref{eq:C0_plays_an_improtant_role_independence_theta} via Proposition \ref{prop:quenched_estimate_w}. In contrast, the application of these estimates in \eqref{eq:estimate_f_overlinef_enhanced_dissip} can be avoided by exploiting the presence of $\nu^{-1}$ to control explosion in $\|\theta\|_{\ell^\infty}$; see \cite[Theorem 2.3 and Lemma 4.1]{L21} for a related situation. 
\end{itemize}

\end{remark}

\appendix

\section{Proof of Lemma \ref{lem:Lebesgue_initial_data_continuity}}
\label{app:continuity_Lq}
In this appendix, we prove Lemma \ref{lem:Lebesgue_initial_data_continuity}. Before going into the proof, we first discuss the main strategy. Firstly, from the instantaneous regularization of solutions to \eqref{eq:reaction_diffusion} given in \eqref{eq:LWP2}, it suffices to discuss the regularity at time $t=0$. The latter case follows from an approximation argument together with the uniform convexity of $L^q$-spaces (see \cite[Proposition 5.3]{AV24_dissipative} for a similar situation). Indeed, at least for smooth initial data, applying It\^o's formula to $v_i\mapsto \|v_i\|_{L^q(\T^d)}^q$ for a fixed $i\in \{1,\dots,\ell\}$ yields (see Lemma \ref{l:Lq_estimate_small} for a similar situation)
\begin{align*}
\|v_i(t)\|_{L^q}^q +q(q-1)\nu_i \int_0^t \int_{\T^d} |v_i|^{q-2} |\nabla v_i|^2 \,\dd x \,\dd s 
= \|v_{i,0}\|_{L^q}^q+ q\int_0^t \int_{\T^d} |v_i|^{q-2} v_i f_i (v)  \,\dd x \,\dd s, 
\end{align*}
where we used $\nabla \cdot \sigma_{k,\alpha}=0$ for all $k,\alpha$. Recall from Subsection \ref{subsec:CRN} that $|f_i(v)|\lesssim 1+|v|^h$, then, by \eqref{eq:LWP1} to obtain $\lim_{t\downarrow 0}v_i=v_{i,0}$ in $L^q$, it suffices to show $|v|\in L^{q+h-1}_{\loc}([0,\tau)\times \O;\R^\ell)$. 
This motivates Step 1 in the proof below.

\begin{proof}[Proof of Lemma \ref{lem:Lebesgue_initial_data_continuity}]
For brevity, we focus on the critical case $q=\frac{d(h-1)}{2}\vee 2$; the subcritical case is simpler. 
For convenience, we set $h_0=1+\frac{2q}{d}\in [h,\infty)$. In particular, $q=\frac{d(h_0-1)}{2}$. For convenience, we set $q_0\coloneq q$. Finally, by compatibility of solutions to stochastic reaction--diffusion equations \eqref{eq:reaction_diffusion} (see \cite[Proposition 3.5]{AV23_RDE_local}), we may choose $\delta_0>1$ and $p_0$ such that
\begin{equation}
\label{eq:choice_p_proof_lebesgue_initial_data}
\frac{2}{2-\delta_0}<q_0 \qquad \text{ and }\qquad 
p_0=q+h_0-1.
\end{equation}
In particular, $p_0=(h_0-1)\frac{d+2}{2}>q$ and the conditions in \eqref{eq:pq_condition_local_WP} are satisfied with $(p,\delta,q,h)$ replaced by $(p_0,\delta_0,q_0,h_0)$. Finally, we set $\a_0\coloneq p_0(1-\frac{\delta_0}{2})-1$.

\smallskip

\emph{Step 1: For all $t>0$, it holds that}
\begin{equation}
\label{eq:embedding_Hp0}
H^{\kappa_0/p_0,p_0}(0,t,w_{\a_0};H^{2-\delta_0-2\a_0/p_0,q_0}(\T^d))\embed L^{q_0+h_0-1}(0,t;L^{q_0+h_0-1}(\T^d)).
\end{equation}
\emph{In particular, $v\in L^{q_0+h_0-1}_{\loc}([0,\tau)\times \T^d;\R^\ell)$ a.s.}
Clearly, the second assertion of Step 1 is a consequence of \eqref{eq:LWP1} with $\vartheta=\frac{\a_0}{p_0}<\frac{1}{2}$ as $\a_0\in [0,\frac{p_0}{2}-1)$. To prove \eqref{eq:embedding_Hp0}, note that  
$$
H^{\kappa_0/p_0,p_0}(0,t,w_{\a_0};H^{2-\delta_0-2\kappa_0/p_0,q_0})
\embed
L^{p_0}(0,t;H^{2-\delta_0-2\kappa_0/p_0,q_0})
$$
as it follows from vector-valued weighted Sobolev embeddings (see, e.g., \cite[Proposition 2.7]{AV19_QSEE_1}). It remains to show that 
$$
H^{2-\delta_0-2\kappa_0/p_0,q_0}(\T^d)\embed L^{q_0+h_0-1}(\T^d).
$$
The above follows by Sobolev embeddings and the fact that 
$$
2-\delta_0-2\frac{\a_0}{p_0}=\frac{2}{p_0} \qquad \text{ and }\qquad \frac{2}{p_0}-\frac{d}{q_0}\stackrel{\eqref{eq:choice_p_proof_lebesgue_initial_data}}{=}-\frac{d}{q_0+h_0-1},
$$
where we used $q_0=\frac{d(h_0-1)}{2}$.

\smallskip

\emph{Step 2: Conclusion.} 
By uniqueness of $(p_0,\a_0,\delta_0,q_0)$-solutions, it suffices to assume $v_0\in L^{p_0}_{\F_0}(\O;L^{q_0})$ (e.g., replace $v_0 $ by $\one_{\{\|v_0\|_{L^{q_0}}\leq N\}} v_0$ for some $ N\geq 1$).
Let $(v_{0}^n)\subseteq L^{p_0}_{\F_0}(\O;C^1(\T^d;\R^\ell))$ be such that $v_{0}^n \to v_0$ in $L^{p_0}(\O;L^{q_0})$ as $n\to \infty$.
From \cite[Theorem 2.7]{AV23_RDE_local} applied with $\delta=1$ and $(q,p)=(q_1,p_1)$, where $q_1,p_1$ are sufficiently large, for each $n\geq 1$, there exists a $(p_1,0,1,q_1)$-solution $(v^n,\tau^n)$ to \eqref{eq:reaction_diffusion} (see Definition \ref{def:solution}). In particular, 
$$
v^n \in C([0,\tau^n);L^\infty(\T^d;\R^\ell)) \cap L^2([0,\tau^n);H^1(\T^d;\R^\ell)) \ \text{ a.s. }
$$ 
Fix $r\in [2,q_0]$.
From the above and the It\^o formula in either \cite[Section 3]{Kry13} or \cite[Appendix A]{DHV16} (see also \cite[Theorem 4.1(2)]{A22} or \cite[Lemma 2]{DG15_boundedness} for details), it holds a.s.\ for all $n\geq 1$ and $t\in [0,\tau^n)$,
\begin{align}
\label{eq:Itoformula_Lq_continuity_proof}
\|v_i^n(t)\|_{L^r}^r
&+r(r-1) \nu_i \int_{0}^t \int_{\T^d} |v_i^n|^{r-2} |\nabla v_i^n|^2\,\dd x \, \dd t' \\
\nonumber
&=\|v_{0,i}^n\|_{L^r}^r
+ r\int_{0}^t  \int_{\T^d} |v_i^n|^{r-2} v_i^n f_i (v^n)\, \dd x \,\dd t'\\
\nonumber
&\leq \|v_{0,i}^n\|_{L^r}^r
+ r\int_{0}^t  \int_{\T^d} (1+|v^n|^{q_0+h_0-1})\, \dd x \,\dd t'.
\end{align}
From \cite[Proposition 3.3 and Remark 3.4(b)]{AV23_RDE_local} and Step 1, there exist constants $n_0\geq 1$, $C_0>0$ and stopping times $\tau_0,\tau_1\in(0,\tau]$, $(\tau_0^n)_n$ such that $\tau_0^n \in (0,\tau^n]$ and, for all $n\geq n_0$, $t>0$, $\g>0$, 
\begin{align}
\label{eq:local_continuity_1}
\P \Big(\sup_{s\in [0,t]}\|v(s)-v^n(s)\|_{B^0_{q_0,p_0}}\geq \g , \ \tau_0\wedge \tau_0^n >t \Big)
&\leq \frac{C_0}{\g^{p_0}}\E\|v_0-v_0^{n}\|_{L^{q_0}}^{p_0},\\
\label{eq:local_continuity_2}
\P \Big(\|v-v^n\|_{L^{q_0+h_0-1}(0,t;L^{q_0+h_0-1})}\geq \g , \ \tau_0\wedge \tau_0^n >t \Big)
&\leq \frac{C_0}{\g^{p_0}}\E\|v_0-v_0^{n}\|_{L^{q_0}}^{p_0},\\
\label{eq:local_continuity_3}
\P(\tau_0\wedge \tau_0^n\leq t)
&\leq C_0\big[ \E\|v_0-v_0^{n}\|_{L^{q_0}}^{p_0}+ \P(\tau_1\leq t )\big].
\end{align}
Next, we would like to take $n\to \infty$ in \eqref{eq:Itoformula_Lq_continuity_proof}. However, the latter identity holds on an $n$-dependent stochastic interval. To this end, by taking a subsequence, we can assume $\sum_{n\geq 1}2^{np_0}\E\|v_0-v_0^{n}\|_{L^{q_0}}^{p_0}<\infty$.
First, note that, as $\tau_1>0$ a.s., 
\begin{equation*}
\P\big(\liminf_{n\to \infty} (\tau_0\wedge \tau_0^n)= 0\big)
\leq C_0 \lim_{k\to \infty}\sum_{n\geq k}  \E\|v_0-v_0^{n}\|_{L^{q_0}}^{p_0}=0.
\end{equation*} 
Hence, $\tau_* \coleq \liminf_{n\to \infty} (\tau_0\wedge \tau_0^n)>0$ a.s.\ 
Without loss of generality, we assume $\tau_*\leq 1$. 
We claim that 
\begin{align}
\label{eq:claim_convergence_1_continuity_proof}
\P \big( \limsup_{n\to \infty}\sup_{s\in [0,\tau_*/2]}\|v^n(s)-v(s)\|_{B^0_{q_0,p_0}}>0 \big) &=0, \\ 
\label{eq:claim_convergence_2_continuity_proof}
\P \big( \limsup_{n\to \infty}\|v^n-v\|_{L^{q_0+h_0-1}(0,\tau_*/2;L^{q_0+h_0-1})}>0 \big) &=0.
\end{align}
We first show how the above yields the claim of Lemma \ref{lem:Lebesgue_initial_data_continuity}. 
Letting $n\to \infty$ in \eqref{eq:Itoformula_Lq_continuity_proof}, from Step 1 and \eqref{eq:LWP1} with $\vartheta=\kappa_0/p_0<1/2$, we obtain that, for $t\in [0 ,\tau_*/2]$ and a.s., 
\begin{align*}
\|v_i(t)\|_{L^r}^r
+r(r-1) \nu_i \int_{0}^t \int_{\T^d} |v_i|^{r-2} |\nabla v_i|^2\,\dd x \, \dd t' 
\leq \|v_{0,i}\|_{L^r}^r
+ r\int_{0}^t  \int_{\T^d} (1+|v|^{q_0+h_0-1})\, \dd x \,\dd t',
\end{align*}
In the above, we also used the lower semicontinuity of the $L^2$-norm as well as $\big|\nabla [|v_i|^{r/2}]\big|^2= \frac{r^2}{4}|v_i|^{r-2}|\nabla v_i|^2$. In particular,
$$
\lim_{t\downarrow 0}\|v_i(t)\|_{L^r}^r
\leq \|v_{0,i}\|_{L^r}^r \ \text{ a.s.\ }
$$
The latter together with \eqref{eq:LWP1}, Fatou's property, and uniform convexity of $L^r$ leads to $\lim_{t\downarrow 0}v_i=v_{0,i}$ in $L^r$. The claim of Lemma \ref{lem:Lebesgue_initial_data_continuity} now follows from the arbitrariness of $r\in [2,q_0]$. 

\smallskip

Finally, we turn to the proof of \eqref{eq:claim_convergence_1_continuity_proof}--\eqref{eq:claim_convergence_2_continuity_proof}. We only prove \eqref{eq:claim_convergence_2_continuity_proof} as \eqref{eq:claim_convergence_1_continuity_proof} is analogous. For each $k \geq 1$, let $\O_k = \{\tau_* > 1/k\}$. As $\tau_* > 0$ a.s., we have $\P(\cup_{k\geq 1} \O_k) = 1$. 

Fix $k \geq 1$. From \eqref{eq:local_continuity_2} with $t=1/k$ and $\g = 2^{-n}$, we obtain
$$
\sum_{n \geq 1} \P \Big(\|v-v^n\|_{L^{q_0+h_0-1}(0,1/k;L^{q_0+h_0-1})}\geq 2^{-n} , \ \tau_0\wedge \tau_0^n > 1/k \Big) \leq C_0 \sum_{n \geq 1} 2^{np_0} \E\|v_0-v_0^{n}\|_{L^{q_0}}^{p_0} < \infty.
$$
Borel--Cantelli's lemma ensures that for a.a.\ $\om \in \O$, there exists an integer $N_{1,k}(\om)$ such that, for all $n \geq N_{1,k}(\om)$, either
$$
\|v-v^n\|_{L^{q_0+h_0-1}(0,1/k;L^{q_0+h_0-1})} < 2^{-n} \quad \text{or} \quad \tau_0\wedge \tau_0^n \leq 1/k.
$$
The definition of $\O_k$ and the fact that $\tau_*>0$ a.s.\ ensure that, for a.a.\ $\om\in\O$, there exists $N_{2,k}(\om)\geq1$ such that $\tau_0(\om)\wedge \tau_0^n(\om) > 1/k$ for all $n\geq N_{2,k}(\om)$. Thus, for a.a.\ $\om \in \O_k$, we have 
$$
\limsup_{n\to \infty}\|v-v^n\|_{L^{q_0+h_0-1}(0,1/k;L^{q_0+h_0-1})} = 0.
$$ 
To conclude, note that for a.a. $\om \in \O$, we can choose $k$ large enough such that $\tau_*(\om)/2 \leq 1/k < \tau_*(\om)$ (here we used that $\tau_*\leq 1$). For the latter $k$ and $\om \in \O_k$, one has $[0,\tau_*(\om)/2]\subseteq[0,1/k]$ and, thus, the above yields \eqref{eq:claim_convergence_2_continuity_proof}.
\end{proof}

\subsection*{Acknowledgements}
Part of this work was completed during the visits of the second author to TU Delft and of the first author to the University of Graz. The authors gratefully acknowledge the hospitality of both institutions.

\def\polhk#1{\setbox0=\hbox{#1}{\ooalign{\hidewidth
  \lower1.5ex\hbox{`}\hidewidth\crcr\unhbox0}}} \def\cprime{$'$}

\end{document}